\documentclass[]{amsart}
\usepackage[utf8]{inputenc}
\usepackage{amsmath,amssymb,amsthm}
\usepackage{epsfig}
\usepackage{graphicx}
\usepackage{multirow}
\usepackage[width=1.0\textwidth, font=footnotesize]{caption}
\usepackage[linesnumbered,ruled,vlined]{algorithm2e}
\SetKwInput{KwInput}{Input}                
\SetKwInput{KwOutput}{Output}              
\usepackage{hyperref}
 \hypersetup{
 	colorlinks=true,
 	linkcolor=blue,
        citecolor=red,
 	filecolor=red,      
 	urlcolor=cyan,
 }
\newtheorem{theorem}{Theorem}
\newtheorem{assumption}{Assumption}
\newtheorem{remark}{Remark}
\newtheorem{notation}{Notation}
\newtheorem{lemma}{Lemma}
\newtheorem{corollary}{Corollary}

\newcommand{\norm}[1]{\left\| #1 \right\|}
\newcommand{\abs}[1]{\left| #1 \right|}
\newcommand{\ceil}[1]{\left \lceil #1 \right \rceil}
\newcommand{\prt}[1]{\left( #1 \right)}
\newcommand{\cost}[1]{\mathrm{Cost}}
\newcommand{\Frob}[1]{\left \langle #1 \right \rangle_{\mathtt{F}}}

\newcommand{\bF}{\mathbb{F}}
\newcommand{\bN}{\mathbb{N}}
\newcommand{\bP}{\mathbb{P}}
\newcommand{\bR}{\mathbb{R}}
\newcommand{\bZ}{\mathbb{Z}}
\newcommand{\bT}{\mathbf{T}}

\newcommand{\ff}{\mathbf{f}}
\newcommand{\fc}{\mathbf{c}}

\newcommand{\cB}{\mathcal{B}}
\newcommand{\cF}{\mathcal{F}}
\newcommand{\cO}{\mathcal{O}}
\newcommand{\cI}{\mathcal{I}}
\newcommand{\cJ}{\mathcal{J}}
\newcommand{\cK}{\mathcal{K}}

\newcommand{\cN}{\mathcal{N}}

\newcommand{\dimObs}{d_O}
\newcommand{\ML}{\mathrm{ML}}

\newcommand{\vBar}{\bar v}
\newcommand{\vHat}{\hat{v}}
\newcommand{\vBarHat}{\hat{\bar{v}}}

\newcommand{\yTilde}[1]{\tilde{y}_{#1}}

\newcommand{\cBarTilde}{\tilde{\bar{C}}}

\newcommand{\Prob}[1]{\bP_{#1}}
\newcommand{\Ex}[1]{\E \left[ #1 \right]}
\newcommand{\Cov}{\overline{\mathrm{Cov}}}
\newcommand{\E}{\mathbb{E}}

\title[MLEnKF based on independent samples ]{Multilevel Ensemble Kalman Filtering based on a sample average of independent EnKF estimators}

\author[H. Hoel]{H{\aa}kon Hoel} \address[{H{\aa}kon Hoel}]{\newline
  Chair of Mathematics for Uncertainty Quantification, RWTH Aachen
  University, Aachen, Germany \newline (hoel@uq.rwth-aachen.de)}

\author[G. Shaimerdenova]{Gaukhar
  Shaimerdenova$^*$}\thanks{$^*$Corresponding author: G.Shaimerdenova
  (gaukhar.shaimerdenova@kaust.edu.sa)} \address[{Gaukhar
    Shaimerdenova}]{\newline Applied Mathematics and Computational
  Sciences, KAUST, Thuwal, Saudi Arabia \newline
  (gaukhar.shaimerdenova@kaust.edu.sa)}

\author[R. Tempone]{Ra\'ul Tempone} \address[{Raul Tempone}]{\newline
  Chair of Mathematics for Uncertainty Quantification, RWTH Aachen
  University, Aachen, Germany \newline (tempone@uq.rwth-aachen.de)
  \newline
  \and
  \newline
  Applied Mathematics and Computational Sciences, KAUST, Thuwal, Saudi Arabia \newline (raul.tempone@kaust.edu.sa)}

\begin{document}

\begin{abstract}
  We introduce a new multilevel ensemble Kalman filter method (MLEnKF)
  which consists of a hierarchy of independent samples of ensemble
  Kalman filters (EnKF). This new MLEnKF method is fundamentally
  different from the preexisting method introduced by Hoel, Law and
  Tempone in 2016, and it is suitable for extensions towards
  multi-index Monte Carlo based filtering methods.  Robust theoretical
  analysis and supporting numerical examples show that under
  appropriate regularity assumptions, the MLEnKF method has better
  complexity than plain vanilla EnKF in the
  large-ensemble and fine-resolution limits, for weak approximations
  of quantities of interest. The method is developed for
  discrete-time filtering problems with finite-dimensional state
  space and linear observations polluted by additive Gaussian noise.



\end{abstract}

\maketitle

\section{Introduction}
We develop a new multilevel ensemble Kalman filter method (MLEnKF) for
the setting of finite-dimensional state space and discrete-time
partial observations polluted by additive Gaussian noise. Our method
makes use of recent hierarchical variance-reduction
techniques~\cite{heinrich2001,giles2008,abdo2016} to improve the
asymptotic efficiency of weak approximations of filtering
distributions compared to standard ensemble Kalman
filtering (EnKF). We consider settings where solutions of
nonlinear dynamics models must be approximated by numerical methods. 

The herein introduced MLEnKF method consists of a hierarchy of
independent samples of pairwise-coupled EnKF estimators where,
in particular, the Kalman gains of every EnKF sample thus also are
independent. Our method is fundamentally
different from the ``canonical" MLEnKF~\cite{hoel2016}, which consists
of a hierarchy of coupled ensembles on different resolution levels,
all sharing one global ``multilevel'' Kalman gain.

The motivations for developing the new MLEnKF method are
threefold. First, the method is closer to classic EnKF, and we
therefore believe it will be easier to implement for
practitioners. Second, imposing slightly stricter regularity
assumptions, we can prove better asymptotic efficiency results for
this method than for the ``canonical'' one. And third, creating a
rigorous convergence theory for the new MLEnKF method is a stepping
stone towards a multi-index Ensemble Kalman filtering (MIEnKF) method;
See~\cite{abdo2018} for highly efficient approximations of
McKean--Vlasov dynamics by the multi-index Monte Carlo method, and
Appendix~\ref{app:extension} for a sketch of the said extension to
MIEnKF.

The main theoretical contributions of this work are
Theorems~\ref{thm:enkfConv} and~\ref{thm:mlenkfConv}, which respectively
derive
$L^p$-convergence rates for weak approximations in the large-ensemble
and fine-numerical-resolution limits with EnKF. Theorem~\ref{thm:mlenkfConv} is
novel, and while Theorem~\ref{thm:enkfConv} is
similar to~\cite[Theorem 3.11]{hoel2016}; but, to the best of our
knowledge, this is the first fully proved convergence result for EnKF in the said
limits (i.e., in both limits simultaneously). Estimates for EnKF's and
MLEnKF's asymptotic computational cost versus accuracy for the
respective methods are provided in Corollaries~\ref{cor:enkf}
and~\ref{cor:mlenkf}, respectively. From these estimates we conclude
that MLEnKF asymptotically outperforms EnKF whenever
Assumptions~\ref{ass:psi} and~\ref{ass:psi2} hold.

\subsection{Literature review}
The EnKF method was first introduced in the seminal
work~\cite{evensen1994sequential}, and due to its ease of use and
impressive performance in high dimensions, it quickly became a popular
method for weather prediction, ocean-atmosphere science, and oil
reservoir
simulations~\cite{kalnay2003atmospheric,houtekamer2005atmospheric,aanonsen2009ensemble}.
The standard version of EnKF with perturbed observations, which is the
method we will study and extend to the multilevel setting in this
work, first appeared in~\cite{houtekamer1998data} and ensuing
analysis~\cite{burgers1998analysis} showed that adding artificial
noise to the observations may be viewed as a consistency step for
avoiding covariance deflation of the empirical filtering
distribution. $L^p$-convergence of the first and second sample moments in
the large-ensemble limit for EnKF was first treated by a short
argument in~\cite{mandel2011convergence} for the linear model setting
and the result was subsequently extended to a set of nonlinear
filtering problems by a more technical argument in~\cite{le2011large}
(in the sense of deriving $L^p$-convergence rates for weak
approximations of sufficiently smooth quantities of interest).

The multilevel Monte Carlo (MLMC) method was introduced for efficient
weak approximations of random fields in~\cite{heinrich2001} and for
stochastic differential equations in~\cite{giles2008}.  The MLEnKF
method was developed for finite-dimensional settings
in~\cite{hoel2016} and countable, infinite-dimensional settings
in~\cite{chernov2017multilevel}.  In~\cite{fossum2019assessment} a
multilevel hybrid EnKF method was developed for solving reservoir
history matching problems. Multilevel particle filters using a
multilevel-coupled resampling algorithm was introduced
in~\cite{jasra2017multilevel}. The multilevel transform particle
filter, applying optimal transportation mapping in the analysis step,
was treated in~\cite{gregory2016multilevel, gregory2017seamless}. In
the context of Bayesian inverse problems, multilevel sequential Monte
Carlo methods have been studied
in~\cite{beskos2017multilevel,beskos2018multilevel,moral2017multilevel,latz2018multilevel}
and a multilevel Markov Chain Monte Carlo method was developed
in~\cite{dodwell2015}. The multi-fidelity Monte Carlo
method~\cite{peherstorfer2016optimal} is a recent and close kin of
MLMC that differs from MLMC by having fixied its estimator's finest
resolution level and by having more extensive coupling of samples than
only pairwise. This often reduces the estimator's statistical error
even more effectively than MLMC. A multi-fidelity Monte Carlo EnKF
method was developed in~\cite{popov2020multifidelity}.

Under sufficient regularity, EnKF converges to the so-called
mean-field Kalman filter in the large-ensemble limit. In
nonlinear-dynamics or observation settings, however, the mean-field
Kalman filter is not equal to the Bayes
filter~\cite{law2016deterministic,le2011large,ernst2015analysis,rosic2013parameter}.
Due to this discrepancy between EnKF and the Bayes filter even in the
large-ensemble limit, due to the large uncertainty in the estimation
of the model error and due to the constraints imposed by challening
high-dimensional problems and limited computational budgets, a
considerable number of works have, instead of studying the
large-ensemble limit, focused on the large-time and/or continuous-time
limit of the fixed-ensemble-size EnKF
cf.~\cite{kelly2014well,tong2016nonlinear,schillings2017analysis,schillings2018convergence,blomker2019well,de2018long,lange2019continuous}.
However, a recurring problem for fixed-ensemble-size EnKF is to
determine how large the ensemble ought to be to equilibrate the model
error with the other error contributions (statistical error and
bias). And the large-ensemble limit convergence rates for EnKF and
MLEnKF that are presented in this work may be helpful for determining
the ensemble size of a finite-ensemble EnKF method.

\subsection{Organization of this work}
Section~\ref{sec:problem} describes the problem, the MLEnKF method,
and $L^p$-convergence results for EnKF and MLEnKF towards the
mean-field EnKF.  Section~\ref{sec:numerics} presents numerical
performance studies of EnKF and MLEnKF for two different problems.
Appendix~\ref{sec:MFEnKF} presents a brief overview of the mean-field
EnKF method. Appendix~\ref{sec:proofs} contains proofs of the main
theoretical results, Appendix~\ref{sec:dmfenkf} describes an
algorithm for obtaining pseudo-reference solutions for nonlinear filtering
problems, and 
Appendix~\ref{app:extension} sketches an extension from
MLEnKF to multi-index EnKF.

\section{Problem setting and main results}
\label{sec:problem}
\subsection{Problem setting}
Let $(\Omega, (\cF_t), \cF=\cF_\infty, \bP)$ denote a complete
filtered probability space, and for any $k \in \bN$ and $p\ge 1$, let
$L^p_t(\Omega,\bR^k)$ denote the space of
$\cF_t \backslash \cB^k-$measurable functions $u:\Omega \to \bR^k$
such that $\Ex{|u|^p}< \infty$. Here, $\cB^k$ represents the Borel
$\sigma$-algebra on $\bR^k$ and $u$ is
said to be $\cF_t \backslash \cB^k-$measurable if and only if 
$u^{-1}(B)\in \cF_t$ for all $B\in \cB^k$. For a given state-space dimension
$d \in \bN$ and initial data $u_0 \in \cap_{p\ge 2} L^p_0(\Omega, \bR^d)$, we
consider the discrete stochastic dynamics for $n=0,1,\ldots$
\begin{equation}\label{eq:hmmDynamics}
u_{n+1}(\omega) = \Psi_n(u_n, \omega), \qquad \omega \in \Omega,
\end{equation}
for a sequence of mappings $\Psi_n: \bR^d \times \Omega \to \bR^d$.
The dynamics is associated with the stochastic differential equation (SDE)
\begin{equation}\label{eq:SDE}
\Psi_n(u_n,\omega) = u_n + \int_{n}^{n+1} a(u_{t}) dt  + \int_{n}^{n+1} b(u_{t}) dW_{t}(\omega), 
\end{equation}
with coefficients $a:\bR^d \to \bR^d$,
$b: \bR^{d}\to \bR^{d \times d_W}$ and the driving noise
$W:[0,\infty)\times \Omega \to \bR^{d_W}$ denoting a $d_W$-dimensional
standard Wiener process. We further assume the coefficients $a$ and $b$ 
are sufficiently smooth so that 
\[
u_n \in L_n^2(\Omega, \bR^d) \implies u_{n+1}  \in L^2_{n+1}(\Omega, \bR^d),
\]
and we note, from now on suppressing the dependence on $\omega$ whenever
confusion is not possible, that $u_{n+1}$
may be expressed as a quasi-iterated mapping of $u_0$:
\[
u_{n+1} = \Psi_n\circ \Psi_{n-1} \circ \cdots  \circ\Psi_0(u_0).
\] 
%
Associated with the dynamics~\eqref{eq:hmmDynamics}, there exists a
series of noisy observations
\begin{equation}\label{eq:observations}
y_{n} = H u_{n} + \eta_n, \quad  n=1,2, \ldots,   
\end{equation}
where $H \in \bR^{\dimObs \times d}$ and $\eta_1, \eta_2, \ldots$ is a
sequence of independent and identically distributed (iid) random
variables satisfying $\eta_1 \sim N(0,\Gamma)$ with positive definite
covariance matrix $\Gamma \in \bR^{\dimObs\times \dimObs}$ and the
independence property $\{\eta_j\}_{j\ge1} \perp \{u_k\}_{k\ge0}$.
The filtration $\cF_t$ is the completion of the smallest
$\sigma$-algebra generated by $u_0$, $\{W_s\}_{s \in [0,t]}$ and
$\{\eta_{j}\}_{j=1}^{\lfloor t \rfloor}$, with $\lfloor t \rfloor = \max\{ k \in \bZ \mid k \le t\}$.
That is,
\[
\cF_t = \overline{ \sigma\Big( \sigma(u_0) \cup \sigma(\{W_s\}_{s \in [0,t]}) \cup \sigma(\{\eta_{j}\}_{j=1}^{\lfloor t \rfloor}) \Big) }.
\]
with $\overline{\sigma}$ denoting the completion of $\sigma$.

The series of observations up to time $k$ is denoted by
\begin{equation}\label{eq:obsSeq}
Y_k := \begin{cases}
  (y_1, y_2, \ldots ,y_k) & \text{if} \quad  k\ge 1,\\
  \emptyset & \text{if} \quad  k =0.
\end{cases}
\end{equation}
For notational convenicence and since no observations have been made
at time $0$, we defined $Y_0 =\emptyset$ above. In consistency with
the propery that $Y_0$ holds no information we write
$\Prob{u_{0} \mid Y_0}:=\Prob{u_{0}}$.
The Bayes filter is a sequential procedure for
determining the (conditional) prediction $\Prob{u_{k} \mid Y_{k-1}}$
and update $\Prob{u_k \mid Y_k}$. Assuming that the densities of said
distributions exist, the stochastic dynamics~\eqref{eq:hmmDynamics}
and the conditional independence $(u_{k}|u_{k-1}) \perp Y_{k-1}$ yield
the proportionality
\begin{equation*}\label{eq:bayesFilter1}
\begin{split}
\rho_{u_k\mid Y_{k-1}}(u) &= \int_{\bR^d} \rho_{u_k, u_{k-1} \mid Y_{k-1}}(u,v) dv\propto \int_{\bR^d} \rho_{u_{k} \mid u_{k-1}}(u) \rho_{u_{k-1} \mid Y_{k-1}}(v) \, dv,
\end{split}
\end{equation*}
and Bayesian inference implies that
\begin{equation*}\label{eq:bayesFilter2}
\begin{split}
  \rho_{u_k \mid Y_k}(u) & \propto \mathcal{L}_{u_k|y_{k}}(u) \rho_{u_k \mid Y_{k-1}}(u),
\end{split}
\end{equation*}
where the likelihood function is given by 
\[
  \mathcal{L}_{u_{k}|y_{k}}(u) 
  = \frac{\exp(-|\Gamma^{-1/2}(y_{k}-Hu )|^2/2)}{  \sqrt{(2\pi)^{p} |\mathrm{det}(\Gamma)}|}.
\]
In other words, if the updated distribution $\Prob{u_{k-1} \mid
  y_{k-1}}$ is known and suitable regularity assumptions hold,
then the prediction and updated distribution at the next
time can be computed up to a proportionality constant (although this
step may be computationally intractable).

In settings where $u_0$ is a Gaussian random variable and the dynamics
$\Psi$ is linear with additive Gaussian noise, the Kalman
filter~\cite{Kalman60} solves the above filtering problem exactly.
When $\Psi$ is nonlinear, however, it is often not possible to solve
the filtering problem exactly and one must resort to approximation
methods.  EnKF is a nonlinear and ensemble-based filtering method that
preserves that tends to be particularly efficient when the ``true''
dimension of a filtering problem is far smaller than the state-space
dimension.

For linear-Gaussian filtering problems, the empirical measure of EnKF
converges towards the exact (Bayes filter) density $\Prob{u_k\mid
  Y_k}$ in the large-ensemble limit~\cite{mandel2011convergence}.
More generally, for instance when $\Psi$ is nonlinear, EnKF converges
to the mean-field EnKF in the large-ensemble limit,
cf.~Section~\ref{sec:MFEnKF}. As a consequence of EnKF employing a
Gaussian-like update of its particles, the mean-field EnKF will in many cases
not be equal to the Bayes
filter~\cite{law2016deterministic,ernst2015analysis}. But, to the best of
our knowledge, there does not exist any thorough scientific comparison
of the mean-field EnKF and the Bayes Filter, and 
Figure~\ref{fig:contrProp} shows that for the nonlinear dynamics
$\Psi$ defined by the SDE
\begin{equation}\label{eq:nonlinearDynamicsFigure}
du=-(u+ \pi\cos(\pi u/5)/5)dt+\sigma dW
\end{equation}
and~\eqref{eq:hmmDynamics}, the dissipative/contractive properties of
the associated Fokker-Planck equation can produce prediction densities
for the respective filtering methods that are indistinguishable to the
naked eye.
\begin{figure}[htbp]
  \includegraphics[width=0.75\textwidth]{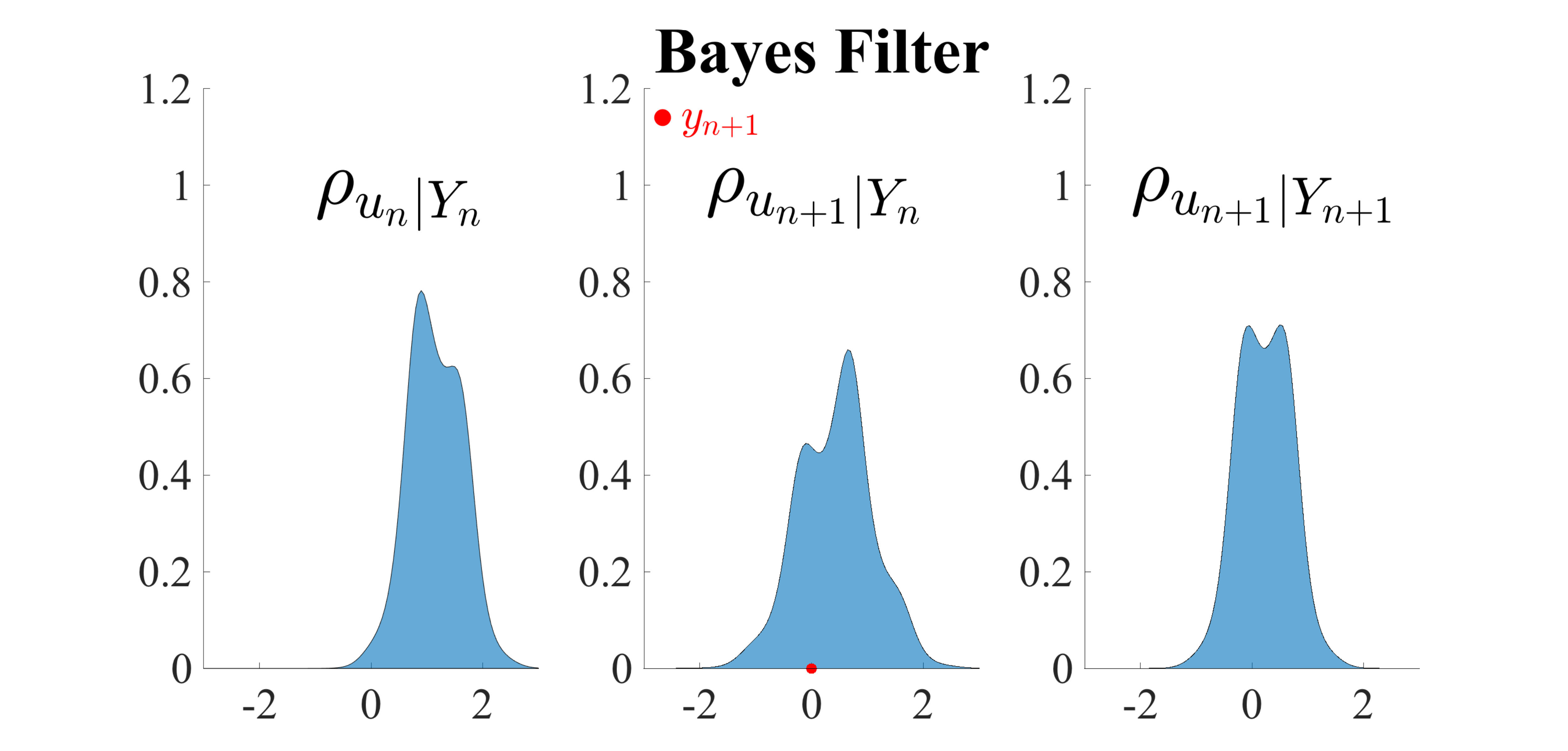}
  \includegraphics[width=0.75\textwidth]{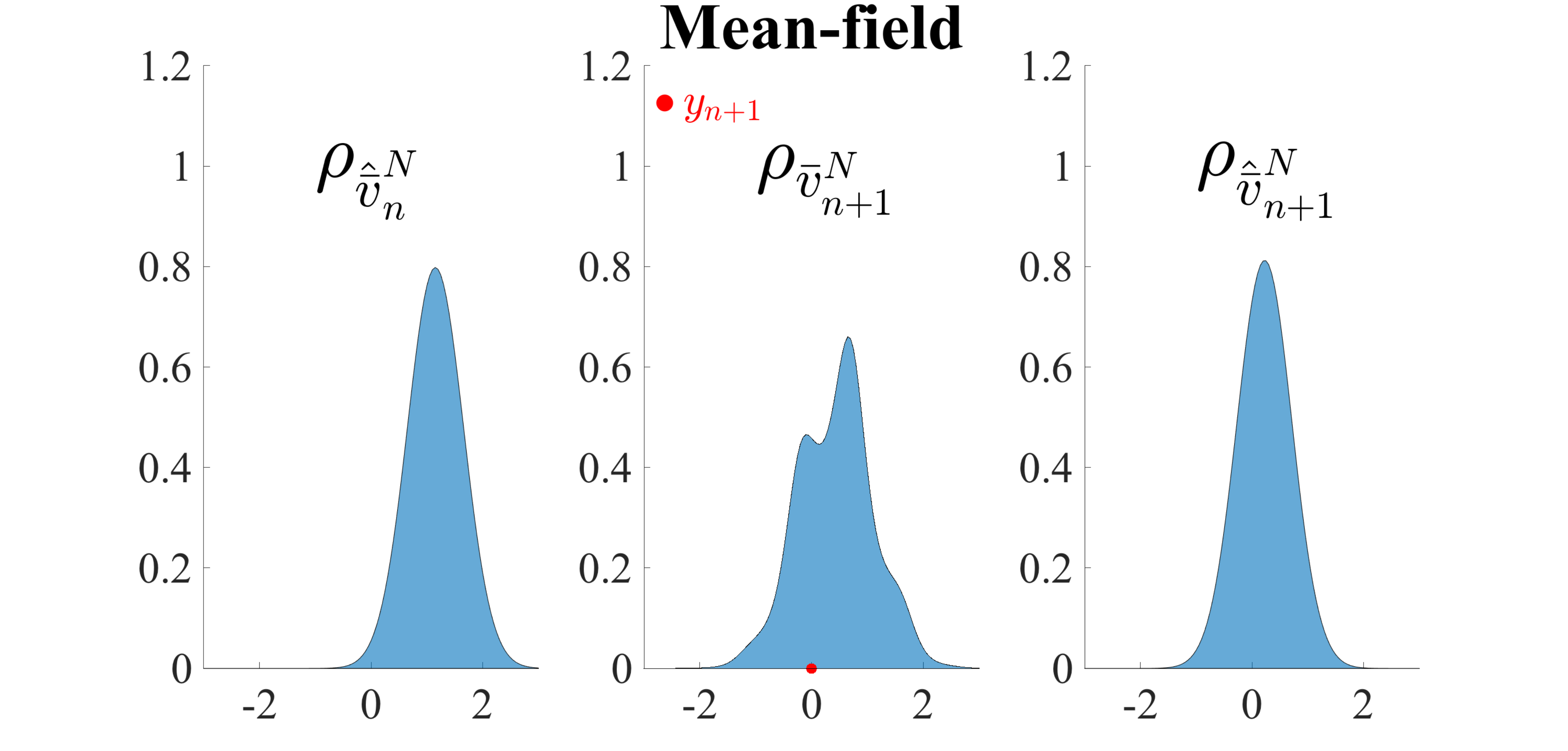}
  \caption{Illustration, based on the nonlinear dynamics~\eqref{eq:nonlinearDynamicsFigure}, of the contracting property which can produce
    almost identical prediction densities (middle panels) for the Bayes filter and
    MFEnKF even when the preceding updated densities differ
    notably.}
  \label{fig:contrProp}
\end{figure}

\subsubsection*{Objective}
For a given quantity of interest (QoI) $\varphi: \bR^d \to \bR$, our objective is
to construct an efficient filtering method for computing
\begin{equation}\label{eq:simplProblem}
\mathbb{E}^{\bar \mu_n}[ \varphi(u) ] = \int_{\bR^d} \varphi(u) \bar \mu_n(du),
\end{equation}
where $\bar \mu_n$ denotes the mean-field EnKF updated measure at time
$n$, cf.~Appendix~\ref{sec:MFEnKF}.  In the best of worlds, one would
rather seek a method computing the exact expectation of $\varphi$
with respect to the Bayes filter posterior, i.e.,
$\Ex{\varphi(u_n)\mid Y_n}$, but since the said posterior $\Prob{u_n
  \mid Y_n}$ is generally not attainable by EnKF-based
filtering methods, and since both EnKF and MLEnKF converge weakly towards mean-field
EnKF, cf.~Theorems~\ref{thm:enkfConv} and~\ref{thm:mlenkfConv}, we
will in this work focus on the simpler (but still very challenging)
goal~\eqref{eq:simplProblem}.

\begin{notation}
~\\
\begin{itemize}

\item For $f,g:(0,\infty) \to [0,\infty)$ the notation $f \lesssim  g$ implies there exists a
  $C>0$ such that
  \[
  f(x) \le  C g(x), \quad  \forall x \in (0,\infty).
  \]

\item The notation $f \eqsim  g$ implies that $f \lesssim g$ and $g\lesssim f$.

  
\item For $r,s \in \bN$,
  $|x|$ denotes the Euclidean norm of a vector $x \in \bR^s$
  and for $A \in \bR^{r\times s}$, $|A|_2:= \sup_{|x|=1} |Ax|$.

\item
For $\cF\backslash  \cB^d$-measurable functions $u:\Omega \to \bR^d$ and $p\ge 1$,
\[
\norm{u}_{p} := \norm{u}_{L^p(\Omega, \bR^d)} = \prt{\int_{\Omega} |u(\omega)|^p \, \mathbb{P}(d\omega)}^{1/p}.
\]

\item For any $\kappa \in \bN_0^r$, with $\bN_0 := \bN \cup \{0\}$,
  and any sufficiently smooth functions
  of the form $f: \bR^r \to \bR$ and $g:\bR^r \times \Omega \to \bR$, the respective
  $\partial^\kappa$-partial derivatives are defined by
\[
\partial^{\kappa} f(x) = \frac{\partial^{\abs{\kappa}_1 }f}{\partial x_1^{\kappa_1} \ldots \partial x_r^{\kappa_d} } (x) \quad 
\text{and} \quad 
\partial^{\kappa} g(x,\omega) = \frac{\partial^{\abs{\kappa}_1 }g}{\partial x_1^{\kappa_1} \ldots \partial x_r^{\kappa_d} } (x,\omega)
\]
for all $x\in \bR^r$ and $\bP-$almost all $\omega \in \Omega$, where
$|\kappa|_1 := \sum_{i=1}^r \kappa_i$. This extends to
vector-valued mappings of the form $f: \bR^r \to \bR^s$ by component-wise
partial derivatives:
$\partial^{\kappa} f(x) = (\partial^{\kappa} f_1(x), \partial^{\kappa}
f_2(x), \ldots, \partial^{\kappa} f_s(x))$, and so on.

\item For $k \in \bN_0$, $C^k_B(\bR^{r}, \bR^s)$
  denotes the space of $k$-times continuously differentiable functions
  $\eta:\bR^r \to \bR^s$ for which $\eta$ and all of its partial
  derivatives of order up to and including $k$ is uniformly bounded.

\item $C^{k}_P(\bR^r, \bR^s)$ denotes the space of
  $k$-times continuously differentiable functions whose
  partial derivatives of order up to and including $k$
  have polynomial growth.

\item For a mapping $\eta: \bR^{r} \to \bR^{s}$,
  its Jacobian is denoted $D\eta$, and its Hessian is
  $D^2\eta$.

\item $\bR_r[x_1,\ldots,x_d]$ denotes the set of polynomials in $d$
  variables that are of total degree smaller or equal to $r \in \bN_0$
  and have coefficients in $\bR$.
  


\end{itemize}

\end{notation}

\subsection{Numerical approximation of the stochastic dynamics}

We assume that for every time $n\ge 0$, there exists a collection of
progressively more accurate numerical solvers
$\{\Psi^{N}_n: \bR^d\times \Omega \to \bR^d\}_{N \in \bN}$ satisfying
the following assumptions:

\begin{assumption}\label{ass:psi}
~\\
The initial distribution $u_0 \in \cap_{p \ge 2} L^p_0(\Omega, \bR^d)$ and for any $n \in \bN_0$,
\begin{itemize} 

\item[(i)] for any $p\ge 2$, there exists a $c_p >0$ such that
  
  \[
    \norm{\Psi^{N}_n(u)}_p \leq c_p (1+\norm{u}_p) \quad \; \forall \big (N \in \bN  \quad \& \quad u \in \cap_{p \ge 2} L^p_n(\Omega, \bR^d) \, \big)\, ;
  \]

\item[(ii)] there exists an $\alpha>0$ and
  a set of mappings $\bF$ with $\bR_2[x_1,\ldots,x_d] \subset \bF \subset C^2_P(\bR^d,\bR)$
  such that if 
  \[
    \begin{split}
      \abs{\Ex{\varphi(u^N) - \varphi(u)}} \le c_\varphi N^{-\alpha} \quad \; \forall (N\ge 1 \quad \& \quad \varphi \in \bF),
    \end{split}
  \]
  for some $u \in \cap_{p\ge 2} L^p_n(\Omega, \bR^d)$ and
  $\{u^N\}_N \subset \cap_{p\ge 2} L^p_n(\Omega, \bR^d)$ and
  (observable-dependent constant) $c_{\varphi}>0$, then there exists another
  (observable-dependent constant) $\tilde c_{\varphi}>0$ such that 
  \[
    \abs{\Ex{\varphi( \Psi^N_n(u^N)) - \varphi(\Psi_n(u))}} \le \tilde c_\varphi N^{-\alpha} \quad\forall (N\ge 1 \quad \& \quad \varphi \in \bF) \, ;
  \]
  
\item[(iii)]  for any $p\ge 2$, there exists a $c_p>0$ such that 
  \[
    \norm{\Psi^{N}_n(u)- \Psi^{N}_n(v)}_p < c_{p}\norm{u-v}_p \quad \; \forall \big( N \in \bN  \quad \& \quad u,v \in \cap_{p \ge 2} L^p_n(\Omega, \bR^d) \, \big);
  \]

\item[(iv)] there exists a $c_3>0$ such that
  the computational cost of the numerical solution $\Psi^N_n$ is
  $\bP$-almost surely bounded by
  \[
   \mathrm{Cost}(\Psi^N_n(u)) \le c_{3} N
  \quad \; \forall \big( N \in \bN  \quad  \& \quad u \in \cap_{p \ge 2} L^p_n(\Omega, \bR^d) \, \big);
  \]

\end{itemize}
\end{assumption}

\begin{remark}\label{rem:firstOnSde}
  If the dynamics~\eqref{eq:hmmDynamics} is given by the
  SDE~\eqref{eq:SDE} where the coefficients $a:\bR^d \to \bR^d$, $b:
  \bR^{d}\to \bR^{d \times d_W}$ satisfy $\partial^{\kappa}a_j,
  \partial^{\kappa}b_{jk} \in C_B(\bR^d, \bR)$ for all $1\le j\le d$,
  $1 \le k \le d_W$ and $|\kappa|\ge1$, and if $\Psi^N_n$ denotes the
  Euler--Maruyama numerical solution of~\eqref{eq:SDE} using $N$
  uniform timesteps, then Assumption~\ref{ass:psi} holds with $\alpha=
  1$ and $\bF=C^{4}_P(\bR^d,\bR)$, cf.~\cite[Chapter 7]{talayGraham}
  and~\cite[Thm 14.5.2]{KloedenPlaten}. Noting that the
  uniformly-bounded constraint is only imposed on partial derivatives
  of the coefficients, and not the coefficients themselves, the stated
  convergence rates for Euler--Maruyama apply for instance to the
  nonlinear dynamics~\eqref{eq:nonlinearDynamicsFigure}.
\end{remark}

\subsection{EnKF} \label{sec:EnKF}

For a solver $\Psi^N$ with fixed resolution $N\ge1$ and a fixed
ensemble size $P \ge1$, the EnKF method consists of an ensemble of $P$
particles that are iteratively simulated forward and updated
in a way that can be viewed as a nonlinear extension of
Kalman filtering.
The evolution of the EnKF ensemble over the times $n=0,1,\ldots$ can be described
as follows:

We denote the updated ensemble at time $n$ by
$\vHat_{n,1:P}^{N,P} := \{ \vHat_{n,i}^{N,P}\}_{i=1}^P$
and the prediction ensemble at the same time by
$v_{n,1:P}^{N,P} := \{ v_{n,i}^{N,P}\}_{i=1}^P$.
At time $0$, $\vHat_{0,1:P}^{N,P} := \{ \vHat_{0,i}^{N,P}\}_{i=1}^P$,
consists of $P$ independent and $\Prob{u_0|Y_0}$-distributed particles,
where we assume the initial distribution can be sampled exactly. The
empirical measure induced by the ensemble $\vHat_{0,1:P}^{N,P}$ may
be viewed as the EnKF approximation of $\Prob{u_0|Y_0}$.  Given
an updated ensemble at time $n\ge0$, the prediction ensemble at time
$n+1$ is computed by simulating each particle forward
\[
v_{n+1,i}^{N,P} = \Psi^N_n(\vHat_{n,i}^{N,P}) \quad \text{for} \quad i =1,2,\ldots,P.
\]
Thereafter, the updated ensemble at time $n+1$
is computed through  assimilating the new measurement $y_{n+1}$
through the Kalman-filter-like and particle-wise formula
\[
\vHat_{n+1,i}^{N,P} = (I-K_{n+1}^{N,P} H) v_{n+1,i}^{N,P} + K_{n+1}^{N,P} \yTilde{n+1,i} \quad \text{for} \quad i =1,2,\ldots,P.
\]
Here
\[
K_{n+1}^{N,P}   = C^{N,P}_{n+1} H^\bT (H C^{N,P}_{n+1} H^\bT + \Gamma)^{-1}
\]
denotes the Kalman gain,
\begin{equation}\label{enkf:biasedCov}
  \begin{split}
C_{n+1}^{N,P} &= \Cov[v_{n+1}^{N,P}] := \sum_{i=1}^P \frac{v_{n+1,i}^{N,P} \prt{v_{n+1,i}^{N,P}}^\bT }{P} - 
\sum_{i=1}^P \frac{v_{n+1,i}^{N,P} }{P} \prt{\sum_{i=1}^P \frac{v_{n+1,i}^{N,P} }{P}}^\bT
\end{split}
\end{equation}
denotes the \emph{biased} sample covariance of the ensemble
$\vHat_{n,1:P}^{N,P}$,
\[ 
\yTilde{n+1,i}	 = y_{n+1} + \eta_{n+1,i}, \quad \text{for} \quad i =1,2,\ldots, P
\]
are iid perturbed observations with $\eta_{n+1,1} \sim N(0, \Gamma)$
and $\eta_{j,i} \perp u_k$ for all $j\ge 1$, $k \ge 0$.
Perturbed observations were originally introduced in \cite{burgers1998analysis} 
to correct an unwanted covariance-deflation feature
in the original formulation of EnKF~\cite{evensen1994sequential}.

We note that for all filtering methods considered in this work, the
iteratively augmented observation sequence $y_1,y_2, y_3, \ldots$ is
assumed to be given, i.e., \emph{non-random}. The sources of
randomness in the EnKF filter are therefore the driving noise
in the dynamics and the perturbed observations. For EnKF with
dynamics resolution $N$, randomness thus enter through
$\{\Psi_n^{N,P}\}_n$ and $\{\eta_{n,i}\}_{n,i}$. Observe further
that since the particles
$\vHat_{n,1}^{N,P}, \vHat_{n,2}^{N,P}, \ldots,\vHat_{n,P}^{N,P}$ are
identically distributed, the distribution of $\vHat_{n}^{N,P}$
will depend on the ensemble size $P$ and the model resolution $N$.

If Assumption~\ref{ass:psi} holds, then Theorem~\ref{thm:enkfConv} implies
that 
$\vHat_{n,1:P}^{N,P} \subset \cap_{p\ge 2} L^p_n(\Omega, \bR^d)$
for any $P\in \bN$, $N \in \bN$, and $n \in \bN_0$.
This ensures the existence of the EnKF empirical measure
\begin{equation*}
\mu_n^{N,P}(dv) = \frac{1}{P} \sum_{i=1}^P \delta(dv;\hat v_{n,i}^{N,P}),
\label{eq:emp}
\end{equation*}
with the Dirac measure 
\[
  \delta(dv;y) := \begin{cases}
    1 &  \text{if} \quad y \in dv \\
    0 & \text{otherwise}.
  \end{cases}
\]

\begin{notation}\label{notation:not2}
The expectation of a QoI $\varphi : \bR^d \to \bR$
with respect to the a probability measure $\mu: \cB^d \to [0,1]$
is denoted by
\begin{equation*}
\mu [\varphi] := \int_{\bR^d} \varphi(v) \, \mu(dv).
\label{eq:empObs}
\end{equation*}
\end{notation}
Note that the expectation with respect to the EnKF empirical measure
takes the form of a sample average
\[
\mu_n^{N,P} [\varphi] = \frac{1}{P} \sum_{i=1}^P \varphi(\vHat_{n,i}^{N,P}),
\]
and that $\mu_n^{N,P} [\varphi]$ in fact is a random variable since it is a
function of the ensemble's $P$ particles $\vHat_{n,i}^{N,P}$
(where $P$ obviously is a finite number). For the mean-field
measure $\bar \mu_n = \mu_{n}^{\infty,\infty}$, on the other hand,
$\bar \mu_n [\varphi]$ is a deterministic value;
this is a consequence of $P = \infty$.

\subsection{MLEnKF}\label{sec:mlenkfIntro}
The MLEnKF estimator that we introduce in this work is a sample
average of pairwise coupled EnKF estimators
$\{\mu_n^{N_\ell,P_\ell}[\varphi]\}_{\ell=0}^{L}$ on a hierarchy of
resolution levels $\{N_\ell\}_{\ell=0}^L \subset \bN$ and
ensemble-size levels $\{P_\ell\}_{\ell=0}^L\subset \bN$.
Here, $L$
refers to the finest resolution level of the MLEnKF estimator, and we
impose the following exponential-growth constraints on the resolution
and ensemble-size sequences:
\begin{equation*}
  P_{\ell+1}= 2P_{\ell} \quad \text{and}\quad  N_\ell \eqsim 2^{s \ell} \quad \text{for some } s>0.
\end{equation*}

Before writing the explicit form of the MLEnKF estimator, note that
since our goal, the value we want to approximate, is deterministic, 
cf.~Notation~\ref{notation:not2}, the following equality holds
\[
  \bar{\mu}_n[\varphi]=\mu_n^{\infty,\infty}[\phi] =
  \Ex{\mu_n^{\infty,\infty}[\phi]}.
\]
This motivates the approximation
\[
  \bar{\mu}_n[\varphi] \approx \Ex{\mu_n^{N_L,P_L}[\phi]} =
  \Ex{\mu_n^{N_{0},P_{0}}[\varphi] } +
  \sum_{\ell=1}^{L}\Ex{\mu_n^{N_\ell,P_\ell}[\varphi]-\mu_n^{N_{\ell-1},P_{\ell-1}}[\varphi]},
\]
where the last equality holds due to the linearity of the expectation
operator. In the expectations for $\ell \ge 1$ we seek to employ
pairwise coupling between the fine-level estimator
$\mu_n^{N_\ell,P_\ell}[\varphi]$ and the coarse-level estimator
$\mu_n^{N_{\ell-1},P_{\ell-1}}[\varphi]$ through associating the
driving noise and perturbed observations for each particle on level
$\ell$ uniquely to one particle on level $\ell-1$. As the number of
particles on the neighboring levels are not equal
($P_{\ell} = 2 P_{\ell-1}$), this is clearly not possible in the
current form of expectations of differences, so we replace
$\mu_n^{N_{\ell-1},P_{\ell-1}}[\varphi]$ by the average of two
identically distributed random variables
$\mu_n^{N_{\ell-1},P_{\ell-1},1}[\varphi]$ and
$\mu_n^{N_{\ell-1},P_{\ell-1},2}[\varphi]$. Importantly, these satisfy
\[
\Ex{\frac{\mu_n^{N_{\ell-1},P_{\ell-1},1}[\varphi] + \mu_n^{N_{\ell-1},P_{\ell-1},2}[\varphi] }{2}} = \Ex{\mu_n^{N_{\ell-1},P_{\ell-1}}[\varphi]} \quad k =1,2,
\]
which implies that the following equality holds
\[
\begin{split}
\Ex{\mu_n^{N_L,P_L}[\phi]} =  &\Ex{\mu_n^{N_{0},P_{0}}[\varphi] } 
  +
 \sum_{\ell=1}^{L}\Ex{\mu_n^{N_\ell,P_\ell}[\varphi]-\frac{\mu_n^{N_{\ell-1},P_{\ell-1},1}[\varphi] + \mu_n^{N_{\ell-1},P_{\ell-1},2}[\varphi]}{2}}.
\end{split}
\]
By approximating the $\ell$-th
the expectation by a sample average using $M_\ell$ iid pairwise-coupled samples
\[
\left\{ \mu_n^{N_\ell,P_\ell,m}[\varphi] - \frac{\mu_n^{N_{\ell-1},P_{\ell-1},1,m}[\varphi] + \mu_n^{N_{\ell-1},P_{\ell-1},2,m}[\varphi]}{2} \right\}_{m=1}^{M_\ell},
\]
we obtain the MLEnKF estimator
\begin{equation}\label{eq:telescsum}
\begin{split}
  \mu^{ML}[\varphi] :=   &\sum_{m=1}^{M_0}\frac{\mu_n^{N_0,P_0,m}[\varphi]}{M_0}\\
  &+
  \sum_{\ell=1}^{L} \sum_{m=1}^{M_\ell}\frac{\mu_n^{N_\ell,P_\ell,m}[\varphi]}{M_\ell}-\frac{ \Big(\mu_n^{N_{\ell-1},P_{\ell-1},1,m}+\mu_n^{N_{\ell-1},P_{\ell-1},2,m} \Big)[\varphi]}{2M_\ell}.
\end{split}
\end{equation}
As for classic MLMC estimators, the degrees of freedom $L$ and $\{M_\ell\}_{\ell=0}^L \subset \bN$
are determined through a constrained optimization problem where one
seeks to equilibrate the variance contribution from all of the
sample averages and the bias error cf.~Corollary~\ref{cor:mlenkf}.
The purpose of pairwise-coupled particles is to reduce the variance
of the MLEnKF estimator and thereby improve the efficiency of the 
method by reducing the number of samples $\{M_\ell\}$ needed to
control the approximation error.

To describe the coupling between $\mu_n^{N_\ell,P_\ell}$ and
$(\mu_n^{N_{\ell-1},P_{\ell-1},1}$, $\mu_n^{N_{\ell-1},P_{\ell-1},2})$,
we first introduce the following notation for the
particles of the three underlying EnKF ensembles: for $\quad i =1,2\ldots,P_\ell$,
\[
  \begin{split}
    \vHat_{n, i}^{\ell,\ff} &:= \vHat_{n, i}^{N_\ell,P_\ell} \\
    \vHat_{n, i}^{\ell,\fc} &:=
    \begin{cases}
      \vHat_{n, i}^{\ell,\fc_1}:= \vHat_{n, i}^{N_{\ell-1},P_{\ell-1},1}& \text{if} \quad i \le P_{\ell-1}\\
      \vHat_{n, i}^{\ell,\fc_2}:= \vHat_{n, i-P_{\ell-1}}^{N_{\ell-1},P_{\ell-1},2}& \text{if} \quad P_{\ell-1} <  i \le P_\ell,
      \end{cases}
  \end{split} 
\]
where $\vHat_{n, i}^{N_{\ell-1},P_{\ell-1},k}$ refers to the $i$-th
updated particle at time $n$ of the EnKF measure
$\mu_n^{N_{\ell-1},P_{\ell-1},k}$, for $k=1,2$.  The superscript $\ff$
refers to fine-level particles simulated with numerical resolution 
$N_\ell$ whereas $\fc$ refers to the coarse-level
particles simulated with numerical resolution $N_{\ell-1}$.  The
set of all the fine- and coarse-level particles are respectively
denoted by
\[
  \vHat_{n, 1:P_\ell}^{\ell,\ff}:=\{\vHat_{n,
    i}^{\ell,\ff}\}_{i=1}^{P_\ell} \quad \text{and} \quad \vHat_{n,
    1:P_\ell}^{\ell,\fc}:=\{\vHat_{n, i}^{\ell,\fc}\}_{i=1}^{P_\ell},
\]
with the EnKF ensemble $\vHat_{n, 1:P_\ell}^{\ell,\ff}$ inducing the empirical
measure $\mu_n^{\ell,\ff} := \mu_n^{N_{\ell},P_{\ell}}$.
The set of all the coarse-level particles is a union of two EnKF ensembles with
$P_{\ell-1}$ particles in each, namely  
$\vHat_{n, 1:P_\ell}^{\ell,\fc}=\vHat_{n, 1:P_{\ell-1}}^{\ell,\fc_1}
\cup \vHat_{n, 1:P_{\ell-1}}^{\ell,\fc_2}$, where
$\vHat_{n, 1:P_{\ell-1}}^{\ell,\fc_1}:=\{\vHat_{n,
  i}^{\ell,\fc}\}_{i=1}^{P_{\ell-1}}$ and
$\vHat_{n, 1:P_{\ell-1}}^{\ell,\fc_2}:=\{\vHat_{n,
  i}^{\ell,\fc}\}_{i=P_{\ell-1}+1}^{P_{\ell}}$.
The EnKF ensemble $\vHat_{n, 1:P_\ell}^{\ell,\fc_k}$ induces the empirical
measure $\mu_n^{\ell,\fc_k} := \mu_n^{N_{\ell-1},P_{\ell-1},k}$ for $k=1,2$.
By defining $\mu_{n}^{\ell,\fc}:=\prt{\mu_{n}^{\ell,\fc_1}+\mu_{n}^{\ell,\fc_2}}/2$ and imposing the convention
\[
  \mu_{n}^{0,\fc}=\mu_{n}^{0,\fc_1}=\mu_{n}^{0,\fc_2}=0 \quad \forall n \ge 0,
\]
the MLEnKF estimator~\eqref{eq:telescsum} can be written
\begin{equation}\label{eq:mlEstimator}
  \begin{split}
\mu_n^{\ML}[\varphi]=
  \sum_{\ell=0}^{L} \sum_{m=1}^{M_\ell}\frac{\mu_n^{\ell,\ff,m}[\varphi]- \mu_n^{\ell,\fc,m}[\varphi]}{M_\ell}
\end{split} 
\end{equation}
where the sequence $\{(\mu_n^{\ell,\ff,m}[\varphi], \mu_n^{\ell,\fc,m}[\varphi])\}_m$ are iid draws from the joint distribution of coupled random variables $\bP_{(\mu_n^{\ell,\ff}[\varphi], \mu_n^{\ell,\fc}[\varphi])}$.

\subsubsection{One prediction-update iteration of pairwise coupled ensembles}\label{sssec:alg}
For $\ell =0,1,\ldots, L$, the fine-level initial ensemble
$\{\vHat_{0, i}^{\ell,\ff}\}_{i=1}^{P_\ell}$ is independent and
identically $\Prob{u_0|Y_0}$-distributed and it is pairwise coupled to
the coarse-level initial ensemble through setting
$\vHat_{0, i}^{\ell,\ff}=\vHat_{0, i}^{\ell,\fc}$.  For
pairwise-coupled updated-state particles $\vHat_{n, i}^{\ell,\ff}$ and
$\vHat_{n, i}^{\ell,\fc}$ at time $n$, we denote by
the next-time prediction state respectively by 
$v^{\ell,\ff }_{n+1,i}$ and $v^{\ell,\fc }_{n+1,i}$.
The prediction state is obtained by
simulation of model dynamics on neighboring resolutions:
\begin{equation}\label{ml:prediction}
\left.\begin{split}
v^{\ell,\ff }_{n+1,i}&= \Psi^{N_\ell}_n(\hat{v}^{\ell,\ff }_{n,i}) \\
v^{\ell,\mathbf{c} }_{n+1,i} &= \Psi^{N_{\ell-1}}_n (\hat{v}^{\ell, \mathbf{c}}_{n,i})
\end{split}\right\} \quad i =1,\ldots, P_\ell.
\end{equation}
The simulations are pairwise coupled through sharing the same driving noise $W$.
Then the sample covariance matrices and Kalman gains for the
respective ensembles are computed by
\begin{equation*}
\begin{split}
C_{n+1}^{\ell, \ff}=\Cov [v^{\ell, \ff}_{n+1}], \qquad K_{n+1}^{ \ell, \ff }&=C_{n+1}^{\ell, \ff}H^\bT(HC_{n+1}^{\ell, \ff}H^\bT+\Gamma)^{-1},\\
C_{n+1}^{\ell, \mathbf{c_1} }=\Cov[v^{\ell, \mathbf{c_1} }_{n+1}], \qquad K_{n+1}^{\ell,\mathbf{c_1}}&= C_{n+1}^{\ell, \mathbf{c_1} }H^\bT(HC_{n+1}^{\ell, \mathbf{c_1} }H^\bT+\Gamma)^{-1},\\
C_{n+1}^{\ell, \mathbf{c_2} }=\Cov[v^{\ell, \mathbf{c_2} }_{n+1}], \qquad  K_{n+1}^{\ell, \mathbf{c_2} }&= C_{n+1}^{\ell, \mathbf{c_2} }H^\bT(HC_{n+1}^{\ell, \mathbf{c_2} }H^\bT+\Gamma)^{-1},
\end{split}
\end{equation*}
where the sample covariances are defined in a similar fashion as in for EnKF:
\begin{equation}\label{ml:biasedCov}
\begin{split}
\Cov[v^{\ell, \ff }_{n+1}]:&=\sum_{i=1}^{P_{\ell}} \frac{v^{\ell, \ff }_{n+1,i}\prt{v^{\ell, \ff }_{n+1,i}}^\bT }{P_{\ell}} - 
\sum_{i=1}^{P_{\ell}} \frac{v^{\ell, \ff}_{n+1,i} }{P_{\ell}} \prt{\sum_{i=1}^{P_{\ell}} \frac{v^{\ell,\ff }_{n+1,i} }{P_{\ell}}}^\bT, \\
\Cov[v^{\ell, \mathbf{c_j} }_{n+1}]:&=\sum_{i=1}^{P_{\ell-1}} \frac{v^{\ell, \mathbf{c_j} }_{n+1,i}\prt{v^{\ell, \mathbf{c_j} }_{n+1,i}}^\bT }{P_{\ell-1}} - 
\sum_{i=1}^{P_{\ell-1}} \frac{v^{\ell, \mathbf{c_j} }_{n+1,i} }{P_{\ell-1}} \prt{\sum_{i=1}^{P_{\ell-1}} \frac{v^{\ell, \mathbf{c_j} }_{n+1,i} }{P_{\ell-1}}}^\bT, \quad j=1,2.
\end{split}
\end{equation}
And the new observation $y_{n+1}$ is assimilated into the filter in the update step:
\begin{equation}\label{ml:update}
  \begin{split}
&\left.
  \begin{split}
\tilde{y}_{n+1,i}^{\ell}&=y_{n+1}+\eta_{n+1,i}^{\ell}  \\
\vHat_{n+1,i}^{\ell,\ff}&=(I-K_{n+1}^{ \ell, \ff }H)v^{\ell,\ff}_{n+1,i}+K_{n+1}^{ \ell, \ff }\tilde{y}_{n+1,i}^{\ell}
\end{split}
\right\}\quad i=1,...,P_\ell,\\
&\left.\begin{split}
\vHat_{n+1,i}^{\ell,\mathbf{c_1}}&=(I-K_{n+1}^{\ell,\mathbf{c_1}}H)v_{n+1,i}^{\ell,\mathbf{c_1}}+K_{n+1}^{\ell,\mathbf{c_1}}\tilde{y}_{n+1,i}^{\ell}\\
\vHat_{n+1,i}^{\ell,\mathbf{c_2}}&=(I-K_{n+1}^{\ell,\mathbf{c_2}}H)v_{n+1,i}^{\ell,\mathbf{c_2}}+K_{n+1}^{\ell,\mathbf{c_2}}\tilde{y}_{n+1,i+P_{\ell-1}}^{\ell}
\end{split}\right\} \quad i=1,...,P_{\ell-1},
\end{split}
\end{equation}
where $\{\eta_{n,i}^{\ell}\}_{n,\ell,i}$ are iid
$N(0,\Gamma)$-distributed random variables. We note that the pairwise
coupling of fine- and coarse-level particles is obtained through them
sharing the same initial condition, driving noise $W$ in the
dynamics~\eqref{ml:prediction}, and perturbed observations. And to further
elaborate on the relation between pairwise-coupled particles and
pairwise-coupled EnKF estimators, the numerator in the $(\ell,m)$-th
summand/sample of~\eqref{eq:mlEstimator} takes the following form when
represented in terms of its particles:
\[
\mu_n^{\ell,\ff,m}[\varphi]- \mu_n^{\ell,\fc,m}[\varphi] =
\begin{cases}
  P_0^{-1} \sum_{i=1}^{P_0} \varphi(\hat v^{0,\ff,m}_{n,i}) & \text{if} \quad \ell =0\\
  P_\ell^{-1}\sum_{i=1}^{P_\ell} \varphi(\hat v^{\ell,\ff,m}_{n,i}) - \varphi(\hat v^{\ell,\fc,m}_{n,i}) & \text{if} \quad \ell \ge 1.
\end{cases}
\]
The motivation for pairwise-coupled particles is primarily to reduce
the variance of the above sample/summand numerators
$\varphi(\vHat_{n}^{\ell,\ff,m})-\varphi(\vHat_{n}^{\ell,\fc,m})$, and
ultimately to improve the efficiency of your MLEnKF estimator.
Provided the variance of these terms are substantially reduced, the
number of samples $\{M_\ell\}$ needed to achieve a certain accuracy
with your MLEnKF estimator may be lowered substantially, and this
will improve the efficiency of your MLEnKF estimator.
See Figure~\ref{MLestimDiagram} for an illustration of one
prediction-update iteration of the full MLEnKF estimator.
\begin{figure}[ht!]
	\includegraphics[width=8cm,height=7cm]{{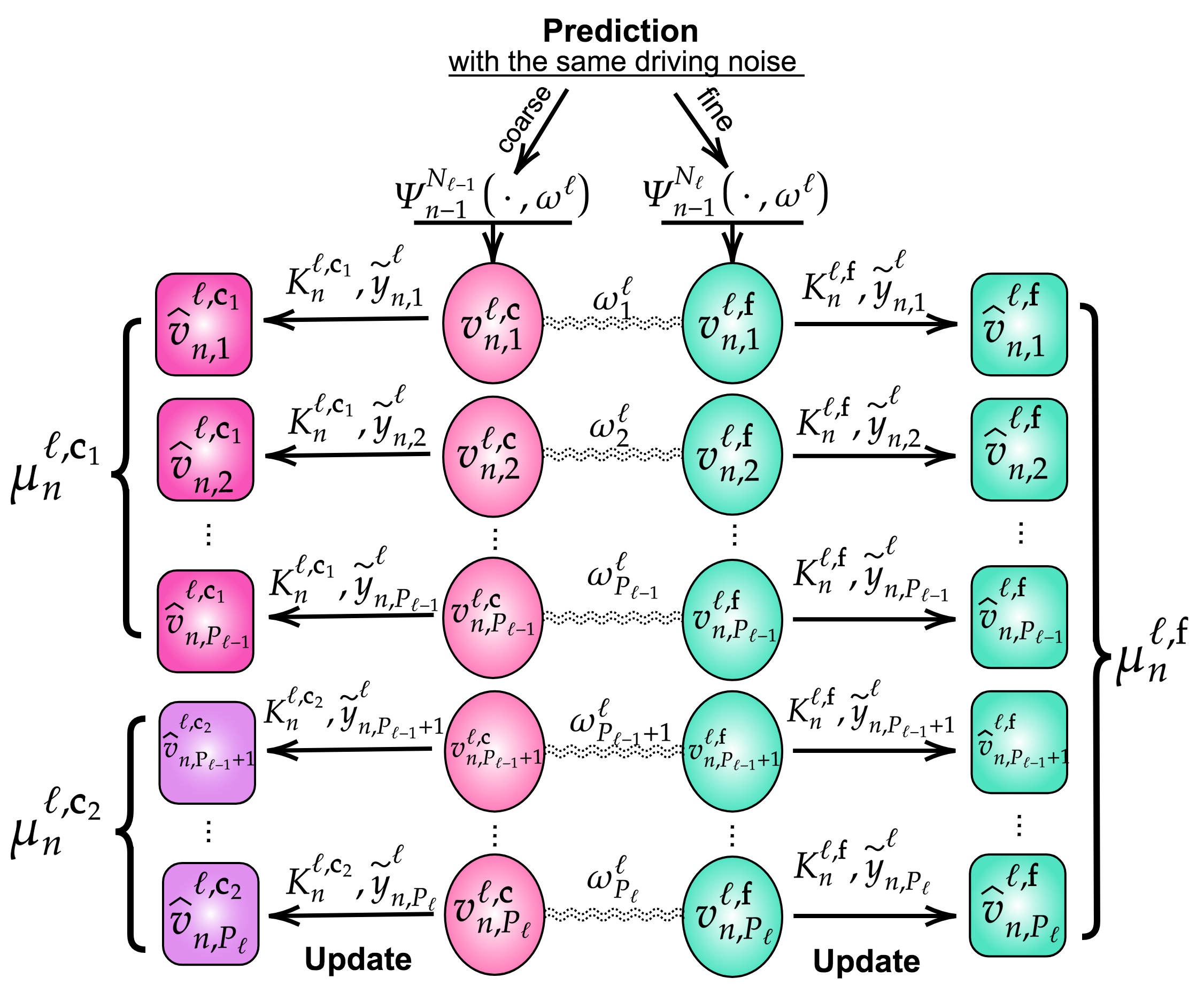}}
	\includegraphics[width=4.5cm,height=5.5cm]{{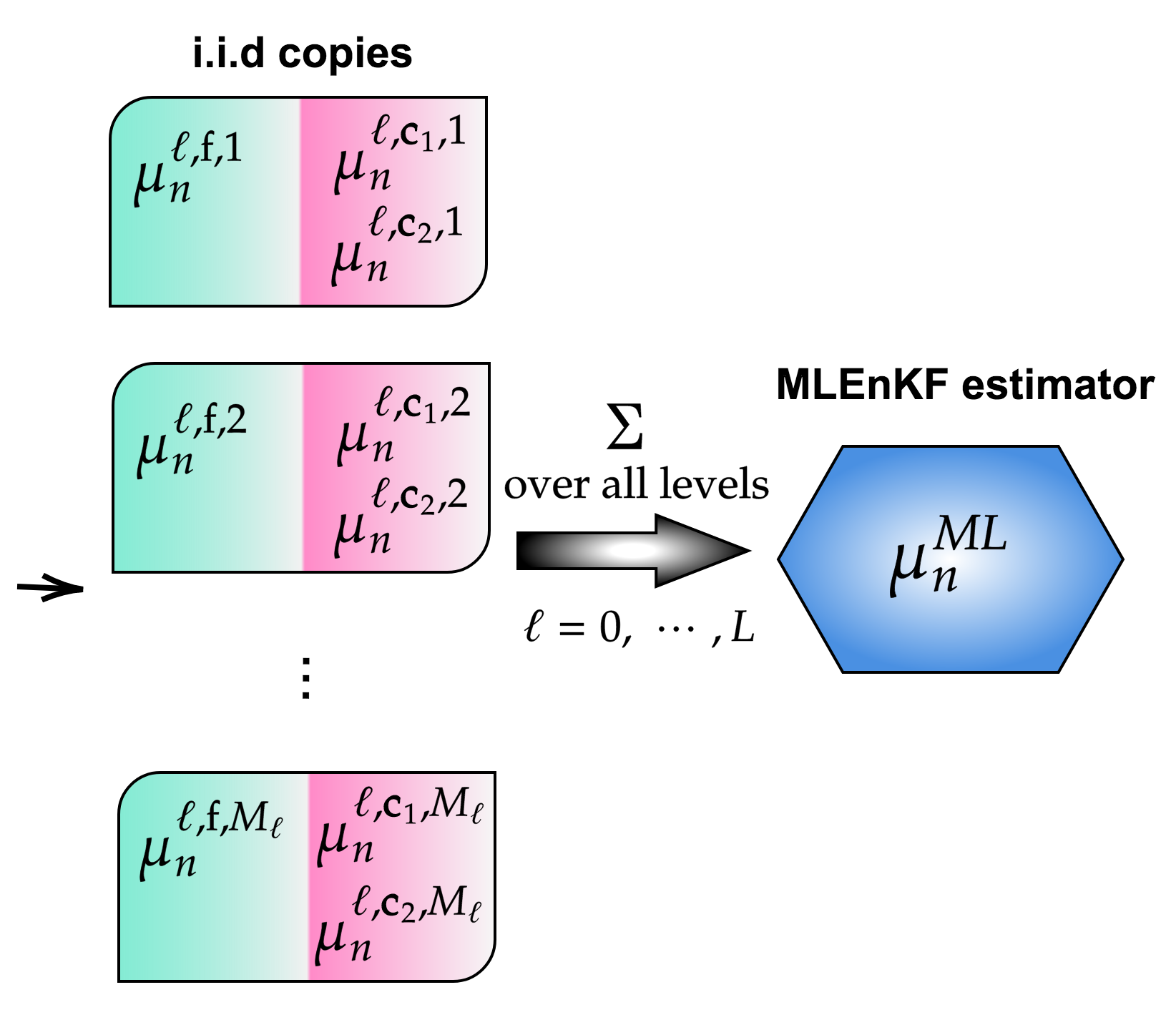}}
	\caption{One prediction-update iteration of the MLEnKF estimator described in Section~\ref{sssec:alg}. Green and pink ovals represent fine- and coarse-level prediction-state particles, respectively, sharing the same initial condition and driving noise $\omega^{\ell}$ and the respective squares represent fine- and coarse-level updated-state particles sharing the perturbed observartions. The MLEnKF estimator is obtained by iid copies of pairwise-coupled samples, cf~\eqref{eq:mlEstimator}.}
	\label{MLestimDiagram}
\end{figure}

\begin{remark}
  For the herein-introduced MLEnKF method all the EnKF estimators
  $\{\mu^{\ell, \cdot,m}[\varphi]\}_{\ell, m}$ are independent for
  different $(\ell,m)-$samples. This in particular implies that also
  all Kalman gains and particles from different $(\ell,m)$-samples are
  independent.  For the ``canonical'' MLEnKF method~\cite{hoel2016}, on
  the other hand, all the particles in the full estimator are updated
  using the one shared Kalman gain. Therefore, the particles in the
  ``canonical'' MLEnKF method are all correlated, while all particles
  from different $(\ell,m)$-samples in the herein-introduced
  estimator are independent. The more extensive
  independence properties of the new MLEnKF estimator simplifies the
  proofs of asymptotic convergence results and in some settings also
  improves theoretical complexity bounds. See
  Corollary~\ref{cor:mlenkf} and Remark~\ref{rem:oldMLEnKF} for a
  comparison of the theoretical complexity of the two methods.

\end{remark}


\subsection{Main results}\label{sec:mainResults}
We now present the main convergence and complexity results for EnKF
and MLEnKF. The proofs for these results are provided in
Appendix~\ref{sec:proofs}. We will need the following assumptions:

\begin{assumption}\label{ass:psi2}
  For any exponentially-growing sequence $\{N_\ell\}$ of natural
  numbers, there exists a constant $c_\Psi >0$ and $\beta>0$ such that
  for all $\ell \in \bN_0 \cup \{\infty\}$, $n \in \bN_0$, and
  $p \ge 2$,
\begin{itemize}
    
    \item[(i)] for all $|\kappa|_1 \le 1$,
    \[
     \norm{\partial^{\kappa} \Psi^{N_\ell}_n(u)}_p \leq c_{\Psi}
    (1+\norm{u}_p) \quad   \forall u \in L^p_n(\Omega, \bR^d),
    \]
    \item[(ii)] for all $|\kappa|_1 =2$,
    \[
    \norm{\partial^{\kappa} \Psi^{N_\ell}_n(u)}_p \leq c_{\Psi}
    (1+\norm{u}_{2p}) \quad   \forall u \in L^{2p}_n(\Omega, \bR^d),
    \]

  \item[(iii)] for all $|\kappa|_1 \le 1$,
    \[
    \norm{\partial^\kappa \Psi^{N_{\ell+1}}_n(u) - \partial^\kappa \Psi^{N_\ell}_n(u)}_p \le c_{\Psi}(1+\norm{u}_p)  N_\ell^{-\beta/2} \quad \forall u\in L^p_n(\Omega, \bR^d).
    \]
      
\end{itemize}
\end{assumption}

\begin{remark}
  If the dynamics~\eqref{eq:hmmDynamics} is given by the SDE
  in Remark~\ref{rem:firstOnSde} and $\Psi^N_n$ denotes the corresponding
  Euler-Maruyama numerical solution, then Assumption~\ref{ass:psi2} holds
  with $\beta= 1$, cf. \cite{szepessy2001adaptive, hoel2012adaptive, hoel2014implementation, hoel2016strong}. 
\end{remark}

\begin{theorem}[Convergence of EnKF]\label{thm:enkfConv}
  If Assumption~\ref{ass:psi} holds, then
  for any $\varphi \in \bF$, $p\ge 2$ and $n \in \bN_0$, 
\begin{equation}\label{eq:enkfConv}
  \|\mu^{N,P}_n [\varphi] - \bar{\mu}_n [\varphi] \|_p \lesssim P^{-1/2} + N^{-\alpha},
\end{equation}
where the hidden constant in $\lesssim$ depends on $p$, $n$ and $d$ (but not on $N$ and $P$).

\end{theorem}

\begin{corollary}\label{cor:enkf}
  If Assumption~\ref{ass:psi} holds, then 
  $N \eqsim \epsilon^{-1/\alpha}$ and $P \eqsim \epsilon^{-2}$
  ensure that 
  \[
    \|\mu^{N,P}_n [\varphi] - \bar{\mu}_n [\varphi] \|_p \lesssim \epsilon
  \]
  for any $\varphi \in \bF$, $p\geq 2$ and $n \in \bN_0$.
  The resulting computational cost is
\begin{equation*}
    \mathrm{Cost}(\mathrm{EnKF}) \eqsim \epsilon^{-(2+\frac{1}{\alpha})}.
\end{equation*}
\end{corollary}

\begin{remark}
  The proof of Theorem~\ref{thm:enkfConv} shows that
  Theorem~\ref{thm:enkfConv} and Corollary~\ref{cor:enkf} apply even
  if the constraint $\bF \subset C^2_P(\bR^d,\bR)$ in
  Assumption~\ref{ass:psi}(ii) is weakened to
  $\bF \subset C^1_P(\bR^d,\bR)$.
\end{remark}

\begin{theorem}[Convergence of MLEnKF]\label{thm:mlenkfConv}
  If Assumptions~\ref{ass:psi} and~\ref{ass:psi2} hold,
  then for any $L\ge 1$ and triplet of sequences
  $\{M_\ell\}, \{N_\ell\}, \{P_\ell\} \subset \bN$
  with $P_\ell \eqsim 2^\ell$ and exponentially-growing $N_\ell$,
  it holds for any
  $\varphi \in \bF$, $n\ge 0$, and $p\ge 2$ that
  \[
    \|\mu_n^{\ML}[\varphi] - \overline \mu_n[\varphi]\|_p \lesssim
    N_L^{-\beta/2}P_L^{-1/2} + P_L^{-1} +N_L^{-\alpha}+
    \sum_{\ell=0}^L M_\ell^{-1/2}(N_\ell^{-\beta/2}P_\ell^{-1/2}+ P_\ell^{-1}),
  \]
  where the hidden constant in $\lesssim$ depends on $p$, $n$ and $d$ (but not on
  $\{M_\ell\}, \{N_\ell\}, \{P_\ell\}$ and $L$).
\end{theorem}

\begin{remark}
In Appendix~\ref{sec:proofs}, the Theorems~\ref{thm:enkfConv}
and~\ref{thm:mlenkfConv} are for simplicity proved for EnKF and
MLEnKF methods using biased sample covariances, precisely as introduced
in~\eqref{enkf:biasedCov} and \eqref{ml:biasedCov},
respectively. However, the proofs can easily be extended, without any changes in
convergence rates, to methods that instead use unbiased sample covariances,
cf.~Remark~\ref{unbiasedCov}.
\end{remark}

\begin{corollary}\label{cor:mlenkf}
  If Assumptions~\ref{ass:psi} and~\ref{ass:psi2} hold, then for any
  $\epsilon, s>0$, the configuration
  $L = \ceil{\frac{\log_2(\epsilon^{-1})}{\min(1,(1+\beta s)/2,
      \alpha s)}}$, $P_\ell \eqsim 2^\ell$, $N_\ell \eqsim 2^{s\ell}$
  and
\begin{equation}\label{eq:chooseMlr}
M_\ell  \eqsim
\begin{cases} 
\epsilon^{-2} \, 2^{-\mathlarger{\frac{3+ 2s+\min(\beta s,1)}{3} \ell}}  +1& \text{if } \min(\beta s,1)>s, \\
\epsilon^{-2} L^2 2^{-(1 + s) \ell} +1,  &\text{if } \min(\beta s,1)=s,\\
\epsilon^{-2 -\mathlarger{ \frac{2(s-\min(\beta s,1))}{3\min(1, (1+\beta s)/2, \alpha s)}}}
\, 2^{-\mathlarger{\frac{3+ 2s+\min(\beta s,1)}{3} \ell}} +1,  &\text{if } \min(\beta s,1)<s,
\end{cases}
\end{equation}
ensures that for any $\varphi \in \bF$, $n\ge 0$ and $p\ge 2$,
\begin{equation}
\|\mu^{\rm ML}_n [\varphi] - \bar{\mu}_n [\varphi] \|_p \lesssim \epsilon.
\label{eq:lperror}
\end{equation}
Moreover, if one sets the parameter $s$ such that 
\begin{equation}\label{eq:sOptimization}
s \in
\begin{cases}
[\alpha^{-1},1) & \text{if } \beta>1 \quad \& \quad \alpha > 1,\\              
[\alpha^{-1},1]  & \text{if } (\beta> 1 \quad \& \quad \alpha =1) \text{ or } (\beta =1  \text{ and } \alpha \ge 1),\\
\{\alpha^{-1}\} & \text{if } (\beta\ge 1 \quad \& \quad \alpha <1) 
\text{ or } (\beta <1  \text{ and } \alpha \le \beta),\\
\{(2\alpha-\beta)^{-1}\} & \text{if } \beta<1 \quad \& \quad \alpha >\beta,
\end{cases}
\end{equation}
then the cost of the MLEnKF estimator is bounded by
\begin{equation}\label{eq:mlenkfCosts}
\cost_(\mathrm{MLEnKF}) \eqsim 
\begin{cases}
\epsilon^{-2} & \text{if } \beta>1 \quad \& \quad \alpha > 1,\\              
\epsilon^{-2}|\log(\epsilon)|^3  & \text{if } (\beta> 1 \quad \& \quad \alpha =1) \text{ or } (\beta =1  \text{ and } \alpha \ge 1),\\
\epsilon^{-(1+ 1/\alpha)} & \text{if } (\beta\ge 1 \quad \& \quad \alpha <1) 
\text{ or } (\beta <1  \text{ and } \alpha \le \beta),\\
\epsilon^{-(2+ (1-\beta)/\alpha)} & \text{if } \beta<1 \quad \& \quad \alpha >\beta.
\end{cases}
\end{equation}
\end{corollary}

The optimization of the parameter $s$ is stated in terms of inclusion sets in~\eqref{eq:sOptimization}
as that may be useful for implementations. For further details on this optimization and
the computational cost of the MLEnKF estimator, see Appendix~\ref{sec:proofs}.3.
The a priori cost-versus-accuracy results of Corollaries~\ref{cor:enkf}
and~\ref{cor:mlenkf} show that in settings where both rates are sharp,
MLEnKF will asymptotically have better complexity than
EnKF.

\begin{remark}\label{rem:oldMLEnKF} For comparison,
  we recall in compact form
  the cost-versus-accuracy result for the ''canonical" MLEnKF~\cite[Theorem 3.2]{hoel2016}:
  For any $\epsilon>0$ and sufficiently smooth QoI $\varphi$ one may choose the estimator's degrees of freedom
  so that
\begin{equation}\label{eq:lperrorOld}
\|\mu^{\rm ML^{OLD}}_n [\varphi]- \bar{\mu}_n [\varphi] \|_p \lesssim  \epsilon,
\end{equation}
where $\mu^{\rm ML^{OLD}}_n [\varphi]$ denotes the ``canonical'' MLEnKF estimator.
The computational cost of achieving this accuracy is bounded by
\begin{equation}\label{eq:mlenkfCostsOld}
\cost_(\mathrm{MLEnKF^{OLD}}) \eqsim 
\begin{cases}
\abs{\log(\epsilon)}^{2n} \epsilon^{-2}, & \text{if} \quad \beta > 1,\\              
\abs{\log(\epsilon)}^{2n+3} \epsilon^{-2} , & \text{if} \quad \beta = 1,\\
\Big(\epsilon/\abs{\log(\epsilon)}^{2n}\Big)^{-\left( 2 + \frac{1-\beta}{\alpha} \right)}, & \text{if} \quad \beta < 1.
\end{cases}
\end{equation}

To achieve the same $\cO(\epsilon)$-accuracy for the two MLEnKF
methods, the theoretical bounds imply that for any
$n\ge 2$ and any admissible $(\alpha,\beta)$-values with $\alpha \ge 1$, 
\[
\lim_{\epsilon \to 0} \frac{\cost_(\mathrm{MLEnKF^{NEW}})}{\cost_(\mathrm{MLEnKF^{OLD}})} = 0.
\]
In other words, if our theoretical bounds are sharp for both methods,
then the new MLEnKF method asymptotically outperforms the
``canonical'' one in terms of computational cost versus accuracy. We
note however that the regularity assumptions imposed for the new
method's hierarchy of numerical solvers,
cf.~Assumption~\ref{ass:psi2}, are more restrictive than those needed
for the ``canonical'' method, and that we have previously
made the conjecture that the error bounds for the ``canonical'' method are
not sharp, cf.~\cite[Remark 3]{hoel2016}.  
\end{remark}

\section{Numerical examples}

\label{sec:numerics}
In this section, we numerically compare the performance of MLEnKF
and EnKF and (numerically) verify the results predicted by our theory.
We will consider a couple of test problems with dynamics
of the form 
\begin{equation}\label{sde:genform}
u_{n+1} = \Psi(u_{n}) :=u_n+ \int_{n}^{n+1} -V'(u_t)dt+ \int_{n}^{n+1} \sigma dW_t, 
\end{equation}
where the potential $V\in C^\infty_P(\bR)$ satisfies
$V'' \in C^\infty_B(\bR)$, the diffusion coefficient is constant; $\sigma=0.5$, and
linear observations~\eqref{eq:observations} with $H=1$ and
$\Gamma=0.1$. For any $N\ge 1$, $\Psi^N$ will in this section denote
the Milstein numerical scheme with uniform timestep $\Delta t = 1/N$.
This yields the convergence rates $\alpha=1$, $\beta=2$.

For any $x \in \bR$, let $\mathrm{Round}(x)$ denote the nearest integer to $x$.
Given an input RMSE constraint $\epsilon>0$, we
seek to equilibrate the magnitude of all error terms
(cf.~\eqref{eq:enkfConv} for EnKF and Corollary~\ref{cor:mlenkf}
for MLEnKF) by setting the degrees of freedom to
\begin{equation}\label{Par: enkf}
P = \mathrm{Round}(8 \epsilon^{-2}) \quad \text{and} \quad  N  = \mathrm{Round}(\epsilon^{-1})
\end{equation}
for EnKF, and for MLEnKF we set
$L= \mathrm{Round}(\log_2(\epsilon^{-1}))-1$, $N_\ell = 2^{\ell+1}$, $P_\ell = 10 \times 2^\ell$,
and
\begin{equation*}\label{mlenkf}
  M_\ell =
  \begin{cases}
    2 \mathrm{Round}(\epsilon^{-2} L^2 2^{-3}) & \text{if} \quad \ell =0\\
    \mathrm{Round}(\epsilon^{-2} L^2 2^{-2\ell-3} ) & \text{if} \quad  1 \le \ell \le  L.
  \end{cases}
\end{equation*}

Over a sequence of iterations $n= 1,\ldots, \cN$, we compare the
performance of the methods in terms of error versus computer runtime
(wall-clock time), where the error is measured in time-averaged-root-mean-squared error (RMSE):
\[
\mbox{RMSE(EnKF)}:=\sqrt{ \frac{1}{100(\mathcal{N}+1)}\sum_{i=1}^{100}\sum_{n=0}^{\mathcal{N}}\abs{\mu^{N,P}_{n,i} [\varphi] - \bar{\mu}_n [\varphi]}^2}\lessapprox \mathrm{Runtime}^{-1/3},
\]
where for a given $\epsilon>0$ with $\alpha=1$ and $\beta=2$, the last approximate asymptotic inequality follows from Corollary~\ref{cor:enkf}
\[
\mbox{RMSE(EnKF)}\approx \prt{ \sum_{n=0}^{\mathcal{N}}\frac{\norm{\mu^{N,P}_{n} [\varphi] - \bar{\mu}_n [\varphi]}_2^2}{\mathcal{N}+1}}^{1/2}\lesssim \epsilon \eqsim \cost_(\mathrm{EnKF})^{-1/3}.
\]
Similarly, 
\begin{equation*}
\begin{split}
\mbox{RMSE(MLEnKF)}&:=\sqrt{ \frac{1}{100(\mathcal{N}+1)}\sum_{i=1}^{100}\sum_{n=0}^{\mathcal{N}}\abs{\mu_{n,i}^{\rm ML}[\varphi]-\bar{\mu}_n[\varphi]}^2} \\
&\approx \prt{ \sum_{n=0}^{\mathcal{N}}\frac{\norm{\mu_{n}^{\rm ML}[\varphi]-\bar{\mu}_n[\varphi]}_2^2}{\mathcal{N}+1}}^{1/2}\lesssim \log(10+\mathrm{Runtime})^{3/2} \mathrm{Runtime}^{-1/2},
\end{split}
\end{equation*}
where the last inequality follows from Corollary~\ref{cor:mlenkf}, since it implies that
\[
\mbox{RMSE(MLEnKF)}\lessapprox \epsilon \eqsim \log(\cost_(\mathrm{MLEnKF}))^{3/2}\cost_(\mathrm{MLEnKF})^{-1/2}.
\]
Note here that $\{\mu_{.,i}^{N,P}[\varphi]\}_{i=1}^{100}$ and $\{\mu_{.,i}^{\rm ML}[\varphi]\}_{i=1}^{100}$ are
computed using iid sequences of EnKF and MLEnKF empirical measures, respectively.

\subsection*{Reference solutions and computer architecture}
For the first test problem we will consider,
the dynamics $\Psi$ is linear. Therefore, we
compute the exact reference solution $\bar{\mu}_n[\varphi]$ 
straightforwardly using the Kalman filter. The second
test problem we will consider has nonlinear dynamics,
and we need to approximate the reference solution. 
A pseudo-reference solution is then obtained by the deterministic mean-field EnKF (DMFEnKF)
algorithm described in \cite{law2016deterministic},
cf.~Appendix~\ref{sec:dmfenkf}.

The numerical simulations were computed in parallel on five cores on
a Mac Pro with Quad-Core Intel Xeon X-5550 8-core processor with 16 GB
of RAM.  The computer code is written in the Julia programming
language~\cite{Julia-2017}, and it can be downloaded from
\url{https://github.com/GaukharSH/mlenkf}.

\subsection{Ornstein-Uhlenbeck process} 

\label{ssec:ou}
We consider the SDE~\eqref{sde:genform} with 
$V(u)=u^2/2$, and the initial condition is $u_0 \sim N(0,\Gamma)$.
This is an Ornstein-Uhlenbeck (OU) process with
the exact solution
\[
\Psi(u_{n}) =u_ne^{-t}+\int_{0}^{t}\sigma e^{s-t}dW_s.
\]
For a sequence of inputs,
$\epsilon = [2^{-4}, 2^{-5}, 2^{-6}, 2^{-7}, 2^{-8}]$ for EnKF and
MLEnKF, $\epsilon = [2^{-4}, 2^{-4.5}, 2^{-5},\ldots, 2^{-9}]$ for
"canonical" MLEnKF, we study the performance of the three methods in
terms of runtime versus RMSE over timeframes $\mathcal{N}=10$ and
$\mathcal{N}=20$ observation times in Figure~\ref{fig:fig1Ex1}.  For
both QoI considered (the mean and the variance), the observations are
in agreement with Corollaries~\ref{cor:enkf} and~\ref{cor:mlenkf}, and
Remark~\ref{rem:oldMLEnKF}.  The new MLEnKF method performs
similarly as the ``canonical'' MLEnKF method and that both methods
asymptotically outperform EnKF.
\begin{figure}[h!]
	\centering
	\includegraphics[width=0.47\textwidth]{{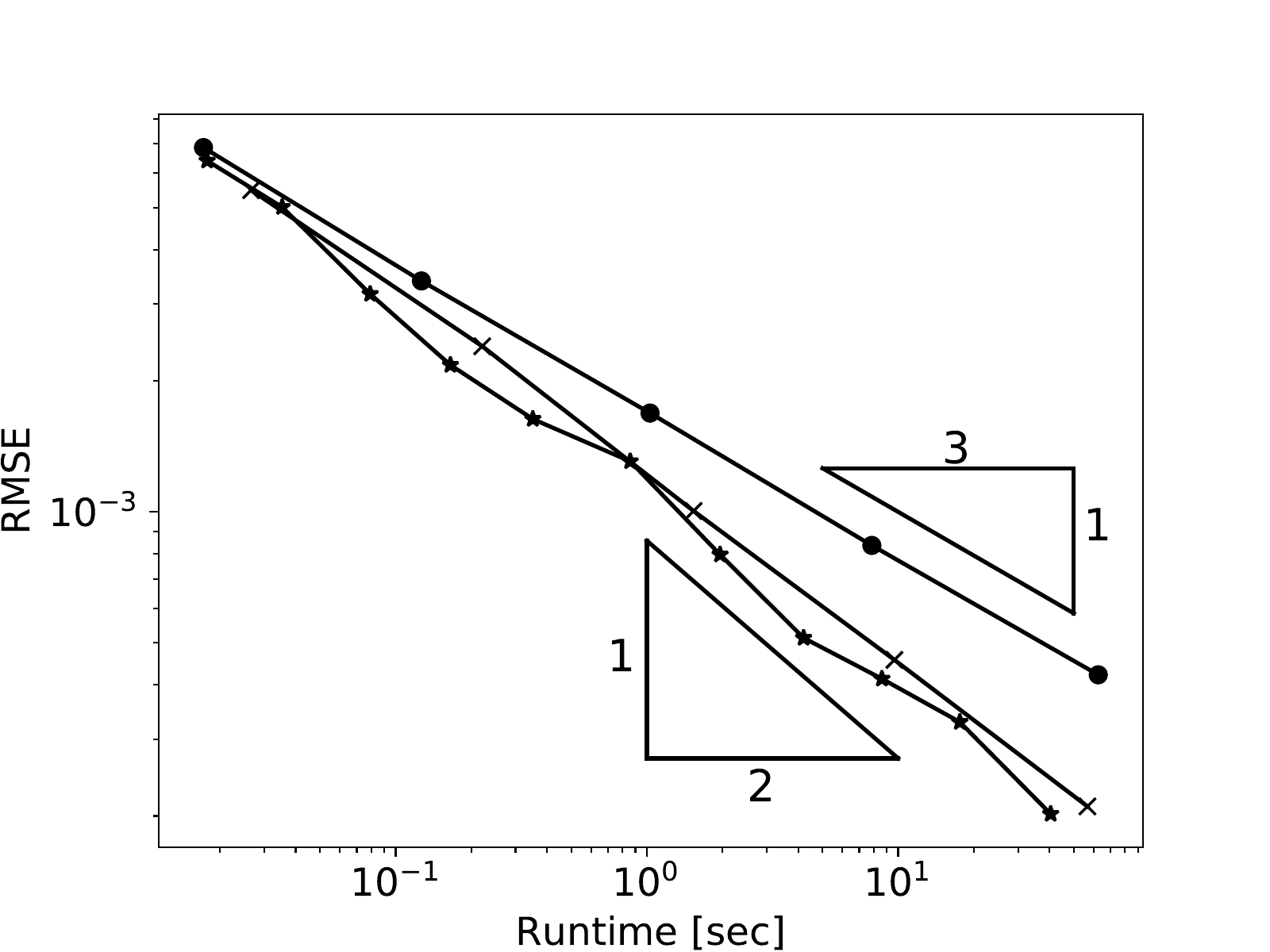}}
\includegraphics[width=0.47\textwidth]{{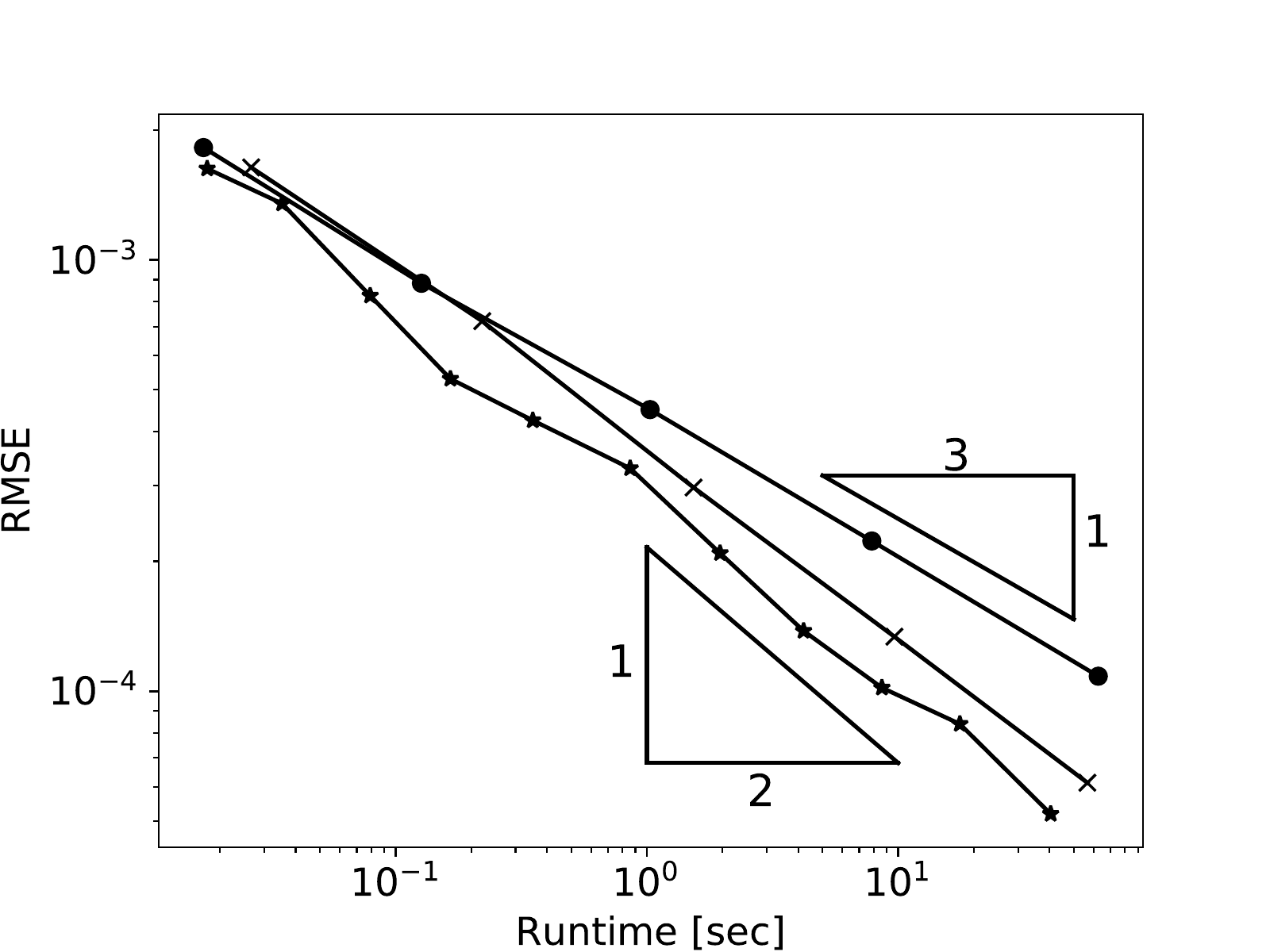}}\\
\includegraphics[width=0.47\textwidth]{{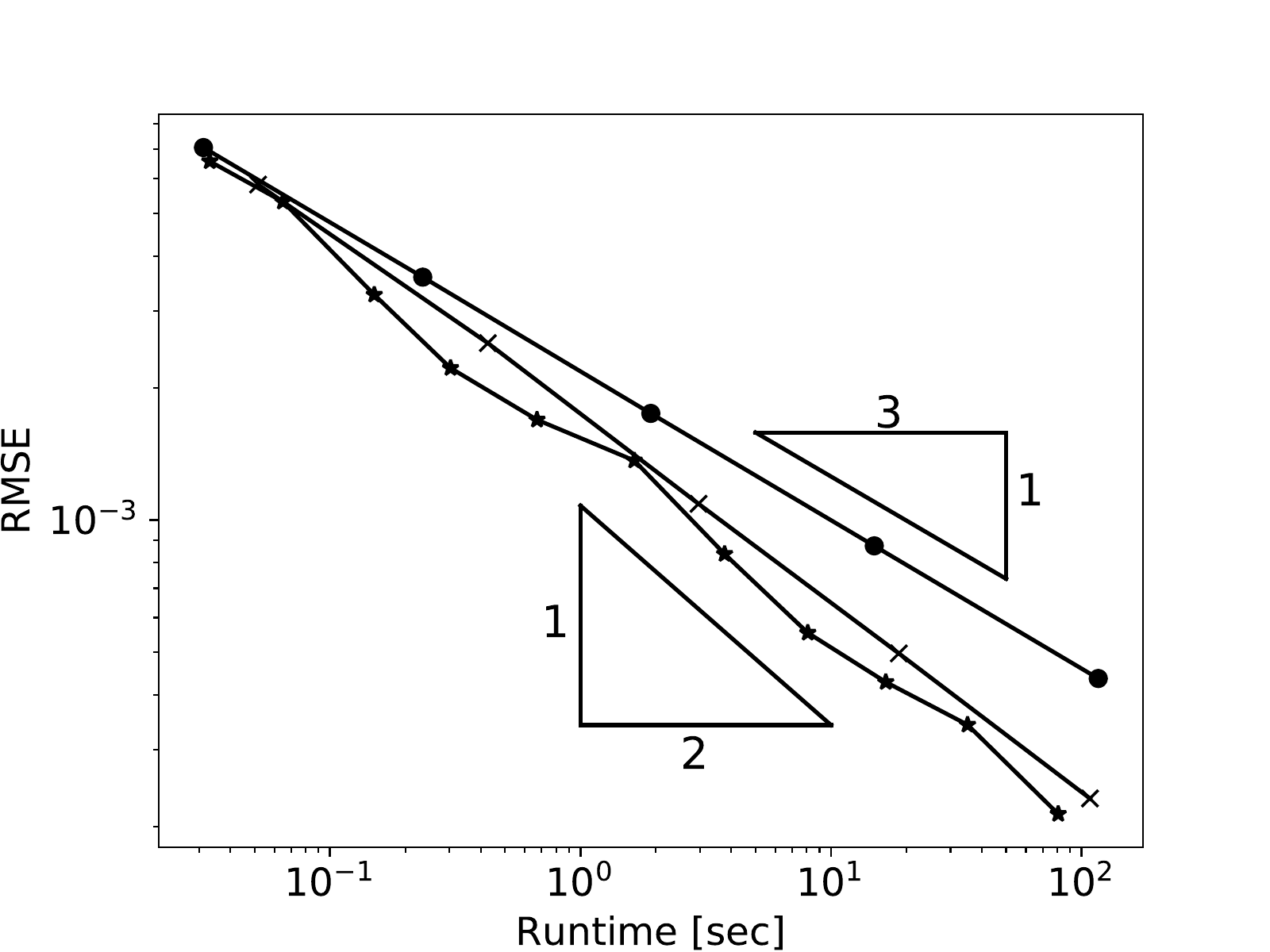}}
\includegraphics[width=0.47\textwidth]{{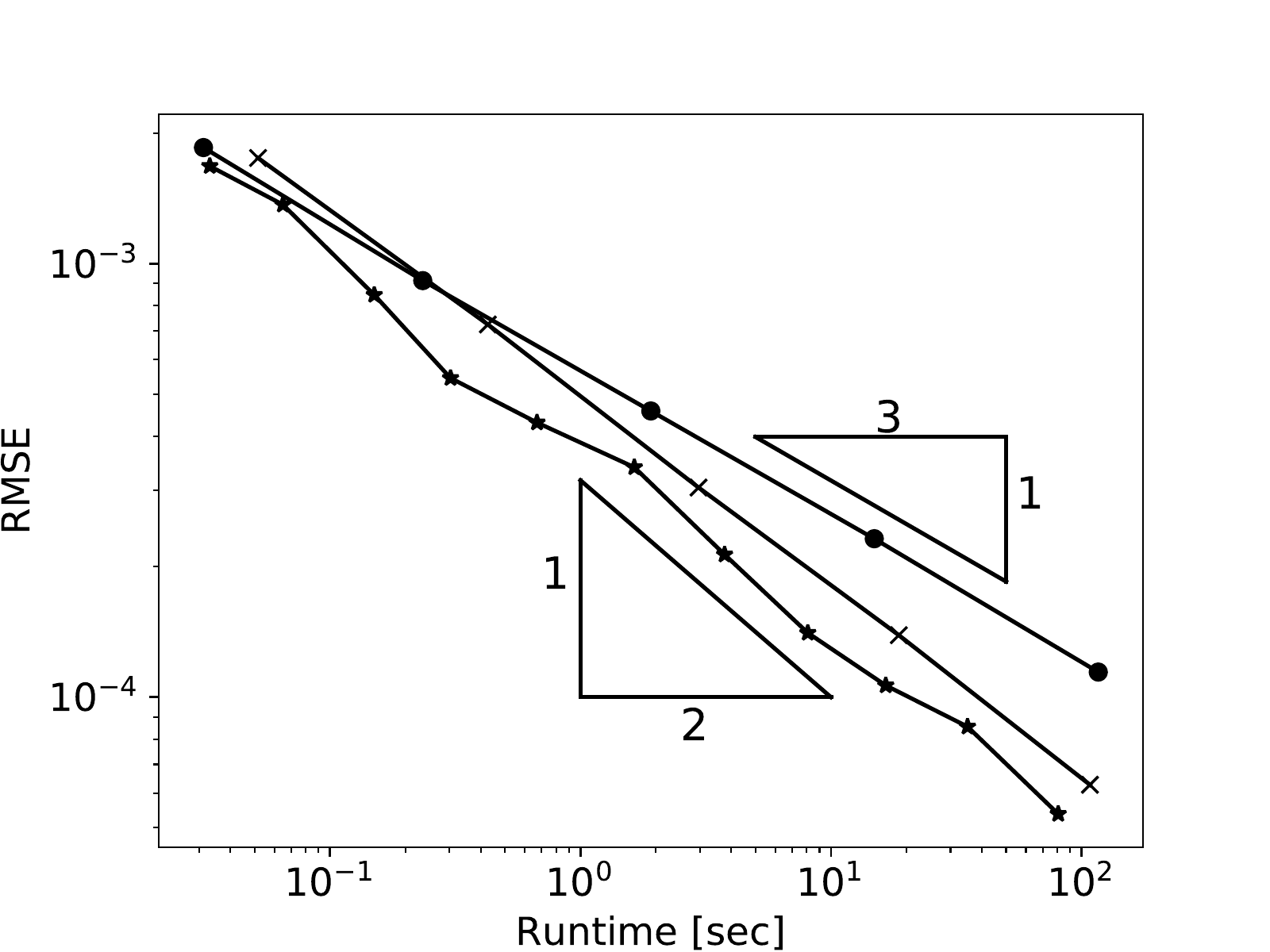}}
	\caption{Top row: comparison of the runtime versus RMSE for
		the QoIs mean (left) and variance (right) over $\cN=10$
		observation times for the problem in Section~\ref{ssec:ou}. The solid-crossed line represents MLEnKF, the solid-asterisk line represents "canonical" MLEnKF and the bottom reference triangle with the slope $\frac{1}{2}$, the solid-bulleted line represents EnKF and
		the upper reference triangle with the slope $\frac{1}{3}$.  Bottom row:
		similar plots over $\cN =20$ observation times.}
	\label{fig:fig1Ex1}
\end{figure}
\subsection{Double-well SDE} 
\label{ssec:dw}
We consider the SDE ~(\ref{sde:genform}) with the
double-well potential
\begin{equation*}\label{eq:dwPot}
V(u)=\frac{1}{2+4u^2}+\frac{u^2}{4}
\end{equation*}
and $u_0 \sim N(0, \Gamma)$. 
This SDE is the metastable motion of particles between the two wells,
cf.~\cite{conrad2018human,
  schutte2013metastability}. Figure~\ref{fig:fig1Ex2} shows the
dynamics of a particle over $\cN=20$ observation times. We observe
that the particle remains in either one of the wells for a relatively
long time; in other words, transitions between wells happen quite
rarely.
\begin{figure}[h!]
	\centering
	\includegraphics[width=0.54\textwidth]{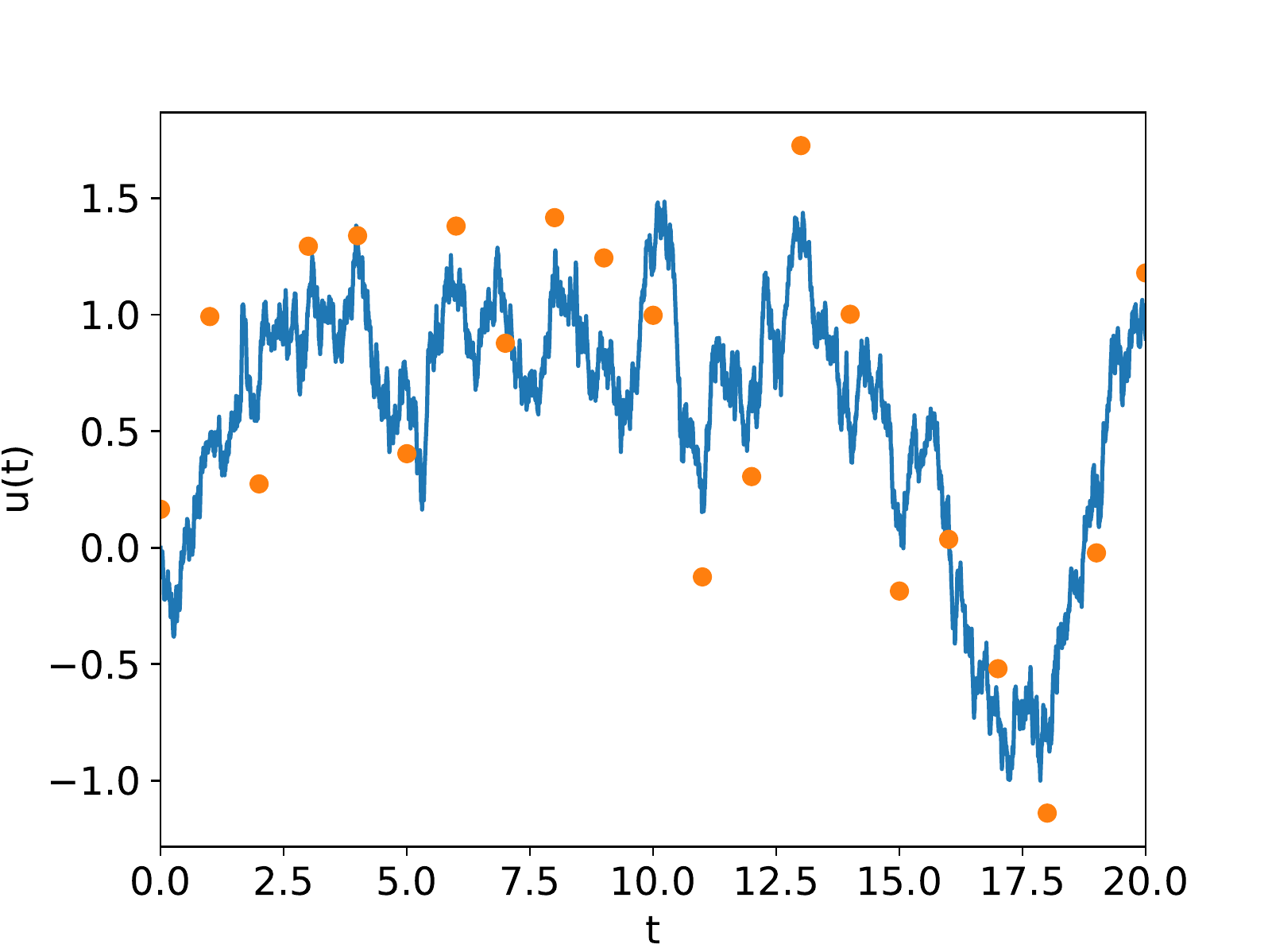}
	\caption{Realization of the double-well SDE from Section~\ref{ssec:dw} over time $\cN=20$ observation times (solid line) and observations (dots). }
	\label{fig:fig1Ex2}
\end{figure}
Figure~\ref{fig:fig2Ex2} illustrates the well transition of an EnKF ensemble over a few predict-update
iterations when the observations and the majority of the ensemble's particles are located in
opposite wells for the first few iterations.
\begin{figure}[h!]
	\centering
	\vspace*{-0.3cm}
	\hspace*{-1.5cm}
	\includegraphics[width=15.2cm, height=17.4cm]{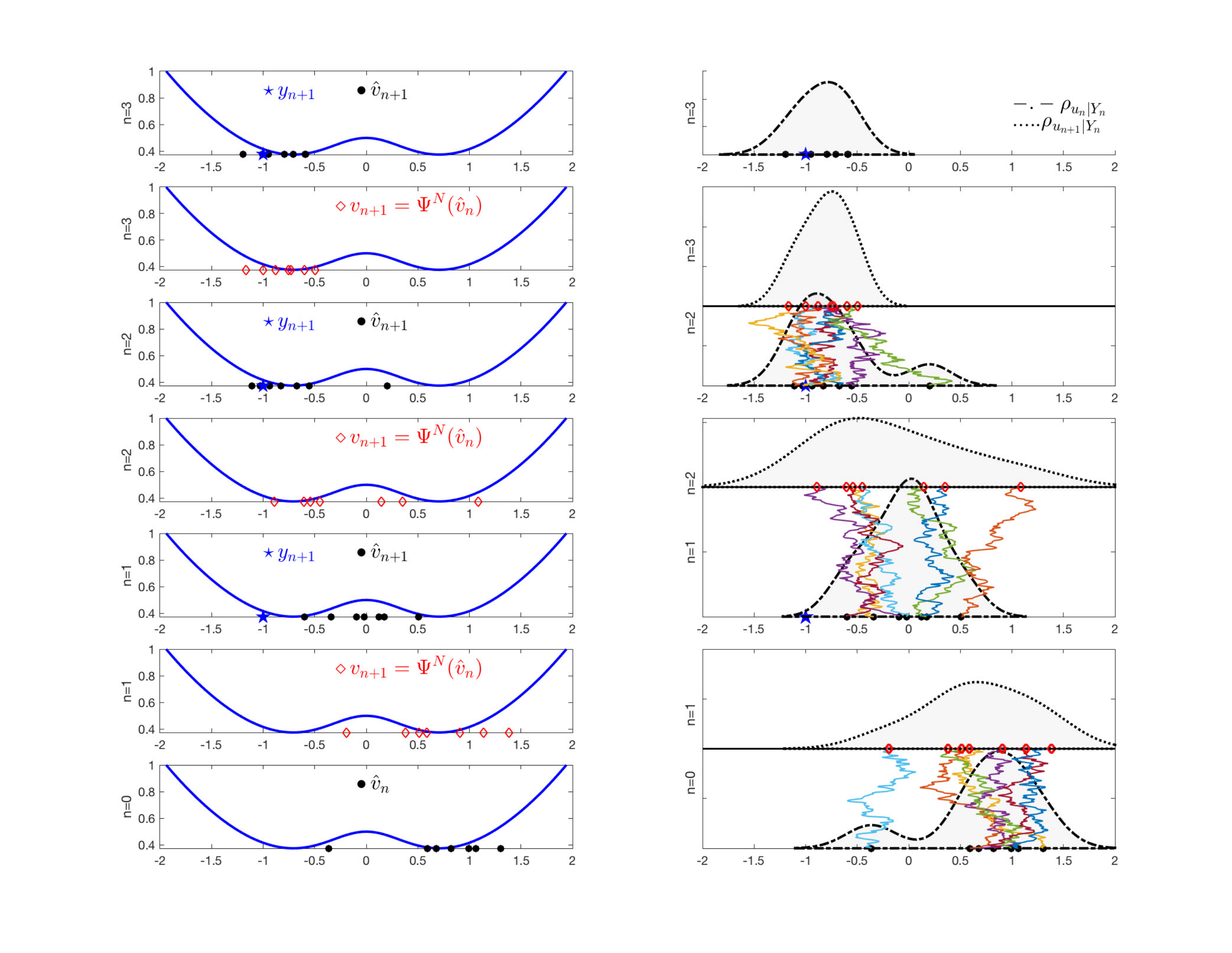}
	\vspace*{-1.5cm}
	\caption{Left column: Well transition of the EnKF ensemble
		when the measurements are located in the opposite
		well. Right column: Animation of the particle paths of the EnKF ensemble
		during the well-transition, and the resulting kernel 
		density estimations of the EnKF prediction and update densities (Section~\ref{ssec:dw}).}
	\label{fig:fig2Ex2}
	\vspace*{-0.5cm}
\end{figure}
Using the same $\epsilon$-input sequences as in the preceding example,
we study the performance of the three methods in terms of runtime
versus RMSE over timeframes of $\mathcal{N}=10$ and $\cN =20$
observation times in Figure~\ref{fig:fig3Ex2}.  Once again, the work rates are in agreement with theory
(Corollaries~\ref{cor:enkf},~\ref{cor:mlenkf} and
Remark~\ref{rem:oldMLEnKF}), that the new MLEnKF method performs
similarly as the ``canonical'' MLEnKF method, and that both methods
asymptotically outperform EnKF.
\begin{figure}[h!]
	\centering
	\includegraphics[width=0.47\textwidth]{{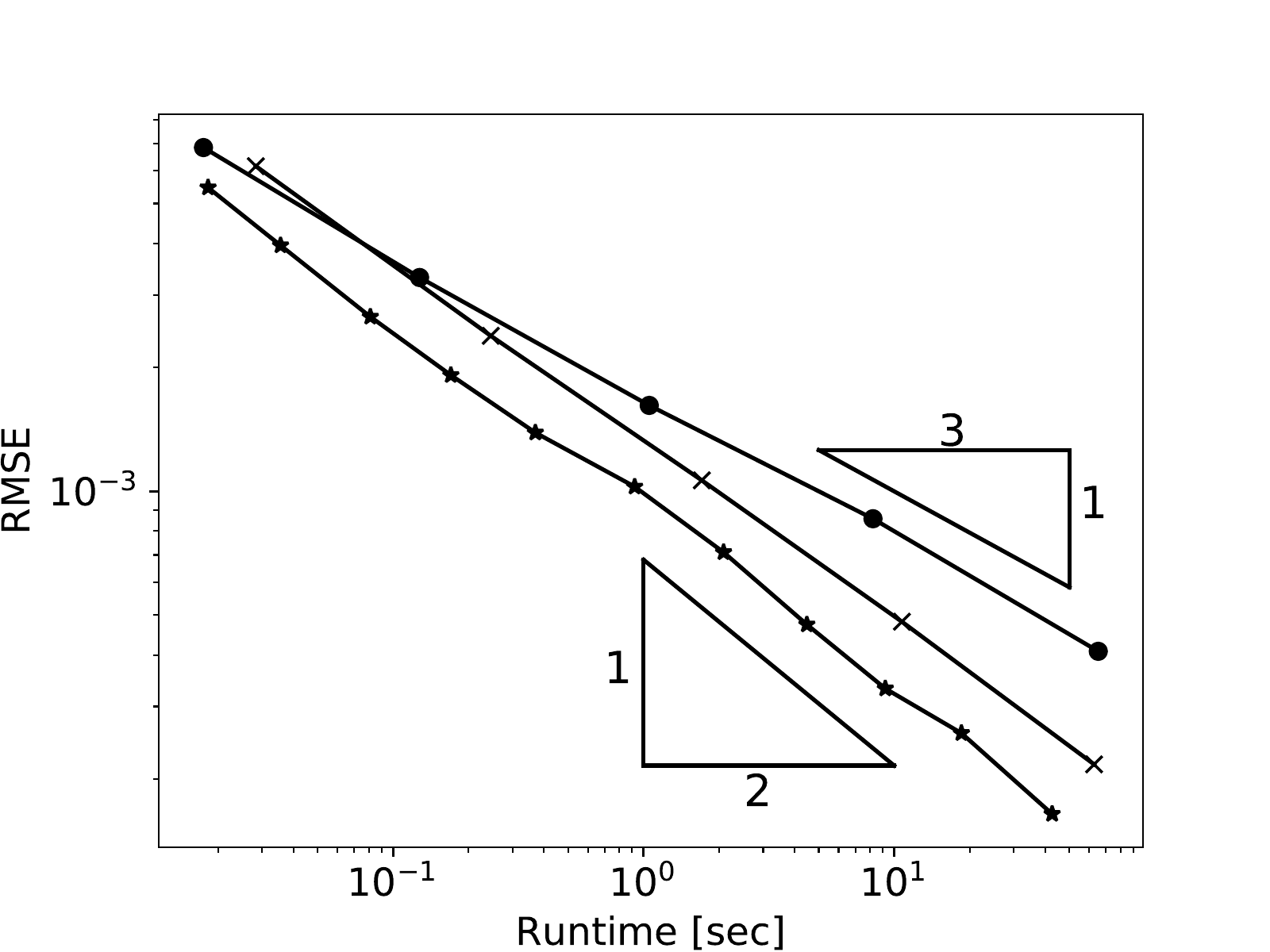}}
\includegraphics[width=0.47\textwidth]{{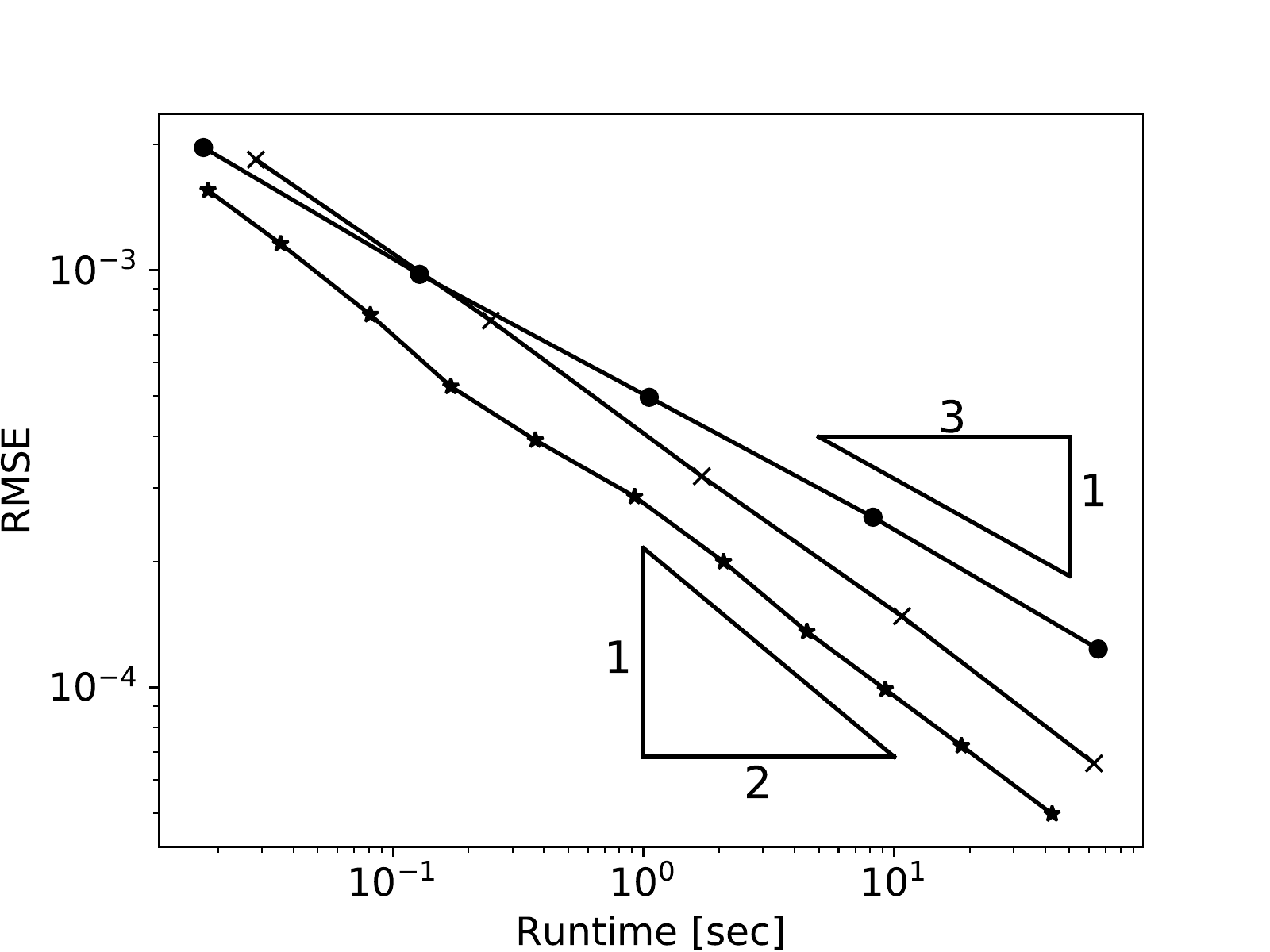}}\\
\includegraphics[width=0.47\textwidth]{{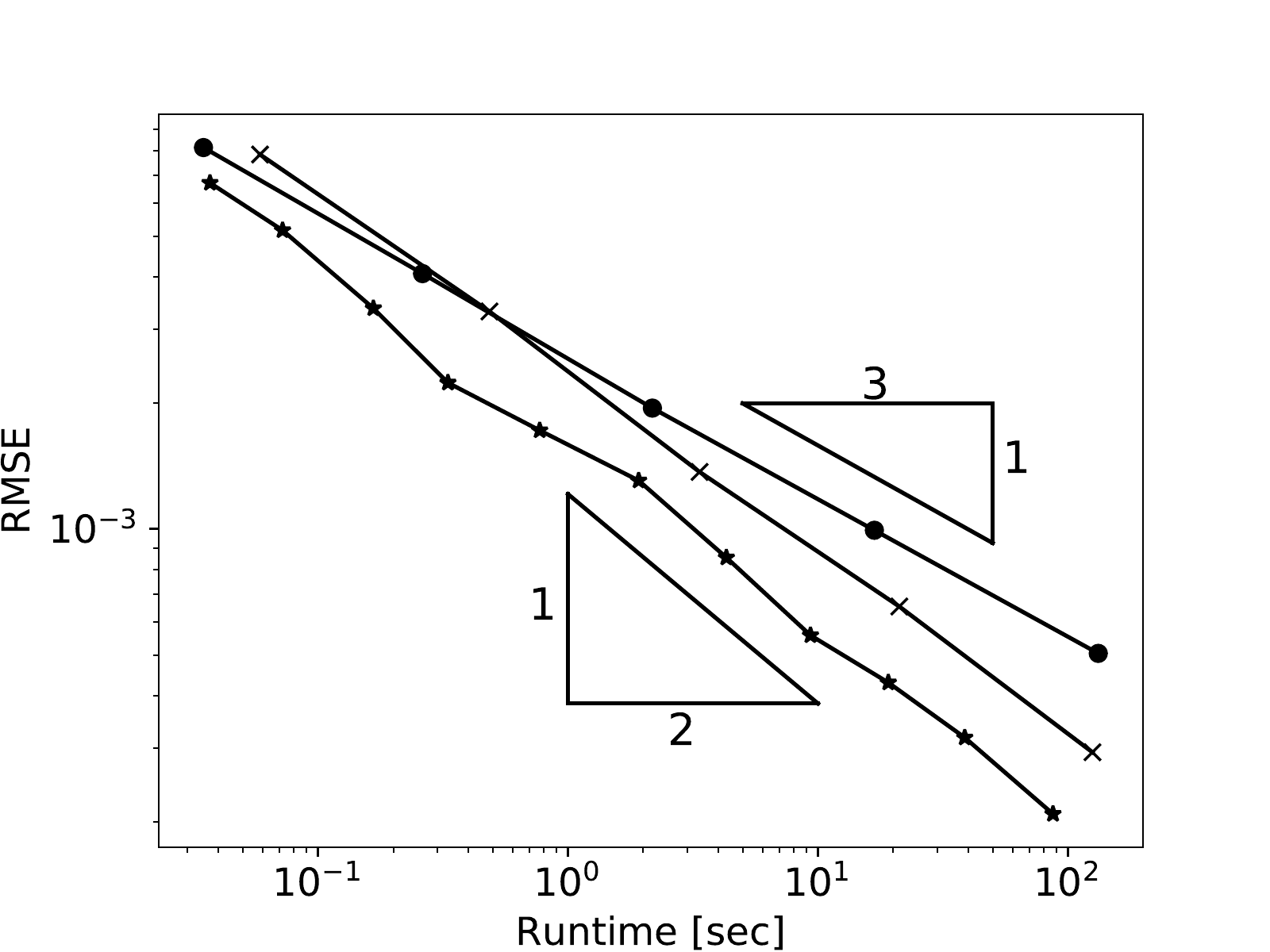}}
\includegraphics[width=0.47\textwidth]{{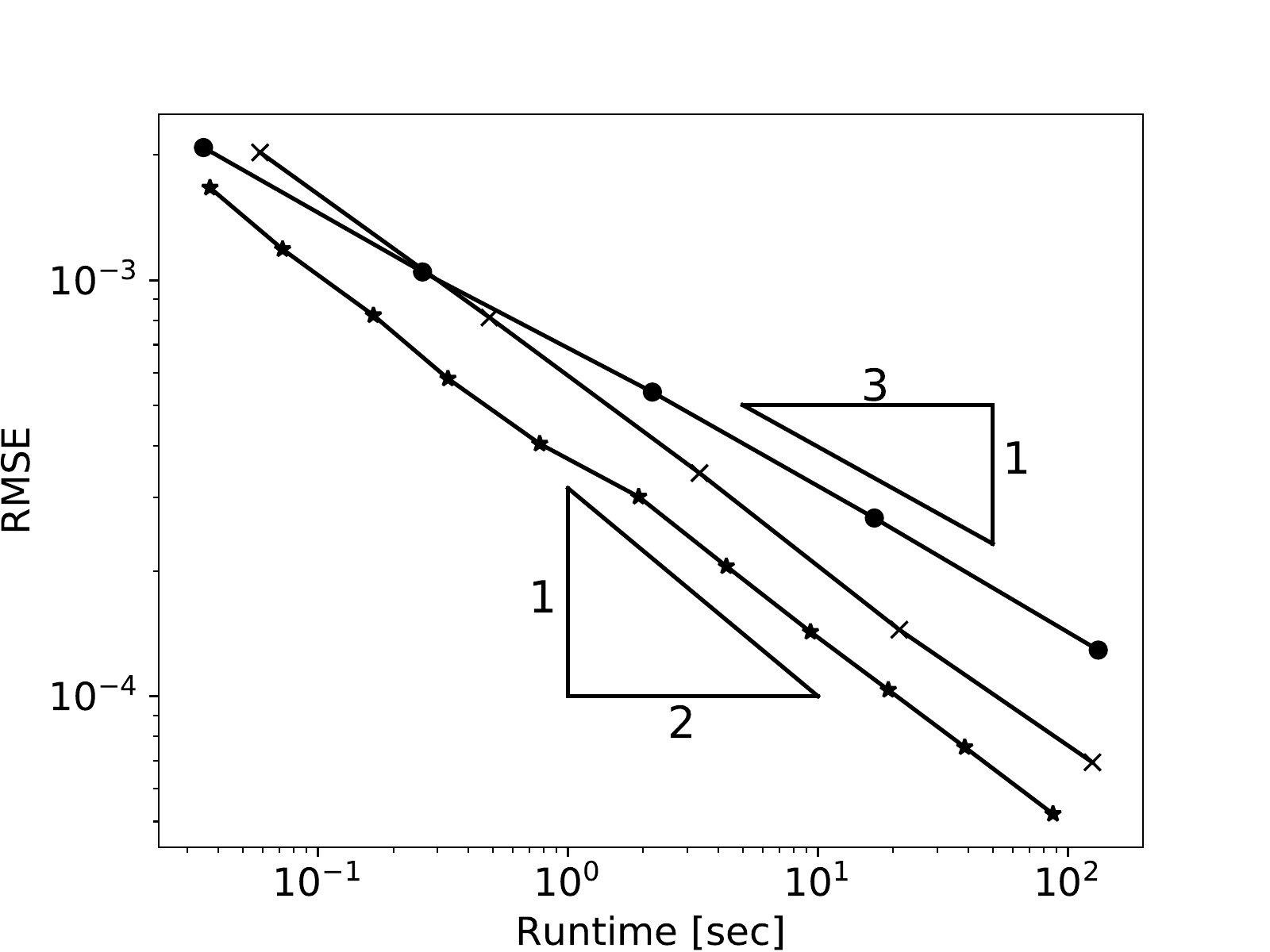}}
\caption{Top row: comparison of the runtime versus RMSE for the QoIs
  mean (left) and variance (right) over $\cN=10$ observation times for
  the problem in Section~\ref{ssec:dw}. The solid-crossed line
  represents MLEnKF, the solid-asterisk line represents "canonical"
  MLEnKF and the bottom reference triangle with the slope
  $\frac{1}{2}$, the solid-bulleted line represents EnKF and the upper
  reference triangle with the slope $\frac{1}{3}$.  Bottom row:
  similar plots over $\cN =20$ observation times }
	\label{fig:fig3Ex2}
\end{figure}

\appendix

\section{Mean-field EnKF}
\label{sec:MFEnKF}
The large-ensemble-limit of EnKF is referred to as the mean-field EnKF (MFEnKF).
In the mean-field limit, the EnKF Kalman gain becomes a deterministic
matrix. Consequently, there is no mixing between ensemble
particles and the particles become iid. To describe the dynamics
of MFEnKF, it thus suffices to consider a single particle.
For fixed dynamics resolution $\Psi^N$, let $\hat\vBar_{n}^N$ denote the updated state of a
mean-field particle at time $n$. The MFEnKF prediction and update dynamics
is given by 
\begin{equation}\label{mf:prediction}
\vBar_{n+1}^N = \Psi^N_n(\hat\vBar_{n}^N),
\end{equation}
and 
\[
\hat \vBar^N_{n+1}=(I-\bar K_{n+1}^N H)\vBar^N_{n+1}+\bar K_{n+1}^N{\tilde y}_{n+1},
\]
where 
\[
\bar K_{n+1}^N ={\bar C^N }_{n+1}H^\bT \prt{H{{\bar C^N}}_{n+1}H^\bT +\Gamma}^{-1} \\
\]
denotes the \emph{deterministic} mean-field Kalman gain,
\[
\bar C_{n+1}^N =\Ex{\prt{\vBar^N_{n+1}-\Ex{{\vBar}^N_{n+1}}}\prt{{\vBar}^N_{n+1}-\Ex{{\vBar}^N_{n+1}}}^\bT}
\]
is the mean-field prediction covariance, and 
\[
\yTilde{n+1}=y_{n+1} +\tilde{\eta}_{n+1}
\]
denotes the perturbed observation with 
$\{\tilde{\eta}_n\}_{n=1}^N$ being independent and $N(0,\Gamma)-$distributed random variables 
satisfying $\tilde{\eta}_j \perp u_k$ for all $j,k \ge 0$.

By Assumption~\ref{ass:psi} (i) with $|\kappa|_1 =0$, and using that $L_n^p(\Omega, \bR^d) \subset L^p(\Omega, \bR^d)$,
it follows straightforwardly by induction that 
$\vBar_n^N, \hat \vBar_n^N \in \cap_{p \ge 2} L^p(\Omega, \bR^d)$
uniformly in $N \in \bN \cup \{\infty\}$:
\begin{equation}\label{eq:inductionVBarBound}
\begin{split}
\hat \vBar_0^N \sim \Prob{u_0|Y_0} &\implies \vBar_1^N = \Psi^N_0(\hat \vBar_0^N)\in \cap_{p \ge 2} L^p(\Omega, \bR^d)\\
&\implies |\bar C_{1}^N|_2 < \infty
\implies  |\bar K_1|_2 < \infty\\
& \implies \|\hat \vBar^N_1 \|_p \le c_p (1+ |H|_2) \prt{\|\vBar^N_1\|_p  + \|\yTilde{1}\|_p} < \infty \quad \forall p \ge 2 \\
& \implies \ldots \implies \hat \vBar_n^N \in \cap_{p \ge 2} L^p(\Omega, \bR^d).
\end{split}
\end{equation}
This ensures the existence of the mean-field measure $\bar \mu_n^N$,
i.e., the distribution $\hat \vBar_n^N \sim \bar \mu_n^N$ satisfies
$\hat \vBar_n^N \in \cap_{p \ge 2} L^p(\Omega, \bR^d)$ for all $N\in \bN \cup \{\infty\}$ and $n \in \bN_0$.
For evaluating the expectation of a QoI $\varphi : \bR^d \to \bR$ wrt to
the mean-field measure $\bar \mu_n^N$,
we recall from~Notation~\ref{notation:not2} that 
\[
\bar \mu_n^N [\varphi] = \int_{\bR^d} \varphi(v)  \bar \mu_n^N(dv).
\]
For a subset of nonlinear dynamics $\Psi$ with additive Gaussian
noise, it has been shown that in the large-ensemble limit,
EnKF with perturbed observations
converges to the Kalman filtering distribution with the standard rate
$\cO(P^{-1/2})$, cf.~\cite{le2011large}.  That is, in $L^p$ for all QoI
$\varphi:\bR^d \to \bR$ with sufficient regularity $n \in \bN_0$, and
$p\ge 2$
\[
\norm{\mu_n^{\infty,P} [\varphi] - \bar \mu_n [\varphi]}_p \le c_{\varphi,p} P^{-1/2} \quad 
\text{for some} \quad C_{\varphi, \Psi^\infty}>0,
\]
where $\bar \mu_n := \bar \mu_n^\infty$.  The result was extended to a
subset of fully non-Gaussian models~\cite{law2016deterministic} with
numerical approximations of the dynamics in~\cite{hoel2016}, i.e.,
$\mu^{N,P}_n[\varphi] \to \bar \mu_n[\varphi]$ as $P,N \to \infty$
(see also Theorem~\ref{thm:enkfConv}).

\section{Theoretical proofs}
\label{sec:proofs}
\subsection{EnKF}
In this section, we present a collection of theoretical results for EnKF
that culminates with proof of Theorem~\ref{thm:enkfConv}. 

\subsubsection*{Notation}

The sample average of a QoI $\varphi \in \bF$ applied to an ensemble of $P$ identically distributed
particles $\{\vHat_{n,i}^{N,P}\}_{i=1}^P$ is defined as
\begin{equation*}\label{eq:sampleAverage}
E_P[\varphi(\vHat_{n}^{N,P})]:= \frac{1}{P} \sum_{i=1}^{P} \varphi(\vHat_{n,i}^{N,P}) \quad (= \mu_n^{N,P}[\varphi]),
\end{equation*}
and similarly
\begin{equation}\label{eq:sampleAverageParticle}
\begin{split}
E_P[\vHat_{n}^{N,P}] := \frac{1}{P} \sum_{i=1}^{P} \vHat_{n,i}^{N,P} \quad \mbox{  and  }\quad 
E_P[\prt{\vHat_{n}^{N,P}}\prt{\vHat_{n}^{N,P}}^\bT] := \frac{1}{P} \sum_{i=1}^{P} \prt{\vHat_{n,i}^{N,P}}\prt{\vHat_{n,i}^{N,P}}^\bT.
\end{split}
\end{equation}
\begin{lemma}[Difference between EnKF and MFEnKF Kalman gains]\label{lem:EnKFGainCont}
  If Assumption~\ref{ass:psi} holds, then 
  \begin{equation*}
	\begin{split}
	  \norm{K_n^{N,P}-\bar{K}_n^{N}}_p
          &\leq |\Gamma^{-1}|_2|H|_2\prt{1+2|\bar{K}_n^{N}H|_2}\norm{C_n^{N,P}-\bar{C}_n^{N}}_p
          \lesssim \norm{C_n^{N,P}-\bar{C}_n^{N}}_p
	\end{split}
  \end{equation*}
  for any $n \in \bN_0$ and $p\ge 2$.
\end{lemma}
\begin{proof}
The proof of the first inequality is analogous to \cite[Lemma 3.4]{hoel2016},
and the latter inequality follows from $\bar{v}^N_n \in \cap_{p\ge2}L^p(\Omega, \mathbb{R}^d)$, uniformly in $N\ge 1$, cf.~\eqref{eq:inductionVBarBound}.
\end{proof}

We next bound the  difference between  $C_n^{N,P} = \Cov [v_n^{N,P}]$ and $\bar{C}_n^{N} = \mathrm{Cov} [\bar{v}_n^{N}]$.
\begin{lemma}[Difference between EnKF and MFEnKF covariance matrices] \label{lem:EnKFCovErr}
	If Assumption~\ref{ass:psi} holds,
	then 
	\begin{equation*}
	\norm{C_n^{N,P}-\bar{C}_n^{N}}_p \lesssim \norm{v_n^{N,P}-\bar{v}_n^{N}}_{2p}+P^{-\frac{1}{2}}
	\end{equation*}
        for any	$n \in \bN_0$ and $p\ge2$.
\end{lemma}
\begin{proof}
In order to bound $\norm{C_n^{N,P}-\bar{C}_n^{N}}_p$, we introduce the auxiliary mean-field ensemble $\vBarHat_{n,1:P}^{N,P}= \{\vBarHat_{n,i}^{N,P}\}_{i=1}^P$ consisting of $P$ identically distributed
particles whose prediction/update dynamics is given by the mean-field Kalman gain, namely
\begin{equation}\label{eq:auxMFEnKF}
\begin{split}
  \bar v^{N,P}_{n+1,i} &= \Psi^N_n(\vBarHat_{n,i}^{N,P}),\\
  \vBarHat^{N,P}_{n+1,i} &= (I- \bar K^N_{n+1}H)\bar v^{N,P}_{n+1,i} + \bar K^N_{n+1} \tilde y_{n+1,i},
\end{split}
\end{equation}
with the initial condition $\vBarHat^{N,P}_{0,i} = \vHat_{0,i}^{N,P}$ for $i=1,...,P$.
The auxiliary ensemble satisfies the following two
crucial properties: 
$\bar v_{n}^{N,P} \stackrel{D}{=} \bar v_n$ and 
$\bar v_{n,i}^{N,P}$ has the same initial data, driving noise $W$
and perturbed observations as the $i$-th EnKF particle
$v_{n,i}^{N,P}$, cf.~Section~\ref{sec:EnKF}. In other words,
$\bar v_{n,i}^{N,P}$ is shadowing $v_{n,i}^{N,P}$.

Introducing the auxiliary ensemble covariance based on the mean-field prediction particles,
\begin{equation}\label{eq:mfenkfSampleCov}
  \cBarTilde_n^{N,P} := \Cov  [\bar{v}_n^{N,P}],
\end{equation}
the triangle inequality implies that 
\begin{equation}\label{covTrianIneq}
  \norm{C_n^{N,P}-\bar{C}_n^{N}}_p \leq \norm{C_n^{N,P}-\cBarTilde_n^{N,P}}_p +\norm{\cBarTilde_n^{N,P}-\bar{C}_n^{N}}_p.
\end{equation}
For the first term, recalling the notation~\eqref{eq:sampleAverageParticle}, we have
\begin{equation*}
  \begin{split}
    \norm{C_n^{N,P}-\cBarTilde_n^{N,P}}_p & \leq \norm{E_{P}[v_n^{N,P}\big((v_n^{N,P})^{\bT}-(\bar{v}_n^{N,P})^{\bT}\big)]}_p+\norm{E_{P}[\big(v_n^{N,P}-\bar{v}_n^{N,P}\big)(\bar{v}_n^{N,P})^{\bT}]}_p\\
	&+\norm{E_{P}[\bar{v}_n^{N,P}]E_{P}[(\bar{v}_n^{N,P})]^{\bT}-E_{P}[\bar{v}_n^{N,P}]E_{P}[(v_n^{N,P})]^{\bT}}_p\\
	&+\norm{E_{P}[\bar{v}_n^{N,P}]E_{P}[(v_n^{N,P})]^{\bT} -E_{P}[v_n^{N,P}]E_{P}[(v_n^{N,P})]^{\bT}}_p.
  \end{split}
\end{equation*}
Jensen's and H{\"o}lder's inequalities and $\bar{v}_n^{N,P} \in L^p(\Omega, \bR^d)$ yield that 
\begin{equation*}
  \begin{split}
    \norm{C_n^{N,P}-\cBarTilde_n^{N,P}}_p & \leq 2\norm{v_n^{N,P}}_{2p}\norm{v_n^{N,P}-\bar{v}_n^{N,P}}_{2p}+2\norm{\bar{v}_n^{N,P}}_{2p}\norm{v_n^{N,P}-\bar{v}_n^{N,P}}_{2p}\\
    &\leq 2\norm{v_n^{N,P}-\bar{v}_n^{N,P}}_{2p}^2+ 4\norm{\bar{v}_n^{N,P}}_{2p}\norm{v_n^{N,P}-\bar{v}_n^{N,P}}_{2p}\\
    &\lesssim \max\left(\norm{v_n^{N,P}-\bar{v}_n^{N,P}}_{2p}^2, \norm{v_n^{N,P}-\bar{v}_n^{N,P}}_{2p}\right).
  \end{split}
\end{equation*}
For the second term in~\eqref{covTrianIneq}, we have
\begin{equation*}
  \begin{split}
    \norm{\cBarTilde_n^{N,P}-\bar{C}_n^{N}}_p &\leq \norm{E_{P}[\bar{v}_n^{N,P}(\bar{v}_n^{N,P})^{\bT}]- \Ex{\bar{v}_n^{N}(\bar{v}_n^{N})^{\bT}}}_p\\
    &+\norm{E_{P}[\bar{v}_n^{N,P}]E_{P}[(\bar{v}_n^{N,P})^{\bT}]-\Ex{\bar{v}_n^{N}}
      \Ex{(\bar{v}_n^{N})^{\bT}}}_p\\
    &\lesssim P^{-\frac{1}{2}} \norm{\bar{v}_n^{N,P}(\bar{v}_n^{N,P})^{\bT}}_p\\
    & \lesssim P^{-\frac{1}{2}} \norm{\bar{v}_n^{N,P}}_{2p}^2,
  \end{split}
\end{equation*}
where the penultimate inequality is obtained by the Marcinkiewicz-Zygmund (M-Z) inequality~\cite[Theorem 5.2]{kwiatkowski2015convergence}.  
The property $\bar{v}_n^{N,P} \in L^p(\Omega,\bR^d)$, together with the above
inequalities imply that 
\begin{equation}\label{eq:firstBoundCovDiff}
  \norm{C_n^{N,P}-\bar{C}_n^{N}}_p \lesssim
  \max(\norm{v_n^{N,P}-\bar{v}_n^{N,P}}_{2p}, \norm{v_n^{N,P}-\bar{v}_n^{N,P}}_{2p}^2) +P^{-\frac{1}{2}}.
\end{equation}
Finally, by the proof of Lemma~\ref{lemma:EnKFdistEns}, it follows that $\norm{v_n^{N,P}-\bar{v}_n^{N,P}}_{2p} \lesssim 1$ and thus 
\begin{equation*}
  \norm{C_n^{N,P}-\bar{C}_n^{N}}_p \lesssim
  \norm{v_n^{N,P}-\bar{v}_n^{N,P}}_{2p} +P^{-\frac{1}{2}}.
\end{equation*}
\end{proof}

\begin{remark}\label{unbiasedCov}
  Lemma~\ref{lem:EnKFCovErr} straightforwardly extends to EnKF methods
  using an unbiased sample prediction covariance $\mathcal{C}_n^{N,P}$.
  By the following relationship between biased and unbiased sample
covariances
  \[
  \mathcal{C}_n^{N,P}:=\frac{P}{P-1}C_n^{N,P},
  \]
  we obtain that 
\[
\begin{split}
\norm{\mathcal{C}_n^{N,P}-\bar{C}_n^{N}}_p& \leq \norm{\mathcal{C}_n^{N,P}-C_n^{N,P}}_p + \norm{C_n^{N,P}-\bar{C}_n^{N}}_p  \\
&=\frac{P}{P-1} \norm{C_n^{N,P}-\bar{C}_n^{N}}_p+\frac{1}{P-1}\norm{\bar C_n^{N,P}}_p.
\end{split}
\]
The first term in the last equality is bounded by
Lemma~\ref{lem:EnKFCovErr}, and the bound for the second term follows
from $\bar{v}_n^{N,P} \in L^p(\Omega,\bR^d)$.  In the proof of
Theorem~\ref{thm:enkfConv}, the error contribution from the sample
covariance, be that a biased or an unbiased one, only enters through
Lemma~\ref{lem:EnKFCovErr}. This implies that the theorem holds and
the convergence rate is not affected when replacing biased with
unbiased sample covariances in the EnKF method.

Moreover, in the proof of Theorem~\ref{thm:mlenkfConv} for MLEnKF, the
error contribution from sample covariances only enter through
Corollary~\ref{covErr} and Lemma~\ref{lemma:dbCovErr}.  By a similar
argument as above, the rates of the said corollary and lemma are not
affected by replacing biased sample covariances by unbiased
ones and Theorem~\ref{thm:mlenkfConv} will also holds with the same
convergence rate. 
\end{remark}

\begin{lemma}[Distance between ensembles] \label{lemma:EnKFdistEns}
  If Assumption~\ref{ass:psi} holds, then 
  \begin{equation*}
    \max \prt{\norm{\vHat_{n}^{N,P}-\vBarHat_{n}^{N}}_p, \norm{v_{n}^{N,P}-\vBar_{n}^{N}}_p }\lesssim P^{-1/2} \quad \text{for any} \quad n \in \bN_0 \quad  \text{and} \quad p\ge 2.
  \end{equation*}
\end{lemma}

\begin{proof}
Since $\vHat_{0}^{N,P} = \vBarHat_{0}^{N}$, the first half of the statement holds for $n=0$.
Assume that for some $n\ge1$,
$\|\vHat_{n-1}^{N,P}-\vBarHat_{n-1}^{N}\|_p \lesssim P^{-1/2}$
holds for all $p\ge 2$. 
Assumption~\ref{ass:psi}(iii) then implies that 
\begin{equation}\label{v-vBar}
\norm{ v_n^{N,P}-\bar{v}_n^{N}}_p \lesssim
\norm{\vHat_{n-1}^{N,P}- \vBarHat_{n-1}^{N}}_p \lesssim P^{-\frac{1}{2}}.
\end{equation}
By H{\"o}lder's inequality and Lemma~\ref{lem:EnKFGainCont}, 
\begin{equation}\label{vHat-vBarHat}
  \begin{split}
    \norm{\vHat_{n}^{N,P}-\vBarHat_{n}^{N}}_p &\lesssim |I-\bar{K}_n^{N} H|_{2} \norm{ v_n^{N,P}-\bar{v}_n^{N}}_{p} \\
    &+\norm{C_n^{N,P}-\bar{C}_n^{N}}_{2p}\big(\norm{v_n^{N,P}-\bar{v}_n^{N}}_{2p}+\norm{\tilde{y}_n -H\bar{v}_n^{N}}_{2p}\big).
  \end{split}
\end{equation}
Since $\bar{v}_n^{N}, \tilde{y}_n^\ell \in \cap_{r\ge 2} L^r(\Omega)$, inequalities~\eqref{v-vBar} and~\eqref{eq:firstBoundCovDiff} imply that 
\begin{equation*}
  \begin{split}
    \norm{\vHat_{n}^{N,P}-\vBarHat_{n}^{N}}_p &\lesssim 
    ( |I-\bar{K}_n^{N} H|_{2} + \norm{\tilde{y}_n -H\bar{v}_n^{N}}_{2p}) P^{-1/2}   + P^{-1}
    \lesssim P^{-1/2}.
  \end{split}
\end{equation*}
The argument holds for any $p \ge 2$, and the proof follows by induction.
\end{proof}

Before proving the main convergence result for EnKF, we introduce the notation
$\bar{\mu}_n^{N,P}[\varphi]$ to denote the average of the QoI $\varphi$ over the empirical measure associated with the
auxiliary mean-field ensemble $\{\vBarHat_{n,i}^{N}\}_{i=1}^P$, cf.~\eqref{eq:auxMFEnKF},
and we recall that $\bar{\mu}_n^{N}[\varphi]$ denotes the evaluation of empirical measure on QoI $\varphi$
associated with the MFEnKF $\vBarHat_{n}^{N}$, cf.~Section~\ref{sec:MFEnKF}.

\begin{proof}[Proof of Theorem~\ref{thm:enkfConv}]
By the triangle inequality,
\begin{equation*}\label{EnKFtriangle}
  \norm{\mu_n^{N,P}[\varphi]-\bar{\mu}_n[\varphi]}_p\leq
  \norm{\mu_n^{N,P}[\varphi]-\bar{\mu}_n^{N,P}[\varphi]}_p+
  \norm{\bar{\mu}_n^{N,P}[\varphi]-\bar{\mu}_n^N[\varphi]}_p+\norm{\bar{\mu}_n^{N}[\varphi]-\bar{\mu}_n[\varphi]}_p.
\end{equation*}
By Assumption~\ref{ass:psi}(ii), $\varphi \in \bF \subset C^2_P(\bR^d,\bR)$,
and Lemma~\ref{lemma:EnKFdistEns} implies that 
\[
\norm{\mu_n^{N,P}[\varphi]-\bar{\mu}_n^{N,P}[\varphi]}_p
= \norm{\frac{1}{P} \sum_{i=1}^P\varphi(\hat{v}_{n,i}^{N,P})-\varphi(\vBarHat_{n,i}^{N})}_p
\leq c_{\varphi} \norm{\widehat{v}_{n}^{N,P}-\vBarHat_{n}^{N}}_{2p} \lesssim P^{-1/2}.
\]
For the second term, using the Marcinkiewicz--Zygmund (M-Z) inequality, that $\varphi(\vBarHat_{n}^{N}) \in \cap_{p\ge2}L^p(\Omega, \bR)$ and that
$\vBarHat_{n,1}^{N,P}, \ldots, \vBarHat_{n,P}^{N,P}$ are iid with $\vBarHat_{n,i}^{N,P} \stackrel{D}{=} \vBarHat_{n}^{N}$
imply that 
\[
\norm{\bar{\mu}_n^{N,P}[\varphi]-\bar{\mu}_n^N[\varphi]}_p=\norm{\sum_{i=1}^{P}\frac{\varphi(\vBarHat_{n,i}^{N})-\Ex{\varphi(\vBarHat_{n}^{N})}}{P}}_p
\lesssim P^{-1/2}.
\]
For the last term, we have that 
\[
\norm{\bar{\mu}_n^{N}[\varphi]-\bar{\mu}_n[\varphi]}_p=\abs{\Ex{\varphi(\vBarHat_{n}^{N})-\varphi(\vBarHat_{n})}},
\]
and it remains to prove by induction that the right-hand side is $\cO(N^{-\alpha})$.

Since $\vBarHat_{0}^{N} \stackrel{D}{=} \vBarHat_{0}$, it holds that $\abs{\Ex{\varphi(\vBarHat_{0}^{N})-\varphi(\vBarHat_{0})}} =0$. Assume that for some $k\ge 1$, there exists an observable dependent constant
$c_{\varphi}>0$ such that
\[
\abs{\Ex{\varphi(\vBarHat_{k-1}^{N})-\varphi(\vBarHat_{k-1})}} \le c_{\varphi} N^{-\alpha} \quad \forall \varphi \in \bF.
\]
Then, Assumption~\ref{ass:psi}(ii) implies there exists a $\tilde c_{\varphi} >0$ such that 
\begin{equation}\label{eq:ass1iiConsequence}
\abs{\Ex{\varphi(\vBar_{k}^{N})-\varphi(\vBar_{k})}} \le \tilde c_{\varphi} N^{-\alpha} \quad \forall \varphi \in \bF.
\end{equation}
In order to bound $ \abs{\Ex{\varphi(\vBarHat_{k}^{N})-\varphi(\vBarHat_{k})}} $, we first recall that
\[
\begin{split}
  \vBarHat_{k}^{N} &= (I- \bar{K}_k^NH)\vBar_{k}^{N} + \bar{K}^N_k y_n  + \bar{K}^N_k \tilde{\eta}_k\\
  \vBarHat_{k} &= (I- \bar{K}_kH)\vBar_{k} + \bar{K}_k y_n  + \bar{K}_k \tilde{\eta}_k,
\end{split}
\]
with $\tilde \eta_{k} \sim N(0,\Gamma)$  
and introduce the functions $\tilde \varphi^{N},\tilde \varphi \in \bF$ defined by
\[
\tilde \varphi^N(x) =
\frac{1}{\sqrt{ (2 \pi)^p |\mathrm{det}(\Gamma)|} }
\int_{\bR^p}\varphi\Big( (I- \bar{K}^N_kH)x + \bar{K}^N_k y_k+ \bar{K}^N_k z \Big) e^{-z^\bT \Gamma^{-1} z/2}\, dz
\]
and
\[
\tilde \varphi(x) =
\frac{1}{\sqrt{ (2 \pi)^p |\mathrm{det}(\Gamma)|} }
\int_{\bR^p}\varphi\Big( (I- \bar{K}_kH)x + \bar{K}_k y_k+ \bar{K}_k z \Big) e^{-z^\bT \Gamma^{-1} z/2}\, dz.
\]
It then follows by the mean-value theorem that 
\[
\begin{split}
  \abs{\Ex{\varphi(\vBarHat_{k}^{N})-\varphi(\vBarHat_{k})}} &\le
  \abs{\Ex{\tilde \varphi^N(\vBar_{k}^{N})-\tilde \varphi^N(\vBar_{k})}} +
  \abs{\Ex{\tilde \varphi^N(\vBar_{k})-\tilde \varphi(\vBar_{k})}}\\
  &\lesssim N^{-\alpha} +  |\bar{K}_k-\bar{K}^N_k|_2.
\end{split}
\]
From~\cite[Lemma 3.4]{hoel2016}, we have that 
\[
\bar{K}_k-\bar{K}^N_k = \bar{K}_k H(\bar{C}_{k}^{N}-\bar{C}_k)H^{\bT}(H\bar{C}_{k}^N H^{\bT}+\Gamma)^{-1}
+(\bar{C}_k-\bar{C}_k^{N})H^{\bT}(H\bar{C}_{k}^{N}H^{\bT}+\Gamma)^{-1},
\]
which, since $\Gamma$ is positive definite and thus $|(H\bar{C}_{k}^{N}H^{\bT}+\Gamma)^{-1}|_{2} \lesssim |\Gamma^{-1}|_2 < \infty$, implies that
\[
\begin{split}
|\bar{K}_k-\bar{K}^N_k|_2 &\lesssim |\bar{C}_{k}^{N}-\bar{C}_k|_2\\
& \le \abs{\Ex{\bar v^N_k (\bar v_k^N)^\bT - \bar v_k (\bar v_k)^\bT } }_2
+ \abs{\Ex{\bar v^N_k} \Ex{ (\bar v_k^N)^\bT} - \Ex{\bar v_k}\Ex{ (\bar v_k)^\bT} }_2\\
&\lesssim N^{-\alpha}.
\end{split}
\]
Since all monomials of degree $1$ and $2$ are contained in $\bF$,
the last inequality follows from~\eqref{eq:ass1iiConsequence}
and the equivalence of the Euclidean and Frobenius norms in any dimension $d<\infty$.
It holds by induction that for any $n\ge 0$,
\begin{equation*}\label{eq:weakConvRate}
| \mathbb{E}[\varphi(\vBarHat_{n}^{N})-\varphi(\vBarHat_{n})]| \lesssim  N^{-\alpha}.
\end{equation*}
\end{proof}

\subsection{MLEnKF}

In this section, we present a collection of theoretical results for MLEnKF,
including the proof of Theorem~\ref{thm:mlenkfConv}.

In order to obtain a connection between $\{\vHat_{n,i}^{\ell,\ff}\}_{i=1}^{P_\ell}$ and 
$\{\vHat_{n,i}^{\ell,\fc}\}_{i=1}^{P_\ell} = \{\vHat_{n,i}^{\ell,\fc_1}\}_{i=1}^{P_{\ell-1}} \cup  \{\vHat_{n,i}^{\ell,\fc_2}\}_{i=1}^{P_{\ell-1}}$
in the superindex $ \fc_j  \leftrightarrow \ff_j$, we introduce 
\[
\vHat_{n,i}^{\ell,\ff_1} := \vHat_{n,i}^{\ell,\ff}  \quad \text{and} \quad 
\vHat_{n,i}^{\ell,\ff_2} := \vHat_{n,i+P_{\ell-1}}^{\ell,\ff} \quad \text{for} \quad i=1,\ldots, P_{\ell-1},
\]
and
\[
\mu_{n}^{\ell,\ff_j}[\varphi]:= E_{P_{\ell-1}}[\varphi(\vHat_{n}^{\ell,\ff_j} )]
= \frac{1}{P_{\ell-1}} \sum_{i=1}^{P_{\ell-1}} \varphi(\vHat_{n,i}^{\ell,\ff_j}) \quad \text{for} \quad j=1,2.
\]  
The random variable $\vHat_{n}^{\ell,\ff_j}$
has the same driving noise and perturbed observations
as $\vHat_{n}^{\ell,\fc_j}$ and we note that 
\[
\begin{split}
\prt{\mu_{n}^{\ell,\ff} - \frac{\mu_{n}^{\ell,\fc_1} + \mu_{n}^{\ell,\fc_2}}{2}} [\varphi] &=
E_{P_\ell}[ \varphi(\vHat_{n}^{\ell,\ff}) - \varphi(\vHat_{n}^{\ell,\fc})]\\
&= \frac{1}{2}\sum_{j=1}^2 \prt{\mu_{n}^{\ell,\ff_j} - \mu_{n}^{\ell,\fc_j}}[\varphi].
\end{split}
\]

We further introduce the auxiliary mean-field MLEnKF ensemble
$\{\vBarHat^{\ell,\ff}_{n,1:P_\ell},\vBarHat^{\ell,\fc}_{n,1:P_\ell}\}_{\ell=0}^L$
where the dynamics for
$\vBarHat^{\ell,\ff}_{n,i}$ and $\vBarHat^{\ell,\fc}_{n,i}$ are coupled
through driving-noise-sharing dynamics 
\begin{equation}\label{eq:mfmlenkfDynamics}
\vBar^{\ell,\ff}_{n+1,i} = \Psi^{N_\ell}_n(\vBarHat^{\ell,\ff}_{n,i}), \qquad 
\vBar^{\ell,\fc}_{n+1,i} = \Psi^{N_\ell}_n(\vBarHat^{\ell,\fc}_{n,i})
\end{equation}
and perturbed-observation-sharing updates
\begin{equation}\label{eq:mfmlenkfUpdates}
\vBarHat^{\ell,\ff}_{n+1,i} = (I- \bar{K}^{\ell,\ff}_n H)\vBar^{\ell,\ff}_{n,i} + \bar{K}^{\ell,\ff}_n \tilde y_{n,i}^\ell, \qquad 
\vBarHat^{\ell,\fc}_{n+1,i} = (I- \bar{K}^{\ell,\fc}_n H)\vBar^{\ell,\fc}_{n,i} + \bar{K}^{\ell,\fc}_n \tilde y_{n,i}^\ell,
\end{equation}
where $\bar{K}^{\ell,\ff}_n :=\bar{K}^{N_\ell}_n$ and $\bar{K}^{\ell,\fc}_n :=\bar{K}^{N_{\ell-1}}_n$,
cf.~Section~\ref{sec:MFEnKF},
and with initial conditions equal identical to MLEnKF:
\begin{equation*}\label{eq:initalAuxML}
\begin{split}
 \vBarHat^{0,\ff}_{0,i} &:= \vHat_{0,i}^{\ell,\ff}\\
  \vBarHat^{\ell,\ff}_{0,i} = \vBarHat^{\ell,\fc}_{0,i} &:= \vHat_{0,i}^{\ell,\ff} = \vHat_{0,i}^{\ell,\fc} \quad   \ell \ge 1.
\end{split}
\end{equation*}
The particles $(\vBarHat^{\ell,\ff}_{n,i}, \vBarHat^{\ell,\fc}_{n,i})$ are thus
coupled through sharing the same initial condition, driving noise $W$, and perturbed observations.
The particle pair is also coupled to the MLEnKF pair $(\vHat_{n,i}^{\ell,\ff},\vHat_{n,i}^{\ell,\fc})$
in all the same ways. In other words, $(\vBarHat^{\ell,\ff}_{n,i}, \vBarHat^{\ell,\fc}_{n,i})$ is shadowing $(\vHat_{n,i}^{\ell,\ff},\vHat_{n,i}^{\ell,\fc})$.
The ensemble
$\vBarHat^{\ell,\ff}_{n,1:P_\ell} = \vBarHat^{\ell,\ff_1}_{n,1:P_{\ell-1}} \cup \vBarHat^{\ell,\ff_2}_{n,1:P_{\ell-1}}$
induces an empirical measure $\bar \mu_n^{\ell,\ff} = (\bar \mu_n^{\ell,\ff_1}+\bar \mu_n^{\ell,\ff_2})/2$ and
$\vBarHat^{\ell,\fc}_{n,1:P_\ell} = \vBarHat^{\ell,\fc_1}_{n,1:P_{\ell-1}} \cup \vBarHat^{\ell,\fc_2}_{n,1:P_{\ell-1}}$
induces $\bar \mu_n^{\ell,\fc} = (\bar \mu_n^{\ell,\fc_1}+\bar \mu_n^{\ell,\fc_2})/2$, with the convention that
$\bar \mu_n^{0,\fc} = \bar \mu_n^{0,\fc_1}= \bar \mu_n^{0,\fc_2} =0$ for all $n\ge 0$. 
Employing sequences $\{M_\ell\},\{N_\ell\} \subset \bN$ and $L \in \bN$ with exactly the same values
as those used for the MLEnKF estimator~\eqref{eq:mlEstimator} we seek to shadow,
we define the auxiliary mean-field MLEnKF estimator by
\begin{equation}\label{eq:MFMLEnKF_est}
\bar \mu^{\ML}_n[\varphi] = \sum_{\ell=0}^L \frac{1}{M_\ell} \sum_{m=1}^{M_\ell}
\prt{\bar \mu_{n}^{\ell,\ff,m} - \bar \mu_{n}^{\ell,\fc,m}}[\varphi] 
\end{equation}
where for $m=1,\ldots, M_\ell$,
$\prt{\bar \mu_{n}^{\ell,\ff,m} - \bar \mu_{n}^{\ell,\fc,m}}[\varphi]$
are iid realizations that are coupled to/shadowing
$\prt{\mu_{n}^{\ell,\ff,m} - \mu_{n}^{\ell,\fc,m}}[\varphi]$ by
sharing the same underlying randomness (initial conditions,
driving noise and perturbed observations).
And, consequently, $\bar \mu^{\ML}_n[\varphi]$ is
shadowing $\mu^{\ML}_n[\varphi]$.

\begin{proof}[Proof of Theorem~\ref{thm:mlenkfConv}]
  
By the triangle inequality,
\begin{equation}\label{3term}
\begin{split}
  \norm{\mu_n^{\ML}[\varphi]- \bar{\mu}_n[\varphi]}_p &\leq \norm{\mu_n^{\ML}[\varphi]-\bar{\mu}_n^{\ML}[\varphi]}_p+  \norm{\bar{\mu}_n^{\ML}[\varphi]-\bar{\mu}_n^{N_L}[\varphi]}_p\\
  &+\norm{\bar{\mu}_n^{N_L}[\varphi]-\bar{\mu}_n[\varphi]}_p.
\end{split}
\end{equation}
By~\eqref{eq:mlEstimator} and~\eqref{eq:MFMLEnKF_est}, $\Delta \tilde \mu^\ell_n[\varphi] := \Ex{\prt{\bar \mu_{n}^{\ell,\ff} - \mu_{n}^{\ell,\ff} + \mu_{n}^{\ell,\fc} - \bar \mu_{n}^{\ell,\fc} }[\varphi]}$ and the M-Z inequality,
\[
\begin{split}
\|\mu_n^{\ML}[\varphi]&-\bar{\mu}_n^{\ML}[\varphi]\|_p \le  \norm{\sum_{\ell=0}^L \Delta \tilde \mu^\ell_n[\varphi]}_p\\
& +\norm{ \sum_{\ell=0}^L \sum_{m=1}^{M_\ell}
  \frac{\prt{\bar \mu_{n}^{\ell,\ff,m} - \mu_{n}^{\ell,\ff,m} + \mu_{n}^{\ell,\fc,m} - \bar \mu_{n}^{\ell,\fc,m} }[\varphi] - \Delta \tilde \mu^\ell_n[\varphi]}{M_\ell} }_p \\
&\lesssim \norm{\Ex{\prt{\bar \mu_{n}^{L,\ff} - \mu_{n}^{L,\ff} }[\varphi]}}_p + \sum_{\ell=0}^L M_\ell^{-1/2} \norm{\prt{\bar \mu_{n}^{\ell,\ff} - \mu_{n}^{\ell,\ff} + \mu_{n}^{\ell,\fc} - \bar \mu_{n}^{\ell,\fc} }[\varphi]}_p,
\end{split}
\]
where we used that $\Ex{\mu_{n}^{\ell,\fc}} = \Ex{\mu_{n}^{\ell-1,\ff}}$ and $\Ex{\bar \mu_{n}^{\ell,\fc}} = \Ex{\bar \mu_{n}^{\ell-1,\ff}}$
for $\ell\ge1$ in the last inequality.
The $L^p(\Omega)$-convergence $\mu_{n}^{L,\ff}[\varphi] \to \bar{\mu}_{n}^{L,\ff}[\varphi]$ as $L\to \infty$, cf.~Theorem~\ref{thm:enkfConv},
further implies that
\[
\prt{\bar \mu_{n}^{L,\ff} - \mu_{n}^{L,\ff} }[\varphi] =
-\sum_{\ell=L+1}^\infty \Ex{\prt{\bar \mu_{n}^{\ell,\ff} - \mu_{n}^{\ell,\ff} + \mu_{n}^{\ell,\fc} - \bar \mu_{n}^{\ell,\fc} }[\varphi]},
\]
and Jensen's inequality and Lemma~\ref{lemma:dbDiffphi} yield
\[
\begin{split}
\|\mu_n^{\ML}[\varphi]-\bar{\mu}_n^{\ML}[\varphi]\|_p &\lesssim
\sum_{\ell=0}^L M_\ell^{-1/2} \norm{\prt{\bar \mu_{n}^{\ell,\ff} - \mu_{n}^{\ell,\ff} + \mu_{n}^{\ell,\fc} - \bar \mu_{n}^{\ell,\fc} }[\varphi]}_p\\
&+
\sum_{\ell=L+1}^\infty \norm{\prt{\bar \mu_{n}^{\ell,\ff} - \mu_{n}^{\ell,\ff} + \mu_{n}^{\ell,\fc} - \bar \mu_{n}^{\ell,\fc} }[\varphi]}_p\\
&\lesssim N_{L}^{-\beta/2}P_L^{-1/2}+P_L^{-1} + \sum_{\ell=0}^L M_\ell^{-1/2}(N_\ell^{-\beta/2}P_\ell^{-1/2}+P_\ell^{-1}).
\end{split}
\]

For the second term in~\eqref{3term}, note first that
\[
\Ex{ \bar \mu_n^{N_L}[\varphi] } = \Ex{ \bar \mu_n^{L,\ff}[\varphi] }  = \sum_{\ell=0}^L \Ex{ (\bar \mu_n^{\ell,\ff}- \bar \mu_n^{\ell,\fc})[\varphi] }
= \sum_{\ell=0}^L \Ex{\varphi(\vBarHat^{\ell,\ff}_n) - \varphi(\vBarHat^{\ell,\fc}_n)}.
\]
By applying the M-Z inequality twice (first in $M_\ell$ and thereafter in $P_\ell$) and Lemma~\ref{lemma:distMFens},
\begin{align*}
\begin{split}
  \norm{\bar{\mu}_n^{\ML}[\varphi]-\bar{\mu}_n^{N_L}[\varphi]}_p &\leq \sum_{\ell=0}^{L}
  \norm{\sum_{m=1}^{M_\ell}\frac{ (\bar{\mu}_{n}^{\ell,\ff,m }- \bar{\mu}_n^{\ell,\fc,m})[\varphi]
        -\Ex{ (\bar \mu_n^{\ell,\ff}- \bar \mu_n^{\ell,\fc})[\varphi] }}{M_\ell} }_p \\
&\lesssim \sum_{\ell=0}^{L}M_{\ell}^{-1/2} \norm{E_{P_\ell}[\varphi(\vBarHat^{\ell,\ff}_n) - \varphi(\vBarHat^{\ell,\fc}_n) -\Ex{\varphi(\vBarHat^{\ell,\ff}_n) - \varphi(\vBarHat^{\ell,\fc}_n)} }_p\\
& \lesssim \sum_{\ell=0}^{L} M_{\ell}^{-1/2} P_{\ell}^{-1/2} \norm{\varphi(\vBarHat_{n}^{\ell,\ff})-\varphi(\vBarHat_{n}^{\ell,\fc})}_p\\
&\lesssim \sum_{\ell=0}^{L} M_{\ell}^{-1/2} P_{\ell}^{-1/2}N_\ell^{-\beta/2}.
\end{split}
\end{align*}

For the third term in~\eqref{3term}, it follows by~\eqref{eq:ass1iiConsequence} that  
\begin{equation*}
  \begin{split}
    \norm{\bar{\mu}_n^{N_L}[\varphi]-\bar{\mu}_n[\varphi]}_p
    = |\Ex{\varphi(\vBarHat_{n}^{L,\ff})-\varphi(\vBarHat_{n})}| \lesssim N_L^{-\alpha}.
  \end{split}
\end{equation*}
\end{proof}	

\begin{lemma}[Continuity of mean-field Kalman gains]\label{lemma:MFGainCont}
It holds that 
\begin{equation*}
\begin{split}
|\bar{K}_{n}^{\ell, \fc}-\bar{K}_n^{\ell, \ff}|_2&\leq |\Gamma^{-1}|_2|H|_2\prt{1+2|\bar{K}_n^{\ell, \ff}H|_2}|\bar{C}_{n}^{\ell, \fc}-\bar{C}_n^{\ell, \ff}|_2\lesssim |\bar{C}_{n}^{\ell,  \fc}-\bar{C}_n^{\ell, \ff}|_2.
\end{split}
\end{equation*}
\end{lemma}
\begin{proof}
The proof is analogous to~\cite[Lemma~3.4]{hoel2016}. 
\end{proof}

\begin{lemma}[Continuity of mean-field covariance matrices] \label{lemma:MFcovErr}
If Assumptions~\ref{ass:psi} and~\ref{ass:psi2} hold,
then for any triplet of sequences 
$\{M_\ell\}, \{N_\ell\}, \{P_\ell\} \subset \bN$
described in Section~\ref{sec:mlenkfIntro}, $L\ge 0$ and $n\ge 1$
it holds that
\begin{equation*}
    |\bar{C}_n^{\ell,\fc}-\bar{C}_n^{\ell,\ff}|_2 \lesssim \norm{\bar{v}_n^{\ell,\fc}-\bar{v}_n^{\ell,\ff}}_{4}.
\end{equation*}
\end{lemma}
\begin{proof}
Using that $\bar{v}_n^{\ell,\fc},\bar{v}_n^{\ell,\ff} \in L^{4}(\Omega, \bR^d)$, H{\"o}lder's and Jensen's inequalities yield that
\begin{equation*}
  \begin{split}
    \abs{\bar{C}_n^{\ell,\fc}-\bar{C}_n^{\ell,\ff}}_2 &\leq
    \abs{\Ex{\bar{v}_n^{\ell,\fc}(\bar{v}_n^{\ell,\fc})^{\bT}-\bar{v}_n^{\ell,\ff}(\bar{v}_n^{\ell,\ff})^{\bT}}}_2
    + \abs{\Ex{\bar{v}_n^{\ell,\fc}}\Ex{(\bar{v}_n^{\ell,\fc})^{\bT}}
      -\Ex{\bar{v}_n^{\ell,\ff}}\Ex{(\bar{v}_n^{\ell,\ff})^{\bT}}}_2\\
		&\lesssim \norm{\bar{v}_n^{\ell,\fc}-\bar{v}_n^{\ell,\ff}}_{4}.
		\end{split}
		\end{equation*}
\end{proof}

In the remaining part of this section, we will assume that 
Assumptions~\ref{ass:psi} and~\ref{ass:psi2} hold,
and that the sequences $\{N_\ell\}, \{P_\ell\} \subset \bN$
satisfy the constraints given in Section~\ref{sec:mlenkfIntro}
(namely, $P_\ell = 2P_{\ell-1}$ and $\{N_\ell\}$
exponentially increasing).

\begin{lemma}[Stability of mean-field particles]\label{lemma:distMFens}
For any $n\ge 0$ and $p\ge 2$, it holds that 
\begin{equation*}
  \begin{split}
    \max \prt{\norm{\vBarHat_n^{\ell,\ff}-\vBarHat_n^{\ell,\fc}},\norm{\bar{v}_{n+1}^{\ell,\ff}-\bar{v}_{n+1}^{\ell,\fc}}_p}\lesssim N_\ell^{-\beta/2},
  \end{split}
\end{equation*}
for multilevel mean-field prediction and update particles defined as in~\eqref{eq:mfmlenkfDynamics} and ~\eqref{eq:mfmlenkfUpdates}.
\end{lemma}
\begin{proof}
Since $\vBarHat_0^{\ell,\ff} = \vBarHat_0^{\ell,\fc}$, we may assume that for some $n\ge1$,
\begin{equation*}
   \norm{\vBarHat_{n-1}^{\ell,\ff}-\vBarHat_{n-1}^{\ell,\fc}}_p \lesssim N_\ell^{-\beta/2}.
\end{equation*}
Assumption~\ref{ass:psi}(iii) and Assumption~\ref{ass:psi2}(iii) imply that 
\begin{equation*}
  \begin{split}
    \norm{\bar{v}_n^{\ell,\ff}-\bar{v}_n^{\ell,\fc}}_p &\leq \norm{\Psi^{N_\ell}_{n-1}(\vBarHat_{n-1}^{\ell,\ff})-\Psi^{N_{\ell-1}}_{n-1}(\vBarHat_{n-1}^{\ell,\ff})}_p+\norm{\Psi^{N_{\ell-1}}_{n-1}(\vBarHat_{n-1}^{\ell,\ff})-\Psi^{N_{\ell-1}}_{n-1}(\vBarHat_{n-1}^{\ell,\fc})}_p\\
    &\lesssim N_\ell^{-\beta/2},
    \end{split}
\end{equation*}
and Lemmas~\ref{lemma:MFGainCont} and~\ref{lemma:MFcovErr} and 
$|\bar{v}_n^{\ell,\ff}|, |\tilde{y}_n^\ell| \in \cap_{r\ge2} L^r(\Omega)$ yield that
\begin{equation*}
\begin{split}
   \norm{\vBarHat_{n}^{\ell,\ff}-\vBarHat_{n}^{\ell,\fc}}_p&\leq \abs{I-\bar{K}_n^{\ell,\fc}H}_{2}\norm{\bar{v}_n^{\ell,\ff}-\bar{v}_n^{\ell,\fc}}_{p}+\abs{\bar{K}_n^{\ell,\fc}-\bar{K}_n^{\ell,\ff}}_{2}\norm{H\bar{v}_n^{\ell,\ff}+\tilde{y}_n^\ell}_{p}\lesssim N_\ell^{-\beta/2}.
    \end{split}
\end{equation*}
The statement holds by induction.
\end{proof}

\begin{corollary}[Continuity of mean-field and EnKF covariance matrices] \label{covErr}
For any $n\ge 0$ and $p\ge 2$, it holds that
\begin{equation*}
 \begin{split}
    \norm{C_n^{\ell,\ff}-\bar{C}_n^{\ell,\ff}}_p &\lesssim \norm{v_n^{\ell,\ff}-\bar{v}_n^{\ell,\ff}}_{2p}+P_\ell^{-1/2},\\
        \norm{C_n^{\ell,\fc_k}-\bar{C}_n^{\ell,\fc}}_p &\lesssim \norm{v_n^{\ell,\fc}-\bar{v}_n^{\ell,\fc}}_{2p}+P_\ell^{-1/2}, \quad k=1,2,
     \end{split}
\end{equation*}
for multilevel prediction particles defined as in~\eqref{ml:prediction} and ~\eqref{eq:mfmlenkfDynamics}.
\end{corollary}
\begin{proof}
  Since $C_n^{\ell,\ff} = C_n^{N_\ell,P_\ell}$, $\bar{C}_n^{\ell,\ff} = \bar{C}_n^{N_\ell}$, $C_n^{\ell,\fc_k} \stackrel{D}{=} C_n^{N_{\ell-1},P_{\ell-1}}$ and $\bar{C}_n^{\ell,\fc} = \bar{C}_n^{N_{\ell-1}}$,
  the result follows from Lemma~\ref{lem:EnKFCovErr}. 
\end{proof}

\begin{corollary}[Distance between ensembles II] \label{distEns}
For any $n\ge0$ and $p\ge 2$, the following asymptotic inequality holds
  \begin{equation*}
\max \prt{\norm{\vHat_{n}^{\ell,\ff}-\vBarHat_{n}^{\ell,\ff}}_p,\norm{\vHat_{n}^{\ell,\fc}-\vBarHat_{n}^{\ell,\fc}}_p} \lesssim P_\ell^{-1/2}.
\end{equation*}
for multilevel update particles defined as in~\eqref{ml:update} and ~\eqref{eq:mfmlenkfUpdates}.
\end{corollary}
\begin{proof}
  Since $\vHat_{n}^{\ell,\ff} -\vBarHat_{n}^{\ell,\ff} \stackrel{D}{=} \vHat_{n}^{N_\ell,P_\ell}- \vBarHat_{n}^{N_\ell}$ and $\vHat_{n}^{\ell,\fc} -\vBarHat_{n}^{\ell,\fc} \stackrel{D}{=} \vHat_{n}^{N_{\ell-1},P_\ell}- \vBarHat_{n}^{N_{\ell-1}}$,
  the result follows from Lemma~\ref{lemma:EnKFdistEns}.
\end{proof}

\begin{lemma}[Continuity of Kalman gain double differences]\label{DoubleGainCont}
For any $n\ge 0$ and $p\ge 2$, it holds that
  \begin{equation}\label{dbKG}
    \begin{split}
      \norm{K_n^{\ell, \ff}-\frac{K_{n}^{\ell,\fc_1}+K_{n}^{\ell,\fc_2}}{2}-(\bar{K}_n^{\ell,\ff}-\bar{K}_n^{\ell,\fc})}_p&\lesssim \norm{C_n^{\ell,\ff}-\frac{C_{n}^{\ell,\fc_1}+C_{n}^{\ell,\fc_2}}{2}-(\bar{C}_n^{\ell,\ff}-\bar{C}_n^{\ell,\fc})}_p\\
	&\quad +P_{\ell}^{-1/2}N_{\ell}^{-\beta/2}+P_{\ell}^{-1}.
    \end{split}
  \end{equation}
\end{lemma}

\begin{proof}From the proof of~\cite[Lemma 3.4]{hoel2016}, one may deduce that
	\begin{equation*}
	\begin{split}
	K_n^{\ell, \ff}-\bar{K}_n^{\ell,\ff}=&K_n^{\ell, \ff}H(\bar{C}_{n}^{\ell,\ff}-C_n^{\ell,\ff})H^{\bT}(H\bar{C}_{n}^{\ell,\ff}H^{\bT}+\Gamma)^{-1}+(C_n^{\ell,\ff}-\bar{C}_n^{\ell,\ff})H^{\bT}(H\bar{C}_{n}^{\ell,\ff}H^{\bT}+\Gamma)^{-1},\\
	K_{n}^{\ell,\fc_j}-\bar{K}_n^{\ell,\fc}=&K_{n}^{\ell,\fc_j}H(\bar{C}_{n}^{\ell,\fc}-C_{n}^{\ell,\fc_j})H^{\bT}(H\bar{C}_{n}^{\ell,\fc}H^{\bT}+\Gamma)^{-1}+(C_{n}^{\ell,\fc_j}-\bar{C}_n^{\ell,\fc})H^{\bT}(H\bar{C}_{n}^{\ell,\fc}H^{\bT}+\Gamma)^{-1},
	\end{split}
	\end{equation*}
	and 
	\begin{equation*}\label{KGformula}
	\begin{split}
	(H\bar{C}_{n}^{\ell,\ff}H^{\bT}+\Gamma)^{-1}=(H\bar{C}_{n}^{\ell,\fc}H^{\bT}+\Gamma)^{-1}+(H\bar{C}_{n}^{\ell,\fc}H^{\bT}+\Gamma)^{-1}(\bar{C}_n^{\ell,\fc}-\bar{C}_n^{\ell,\ff})H^{\bT}(H\bar{C}_{n}^{\ell,\ff}H^{\bT}+\Gamma)^{-1}.
	\end{split}
	\end{equation*}
	The above equations imply that 
	\begin{equation*}
	\begin{split}
	K_n^{\ell, \ff}-\bar{K}_n^{\ell,\ff}=&K_n^{\ell, \ff}H(\bar{C}_{n}^{\ell,\ff}-C_n^{\ell,\ff})H^{\bT}(H\bar{C}_{n}^{\ell,\fc}H^{\bT}+\Gamma)^{-1}\\
	&+K_n^{\ell, \ff}H(\bar{C}_{n}^{\ell,\ff}-C_n^{\ell,\ff})H^{\bT}(H\bar{C}_{n}^{\ell,\fc}H^{\bT}+\Gamma)^{-1}(\bar{C}_n^{\ell,\fc}-\bar{C}_n^{\ell,\ff})H^{\bT}(H\bar{C}_{n}^{\ell,\ff}H^{\bT}+\Gamma)^{-1}\\
	&+(C_n^{\ell,\ff}-\bar{C}_n^{\ell,\ff})H^{\bT}(H\bar{C}_{n}^{\ell,\ff}H^{\bT}+\Gamma)^{-1},
	\end{split}
	\end{equation*}
	and
	\begin{equation*}
	\begin{split}
	&K_n^{\ell, \ff}-\frac{K_{n}^{\ell,\fc_1}+K_{n}^{\ell,\fc_2}}{2}-(\bar{K}_n^{\ell,\ff}-\bar{K}_n^{\ell,\fc})=K_n^{\ell, \ff}-\bar{K}_n^{\ell,\ff}-\frac{1}{2}\bigg( K_{n}^{\ell,\fc_1} - \bar{K}_n^{\ell,\fc}+ K_{n}^{\ell,\fc_2} - \bar{K}_n^{\ell,\fc}\bigg)\\
	&=(I-K_n^{\ell, \ff}H)\bigg(C_n^{\ell,\ff}-\frac{C_{n}^{\ell,\fc_1}+C_{n}^{\ell,\fc_2}}{2}-\bar{C}_{n}^{\ell,\ff}+\bar{C}_{n}^{\ell,\fc}\bigg)H^{\bT}(H\bar{C}_{n}^{\ell,\fc}H^{\bT}+\Gamma)^{-1}\\
	&+(I-K_n^{\ell, \ff}H)(\bar{C}_n^{\ell,\ff}-C_n^{\ell,\ff})H^{\bT}(H\bar{C}_{n}^{\ell,\fc}H^{\bT}+\Gamma)^{-1}H(\bar{C}_n^{\ell,\fc}-\bar{C}_n^{\ell,\ff})H^{\bT}(H\bar{C}_{n}^{\ell,\ff}H^{\bT}+\Gamma)^{-1}\\
	&-(K_{n}^{\ell,\fc_1}-K_n^{\ell,\ff})H(\bar{C}_n^{\ell,\fc}-\frac{C_{n}^{\ell,\fc_1}+C_{n}^{\ell,\fc_2}}{2})H^{\bT}(H\bar{C}_{n}^{\ell,\fc}H^{\bT}+\Gamma)^{-1}\\
	&-\frac{1}{2}(K_{n}^{\ell,\fc_2}-K_{n}^{\ell,\fc_1})H(\bar{C}_n^{\ell,\fc}-C_{n}^{\ell,\fc_2})H^{\bT}(H\bar{C}_{n}^{\ell,\fc}H^{\bT}+\Gamma)^{-1}.
	\end{split}
	\end{equation*}
        By the positive definiteness of $\Gamma$,
        Lemmas~\ref{lemma:MFGainCont}, ~\ref{lemma:MFcovErr} and~\ref{lemma:distMFens}, and corollaries~\ref{covErr} and~\ref{distEns},
	\begin{equation*}
	\begin{split}
	  &\norm{K_n^{\ell, \ff}-\frac{K_{n}^{\ell,\fc_1}+K_{n}^{\ell,\fc_2}}{2}-(\bar{K}_n^{\ell,\ff}-\bar{K}_n^{\ell,\fc})}_p
          \lesssim \norm{C_n^{\ell,\ff}-\frac{C_{n}^{\ell,\fc_1}+C_{n}^{\ell,\fc_2}}{2}-\bar{C}_{n}^{\ell,\ff}+\bar{C}_{n}^{\ell,\fc}}_p\\
	  &+\norm{\bar{C}_n^{\ell,\ff}-C_n^{\ell,\ff}}_p\abs{\bar{C}_n^{\ell,\fc}-\bar{C}_n^{\ell,\ff}}_2
          +\norm{K_{n}^{\ell,\fc_1}-K_n^{\ell,\ff}}_p\norm{\bar{C}_n^{\ell,\fc}-\frac{C_{n}^{\ell,\fc_1}+C_{n}^{\ell,\fc_2}}{2}}_p\\
	&+\frac{1}{2}\norm{K_{n}^{\ell,\fc_2}-K_{n}^{\ell,\fc_1}}_p\norm{\bar{C}_n^{\ell,\fc}-C_{n}^{\ell,\fc_2}}_p\\
	  &\lesssim \norm{C_n^{\ell,\ff}-\frac{C_{n}^{\ell,\fc_1}+C_{n}^{\ell,\fc_2}}{2}-\bar{C}_{n}^{\ell,\ff}+\bar{C}_{n}^{\ell,\fc}}_p
          +P_{\ell}^{-1/2}N_{\ell}^{-\beta/2}+P_{\ell}^{-1}
	\end{split}
	\end{equation*}
	where the last inequality follows from 
	\begin{equation*}
	\begin{split}
	\norm{K_{n}^{\ell,\fc_1}-K_n^{\ell,\ff}}_p &\lesssim \norm{C_{n}^{\ell,\fc_1}-C_n^{\ell,\ff}}_p \\
	&\leq \norm{C_{n}^{\ell,\fc_1}-\bar{C}_n^{\ell,\fc}}_p + \abs{\bar{C}_n^{\ell,\fc}-\bar{C}_n^{\ell,\ff}}_2+\norm{\bar{C}_n^{\ell,\ff}-C_n^{\ell,\ff}}_p
	\end{split}
	\end{equation*}
        and
	\begin{equation}\label{eq:kalmanGainCoarseDiff}
	\norm{K_{n}^{\ell,\fc_1}-K_{n}^{\ell,\fc_2}}_p \lesssim \norm{C_{n}^{\ell,\fc_1}-C_{n}^{\ell,\fc_2}}_p \leq \norm{C_{n}^{\ell,\fc_1}-\bar{C}_n^{\ell,\fc}}_p + \norm{\bar{C}_n^{\ell,\fc}-C_{n}^{\ell,\fc_2}}_p.
	\end{equation}

\end{proof}

We next show how to bound the first term on the right hand side of~\eqref{dbKG}.
\begin{lemma}[Continuity of covariance matrix double differences]\label{lemma:dbCovErr}
For any $n\ge0$ and $p\ge 2$, the following asymptotic inequality holds:
\begin{equation*}
\begin{split}
  \Big\|C_n^{\ell,\ff}-\frac{C_{n}^{\ell,\mathbf{c_1}}+C_{n}^{\ell,\mathbf{c_2}}}{2}-\bar{C}_{n}^{\ell,\ff}+\bar{C}_{n}^{\ell,\fc}\Big\|_p
  &\lesssim \Big\|E_{P_\ell}\Big[\prt{v_{n}^{\ell,\ff_j}-v_{n}^{\ell,\fc_j}-\bar{v}_{n}^{\ell,\ff_j}+\bar{v}_{n}^{\ell,\fc_j}}\prt{ \bar{v}_{n}^{\ell,\ff_j}}^\bT\Big]\Big\|_{p}\\
&+\norm{E_{P_\ell}[v_{n}^{\ell,\ff_j}-v_{n}^{\ell,\fc_j}-\bar{v}_{n}^{\ell,\ff_j}+\bar{v}_{n}^{\ell,\fc_j}]}_{2p}\\
&+P_{\ell}^{-1/2}N_{\ell}^{-\beta/2}+P_{\ell}^{-1}.
\end{split}
\end{equation*}
\end{lemma}
\begin{proof}
Let us first recall that
\begin{equation*}
  \begin{split}
    C_n^{\ell,\ff} = \Cov [v^{\ell,\ff}_{n}],& \quad \bar{C}_n^{\ell,\ff} = \mathrm{Cov} [\bar{v}^{\ell,\ff}_{n}],\\
    C_{n}^{\ell,\fc_j}=\Cov [v^{\ell,\fc_j}_{n}], & \quad \bar{C}_{n}^{\ell,\fc}=\mathrm{Cov}[\bar{v}^{\ell,\fc}_{n}],
  \end{split}
\end{equation*}
and introduce the following covariance matrices for the auxiliary mean-field MLEnKF ensemble
\begin{equation*}
  \cBarTilde_n^{\ell,\ff} := \Cov  [\bar{v}^{\ell,\ff}_{n}], \quad \cBarTilde_{n}^{\ell,\fc_j} := \Cov [\bar{v}^{\ell,\fc_j}_{n}].
\end{equation*}
By the triangle inequality,
\begin{equation}\label{dbCov}
  \begin{split}
    \norm{C_n^{\ell,\ff}-\frac{C_{n}^{\ell,\fc_1}+C_{n}^{\ell,\fc_2}}{2}-\bar{C}_{n}^{\ell,\ff}+\bar{C}_{n}^{\ell,\fc}}_p &\lesssim  \norm{C_n^{\ell,\ff}-\frac{C_{n}^{\ell,\fc_1}+C_{n}^{\ell,\fc_2}}{2}-\cBarTilde_n^{\ell,\ff}+\frac{\cBarTilde_{n}^{\ell,\mathbf{c_1}}+\cBarTilde_{n}^{\ell,\mathbf{c_2}}}{2}}_p\\
    &+\norm{\cBarTilde_n^{\ell,\ff}-\frac{\cBarTilde_{n}^{\ell,\fc_1}+\cBarTilde_{n}^{\ell,\fc_2}}{2}-\bar{C}_{n}^{\ell,\ff}+\frac{\bar{C}_{n}^{\ell,\fc}+\bar{C}_{n}^{\ell,\fc}}{2}}_p.
  \end{split}
\end{equation}
Using that
\begin{equation*}\label{SampleAvNot}
\begin{split}
  E_{P_{\ell-1}}\Big[\sum_{j=1}^2 \frac{v_{n}^{\ell,\fc_j}(v_{n}^{\ell,\fc_j})^{\bT}}{2} \Big] &=  E_{P_\ell}[v_n^{\ell,\fc}(v_n^{\ell,\fc})^{\bT}]\\
  \text{and} \quad  E_{P_{\ell-1}}\Big[\sum_{j=1}^2 \frac{\vBar_{n}^{\ell,\fc_j}(\vBar_{n}^{\ell,\fc_j})^{\bT}}{2} \Big] &=  E_{P_\ell}[\vBar_n^{\ell,\fc}(\vBar_n^{\ell,\fc})^{\bT}],
\end{split}
\end{equation*}
we obtain
\begin{equation*}
  \begin{split}
    \Big\|C_n^{\ell,\ff}&-\frac{C_{n}^{\ell,\fc_1}+C_{n}^{\ell,\fc_2}}{2}-\cBarTilde_n^{\ell,\ff}+\frac{\cBarTilde_{n}^{\ell,\mathbf{c_1}}+\cBarTilde_{n}^{\ell,\mathbf{c_2}}}{2}\Big\|_{p} \\
    &\leq \norm{E_{P_\ell}[v_n^{\ell,\ff}(v_n^{\ell,\ff})^{\bT}-\bar{v}_n^{\ell,\ff}(\bar{v}_n^{\ell,\ff})^{\bT}]-E_{P_\ell}[v_n^{\ell,\fc}(v_n^{\ell,\fc})^{\bT}-\bar{v}_n^{\ell,\fc}(\bar{v}_n^{\ell,\fc})^{\bT}]}_p\\
 &+ \Big\| \E_{P_\ell}[v_n^{\ell,\ff}]E_{P_\ell}[(v_n^{\ell,\ff})]^{\bT}-E_{P_\ell}[\bar{v}_n^{\ell,\ff}]E_{P_\ell}[(\bar{v}_n^{\ell,\ff})]^{\bT} \\
 &\quad -\frac{1}{2}\Big(E_{P_{\ell-1}}[v_{n}^{\ell,\mathbf{c_1}}]E_{P_{\ell-1}}[v_{n}^{\ell,\mathbf{c_1}}]^{\bT}+E_{P_{\ell-1}}[v_{n}^{\ell,\mathbf{c_2}}]E_{P_{\ell-1}}[v_{n}^{\ell,\mathbf{c_2}}]^{\bT}\\
 &\quad -E_{P_{\ell-1}}[\bar{v}_{n}^{\ell,\mathbf{c_1}}]E_{P_{\ell-1}}[\bar{v}_{n}^{\ell,\mathbf{c_1}}]^{\bT}
 \quad -E_{P_{\ell-1}}[\bar{v}_{n}^{\ell,\mathbf{c_2}}]E_{P_{\ell-1}}[\bar{v}_{n}^{\ell,\mathbf{c_2}}]^{\bT}\Big)\Big\|_p\\
 &=:\cI_{11}+\cI_{12}.
   \end{split}
\end{equation*}
For the first term, Lemma~\ref{lemma:distMFens} and Corollary~\ref{distEns} yield 
\begin{equation*}
  \begin{split}
\cI_{11} &\lesssim \norm{E_{P_\ell}[\prt{v_{n}^{\ell,\ff}-v_{n}^{\ell,\fc}-\bar{v}_{n}^{\ell,\ff}+\bar{v}_{n}^{\ell,\fc}}\prt{ \bar{v}_{n}^{\ell,\ff}}^\bT]}_{p}+\norm{v_{n}^{\ell,\ff}-v_{n}^{\ell,\fc}-\bar{v}_{n}^{\ell,\ff}+\bar{v}_{n}^{\ell,\fc}}_{2p}\norm{v_n^{\ell,\ff}-\bar{v}_n^{\ell,\ff}}_{2p}\\
&\quad +\norm{v_n^{\ell,\ff}-v_n^{\ell,\fc}}_{2p}\norm{v_n^{\ell,\fc}-\bar{v}_n^{\ell,\fc}}_{2p}+\norm{v_n^{\ell,\ff}-\bar{v}_n^{\ell,\ff}}_{2p}\norm{\bar{v}_n^{\ell,\ff}-\bar{v}_n^{\ell,\fc}}_{2p}\\
&\lesssim \norm{E_{P_\ell}[\prt{v_{n}^{\ell,\ff}-v_{n}^{\ell,\fc}-\bar{v}_{n}^{\ell,\ff}+\bar{v}_{n}^{\ell,\fc}}\prt{ \bar{v}_{n}^{\ell,\ff}}^\bT]}_{p}+P_{\ell}^{-1/2}N_{\ell}^{-\beta/2}+P_{\ell}^{-1}.
  \end{split}
\end{equation*}
For the second term, the identity $aa^\bT + bb^\bT= \frac{1}{2}\big[(a+b)(a+b)^\bT+(a-b)(a-b)^\bT\big]$
yields
\begin{equation*}
  \begin{split}
\cI_{12}&\le
 \big\|E_{P_\ell}[v_n^{\ell,\ff}]E_{P_\ell}[(v_n^{\ell,\ff})^{\bT}]-E_{P_\ell}[\bar{v}_n^{\ell,\ff}]E_{P_\ell}[(\bar{v}_n^{\ell,\ff})^{\bT}] \\
 &-E_{P_\ell}[v_n^{\ell,\fc}]E_{P_\ell}[(v_n^{\ell,\fc})]^{\bT}+E_{P_\ell}[\bar{v}_n^{\ell,\fc}]E_{P_\ell}[(\bar{v}_n^{\ell,\fc})]^{\bT} \big\|_p \\
 &+ \big\|\frac{1}{4}E_{P_{\ell-1}}[v_{n}^{\ell,\mathbf{c_1}}-v_{n}^{\ell,\mathbf{c_2}}]E_{P_{\ell-1}}[v_{n}^{\ell,\mathbf{c_1}}-v_{n}^{\ell,\mathbf{c_2}}]^\bT\\
 &+\frac{1}{4}E_{P_{\ell-1}}[\bar{v}_{n}^{\ell,\mathbf{c_1}}-\bar{v}_{n}^{\ell,\mathbf{c_2}}]E_{P_{\ell-1}}[\bar{v}_{n}^{\ell,\mathbf{c_1}}-\bar{v}_{n}^{\ell,\mathbf{c_2}}]^\bT\big\|_p =:\cI_{121} + \cI_{122}
   \end{split}
\end{equation*}
The term $\cI_{121}$ can be bounded in a similar fashion as $\cI_{11}$,
and to bound the second term, we employ the identity $a a^\bT - b b^\bT = \frac{1}{2}[(a+b)(a-b)^\bT+(a-b)(a+b)^\bT]$,
Jensen's inequality and Corollary~\ref{distEns}: 
\[
\begin{split}
  \cI_{122} &\lesssim \norm{E_{P_{\ell-1}}[v_{n}^{\ell,\mathbf{c_1}}-v_{n}^{\ell,\mathbf{c_2}}+\bar{v}_{n}^{\ell,\mathbf{c_1}}-\bar{v}_{n}^{\ell,\mathbf{c_2}}]}_{2p}\norm{E_{P_{\ell-1}}[v_{n}^{\ell,\mathbf{c_1}}-v_{n}^{\ell,\mathbf{c_2}}-\bar{v}_{n}^{\ell,\mathbf{c_1}}+\bar{v}_{n}^{\ell,\mathbf{c_2}}]}_{2p}\\
  & \lesssim \norm{v_{n}^{\ell,\mathbf{c}} -\bar{v}_{n}^{\ell,\mathbf{c}} }_{2p}^2 \lesssim P_\ell^{-1}.
\end{split}
\]
Consequently,
\begin{equation*}\label{W12}
\begin{split}
\cI_{12} &\lesssim  \norm{E_{P_\ell}[v_{n}^{\ell,\ff}-v_{n}^{\ell,\fc}-\bar{v}_{n}^{\ell,\ff}+\bar{v}_{n}^{\ell,\fc}]}_{2p}+P_{\ell}^{-1/2}N_{\ell}^{-\beta/2}+P_{\ell}^{-1}.
\end{split}
\end{equation*}
 
For the second term in~\eqref{dbCov}, the equation
$\E_{P_\ell}[\vBar_n^{\ell,\fc}] = \E_{P_{\ell-1}}[(\vBar_n^{\ell,\fc_1} + \vBar_n^{\ell,\fc_2})/2]$
implies that
\begin{equation*}
  \begin{split}
    \Big\|\cBarTilde_n^{\ell,\ff}-&\frac{\cBarTilde_{n}^{\ell,\mathbf{c_1}}+\cBarTilde_{n}^{\ell,\mathbf{c_2}}}{2}-\bar{C}_{n}^{\ell,\ff}+\bar{C}_{n}^{\ell,\fc}\Big\|_p \\
 &\leq \norm{E_{P_\ell}\big[\bar{v}_n^{\ell,\ff}(\bar{v}_n^{\ell,\ff})^{\bT}-\bar{v}_{n}^{\ell,\fc}(\bar{v}_{n}^{\ell,\fc})^{\bT}-\Ex{\bar{v}_n^{\ell,\ff}(\bar{v}_n^{\ell,\ff})^{\bT}-\bar{v}_n^{\ell,\fc}(\bar{v}_n^{\ell,\fc})^{\bT}}\big]}_p\\
&    +\Big\|E_{P_{\ell}}[\bar{v}_{n}^{\ell,\ff}]E_{P_{\ell}}[(\bar{v}_{n}^{\ell,\ff})^\bT]-\Ex{\bar{v}_n^{\ell,\ff}}\Ex{(\bar{v}_n^{\ell,\ff})^{\bT}}\\
& \quad  -E_{P_{\ell}}[\bar{v}_{n}^{\ell,\fc}]E_{P_{\ell}}[(\bar{v}_{n}^{\ell,\fc})^{\bT}]+\Ex{\bar{v}_{n}^{\ell,\fc}}\Ex{(\bar{v}_{n}^{\ell,\fc})^{\bT}}\Big\|_p\\
& + \frac{1}{4}\norm{E_{P_{\ell-1}}[\bar{v}_{n}^{\ell,\fc_1}-\bar{v}_{n}^{\ell,\fc_2}]E_{P_{\ell-1}}[(\bar{v}_{n}^{\ell,\fc_1}-\bar{v}_{n}^{\ell,\fc_2})^\bT]}_p\\
 &=:\cI_{21}+\cI_{22} + \cI_{23}.
  \end{split}
\end{equation*}
H{\"o}lder's inequality, Lemma~\ref{lemma:distMFens} and the M-Z inequality 
imply that 
\[
\cI_{21} \lesssim P_{\ell}^{-1/2} N_{\ell}^{-\beta/2}, \quad \cI_{22} \lesssim P_{\ell}^{-1/2}N_{\ell}^{-\beta/2}+P_{\ell}^{-1}, \quad \text{and} \quad 
\cI_{23} \lesssim P_\ell^{-1}
\]
(where $\Ex{\bar{v}_{n}^{\ell,\fc_1}-\bar{v}_{n}^{\ell,\fc_2}} =0$ was used in the last inequality). 

\end{proof}

\begin{lemma} \label{lemma:phiDoubleUpdate}
For any $n\ge 0$, $p\ge 2$, denoting Jacobian of $\varphi$ by $D\varphi$, it holds that
\begin{equation*}
  \begin{split}
& \Big\|E_{P_{\ell}}\big[ D\varphi( \vBarHat_{n}^{\ell,\fc})(\vHat_{n}^{\ell,\ff}-\vHat_{n}^{\ell,\fc}-\vBarHat_{n}^{\ell,\ff}+\vBarHat_{n}^{\ell,\fc})\Big\|_p\\    
& \lesssim \Big\|E_{P_{\ell}}\Big[ D\varphi(\vBarHat_{n}^{\ell,\fc}) (I-K_n^{\ell,\ff} H) ( v_{n}^{\ell,\ff}-v_{n}^{\ell,\fc}-\bar{v}_{n}^{\ell,\ff}+\bar{v}_{n}^{\ell,\fc}) \Big]\Big\|_{p}\\
    &+\norm{E_{P_\ell}[\prt{v_{n}^{\ell,\ff}-v_{n}^{\ell,\fc}-\bar{v}_{n}^{\ell,\ff}+\bar{v}_{n}^{\ell,\fc}}\prt{ \bar{v}_{n}^{\ell,\ff}}^\bT]}_{2p}
    +\norm{E_{P_\ell}[v_{n}^{\ell,\ff}-v_{n}^{\ell,\fc}-\bar{v}_{n}^{\ell,\ff}+\bar{v}_{n}^{\ell,\fc}]}_{2p}\\
&+P_{\ell}^{-1/2}N_{\ell}^{-\beta/2}+P_{\ell}^{-1}.
\end{split}
\end{equation*}
\end{lemma}
\begin{proof}
  Let $\tilde{y}_n^{\ell,j}$ denote the perturbed observation associated with $(v_{n}^{\ell,\ff_j},v_{n}^{\ell,\fc_j})$.
  By the update equations~\eqref{ml:update} and~\eqref{eq:mfmlenkfDynamics},
  and $(\vBar_{n}^{\ell,\ff_j},\vBar_{n}^{\ell,\fc_j})$, we obtain the representation
\begin{equation}\label{dbUpdate}
    \begin{split}
      \vHat_{n}^{\ell,\ff_j}-\vHat_{n}^{\ell,\fc_j}&-\vBarHat_{n}^{\ell,\ff_j}+\vBarHat_{n}^{\ell,\fc_j}
        = \big(I-K_n^{\ell,\ff} H\big)( v_{n}^{\ell,\ff_j}-v_{n}^{\ell,\fc_j}-\bar{v}_{n}^{\ell,\ff_j}+\bar{v}_{n}^{\ell,\fc_j})\\
        &+(\bar{K}_n^{\ell,\ff}-K_n^{\ell,\ff})H( \bar{v}_{n}^{\ell,\ff_j}-\bar{v}_{n}^{\ell,\fc_j})-(\bar{K}_n^{\ell,\ff}-\bar{K}_{n}^{\ell,\fc})H (v_{n}^{\ell,\fc_j}-\bar{v}_{n}^{\ell,\fc_j})\\
        &- (K_n^{\ell,\ff}-K_{n}^{\ell,\fc_j}-\bar{K}_n^{\ell,\ff}+\bar{K}_{n}^{\ell,\fc})H( v_{n}^{\ell,\fc_j}-\bar{v}_{n}^{\ell,\fc_j})\\
        &- (K_n^{\ell,\ff}-K_{n}^{\ell,\fc_j}-\bar{K}_n^{\ell,\ff}+\bar{K}_{n}^{\ell,\fc})(H\bar{v}_{n}^{\ell,\fc_j} - \tilde{y}_{n}^{\ell,j}),
    \end{split}
\end{equation}
By~\eqref{dbUpdate} and H{\"o}lder's inequality
\begin{equation}\label{dbEnswithVarphi}
  \begin{split}
    & \Big\|E_{P_{\ell}}\big[ D\varphi( \vBarHat_{n}^{\ell,\fc})(\vHat_{n}^{\ell,\ff}-\vHat_{n}^{\ell,\fc}-\vBarHat_{n}^{\ell,\ff}+\vBarHat_{n}^{\ell,\fc})\Big\|_p  \\
      &=
 \Big\|E_{P_{\ell-1}}\big[\sum_{j=1}^{2} D\varphi( \vBarHat_{n}^{\ell,\fc_j})(\vHat_{n}^{\ell,\ff_j}-\vHat_{n}^{\ell,\fc_j}-\vBarHat_{n}^{\ell,\ff_j}+\vBarHat_{n}^{\ell,\fc_j})\Big\|_p\\
 & \lesssim \Big\|E_{P_{\ell-1}}\Big[\sum_{j=1}^{2}  D \varphi(\vBarHat_{n}^{\ell,\fc_j})(I-K_n^{\ell,\ff} H)( v_{n}^{\ell,\ff_j}-v_{n}^{\ell,\fc_j}-\bar{v}_{n}^{\ell,\ff_j}+\bar{v}_{n}^{\ell,\fc_j}) \Big]\Big\|_{p}\\
 &+\norm{D \varphi\prt{\vBarHat_{n}^{\ell,\fc_j}}}_{2p}\Big\{\norm{\bar{K}_n^{\ell,\ff}-K_n^{\ell,\ff}}_{4p}\norm{\bar{v}_{n}^{\ell,\ff_j}-\bar{v}_{n}^{\ell,\fc_j}}_{4p}+\norm{\bar{K}_n^{\ell,\ff}-\bar{K}_{n}^{\ell,\fc}}_{4p}\norm{v_{n}^{\ell,\fc_j}-\bar{v}_{n}^{\ell,\fc_j}}_{4p}\\
 & +\prt{\norm{K_n^{\ell,\ff}-\bar{K}_n^{\ell,\ff}}_{4p}+\norm{K_{n}^{\ell,\fc_j}-\bar{K}_{n}^{\ell,\fc}}_{4p}}\norm{v_{n}^{\ell,\fc_j}-\bar{v}_{n}^{\ell,\fc_j}}_{4p} \Big\} \\
    &+\Big\|E_{P_{\ell-1}}\Big[\sum_{j=1}^{2}D \varphi\prt{\vBarHat_{n}^{\ell,\fc_j}} (K_n^{\ell,\ff}-K_{n}^{\ell,\fc_j}-\bar{K}_n^{\ell,\ff}+\bar{K}_{n}^{\ell,\fc})\prt{H\bar{v}_{n}^{\ell,\fc_j}-\tilde{y}_{n}^{\ell,j}}\Big]\Big\|_p\\
    &=: \cJ_{1,\ell} + \cJ_{2,\ell} + \cJ_{3,\ell}.
     \end{split}
\end{equation}
The properties $\varphi \in \bF \subset C^2_{P}(\bR^d, \bR)$ and
$\vBarHat_n^{\ell,\fc} \in \cap_{r\ge 2} L^r(\Omega,\bR^d)$ imply that $\|D \varphi(\vBarHat_{n}^{\ell,\fc_j})\|_{2p} < \infty$,
and Lemmas~\ref{lemma:MFGainCont},~\ref{lemma:MFcovErr} and~\ref{lemma:distMFens}, and Corollaries~\ref{covErr} and~\ref{distEns} yield
that
\[
\cJ_{2,\ell} \lesssim P_{\ell}^{-1/2}N_{\ell}^{-\beta/2}+P_{\ell}^{-1}.
\]
For the last term, we use $K_{n}^{\ell,\fc_1} = (K_{n}^{\ell,\fc_1}+ K_{n}^{\ell,\fc_2})/2 + (K_{n}^{\ell,\fc_1} -  K_{n}^{\ell,\fc_2})/2$
and the Frobenius scalar product $\langle \cdot, \cdot \rangle_F$ on $\bR^{d \times \dimObs}\times \bR^{d \times \dimObs}$ 
to obtain
\begin{equation*}\label{KGwithObs}
    \begin{split}
      \cJ_{3,\ell} &\leq \norm{D \varphi\prt{\vBarHat_{n}^{\ell,\fc_j}} (K_n^{\ell,\ff}-\frac{K_{n}^{\ell,\mathbf{c_1}}+K_{n}^{\ell,\mathbf{c_2}}}{2}-\bar{K}_n^{\ell,\ff}+\bar{K}_{n}^{\ell,\fc})\prt{H\bar{v}_{n}^{\ell,\fc_j}-\tilde{y}_{n}^{\ell,j}} }_p\\
       &+\norm{\Frob{\frac{K_{n}^{\ell,\mathbf{c_2}}-K_{n}^{\ell,\mathbf{c_1}}}{2}\, , \, E_{P_{\ell-1}}\bigg[ D \varphi\prt{\vBarHat_{n}^{\ell,\mathbf{c_1}}}^\bT \prt{H\bar{v}_{n}^{\ell,\mathbf{c_1}}-\tilde{y}_{n}^{\ell,1}}^\bT-D \varphi\prt{\vBarHat_{n}^{\ell,\mathbf{c_2}}}^\bT \prt{H\bar{v}_{n}^{\ell,\mathbf{c_2}}-\tilde{y}_{n}^{\ell,2}}^\bT\bigg]} }_p\\
       &\lesssim \norm{K_n^{\ell,\ff}-\frac{K_{n}^{\ell,\mathbf{c_1}}+K_{n}^{\ell,\mathbf{c_2}}}{2}-\bar{K}_n^{\ell,\ff}+\bar{K}_{n}^{\ell,\fc}}_{2p}+\norm{\frac{K_{n}^{\ell,\mathbf{c_2}}-K_{n}^{\ell,\mathbf{c_1}}}{2}}_{2p}P_\ell^{-1/2}\\
      &\lesssim \norm{E_{P_\ell}[\prt{v_{n}^{\ell,\ff}-v_{n}^{\ell,\fc}-\bar{v}_{n}^{\ell,\ff}+\bar{v}_{n}^{\ell,\fc}}\prt{ \bar{v}_{n}^{\ell,\ff}}^\bT]}_{2p}+\norm{E_{P_\ell}[v_{n}^{\ell,\ff}-v_{n}^{\ell,\fc}-\bar{v}_{n}^{\ell,\ff}+\bar{v}_{n}^{\ell,\fc}]}_{2p}\\
      &+P_{\ell}^{-1/2}N_{\ell}^{-\beta/2}+P_{\ell}^{-1}.
    \end{split}
\end{equation*}
Here, the second last inequality follows from the M-Z inequality applied to the right argument in the scalar product (as it is a sample average of $P_{\ell-1}$ iid, mean-zero random matrices). The last inequality follows by Lemmas~\ref{DoubleGainCont} and~\ref{lemma:dbCovErr} and~\eqref{eq:kalmanGainCoarseDiff}.
\end{proof}

Lemma~\ref{lemma:dbDiffphi} shows that Theorem~\ref{thm:mlenkfConv} is achieved
through bounding
$\|E_{P_{\ell}}[\varphi(\vHat_{n}^{\ell,\ff})-\varphi(\vHat_{n}^{\ell,\fc})
  -\varphi(\vBarHat_{n}^{\ell,\ff})+\varphi(\vBarHat_{n}^{\ell,\fc})]\|_p$ from above.
To obtain this bound, we introduce the 
following sequence of random $d\times d$ matrices:
for $\ell\ge0$, $r,s \in \{1, \ldots, n\}$ and $j \in \{1,2\}$,
\begin{equation*}
  A_{r,s}^{\ell} =
  \begin{cases}
    \prod_{i=r}^s D\Psi^{N_\ell}_i (\vBarHat_{i}^{\ell,\fc}) (I-K_{i}^{\ell,\ff}H) & \text{if} \quad r\le s\\
    I & \text{if} \quad r>s.
  \end{cases} 
\end{equation*}
Moreover, for $j = 1,2$,
\begin{equation*}
  A_{r,s}^{\ell,j} =
  \begin{cases}
    \prod_{i=r}^s D\Psi^{N_\ell}_i (\vBarHat_{i}^{\ell,\fc_j}) (I-K_{i}^{\ell,\ff}H) & \text{if} \quad r \le s\\
    I & \text{if} \quad r>s.
  \end{cases} 
\end{equation*}
We note that $|A_{r,s}^{\ell,j}| \in \cap_{q\ge2} L^q(\Omega)$ for all index values. 
\begin{corollary} \label{corol:UpdPred}
For any $n\ge2$, $k\le n-1$ and $p\ge2$ it holds that
\begin{equation*}
\begin{split}
 &\big\| E_{P_{\ell}}\big[D \varphi( \vBarHat_{n}^{\ell,\fc})(I-K_{n}^{\ell,\ff}H) A^{\ell}_{k+1,n-1}D \Psi^{N_\ell}_k(\vBarHat_{k}^{\ell,\fc})(\vHat_{k}^{\ell,\ff}-\vBarHat_{k}^{\ell,\ff}-\vHat_{k}^{\ell,\fc}+\vBarHat_{k}^{\ell,\fc}) \big]\big\|_{p}\\
 & \lesssim \big\|E_{P_{\ell}}\big[ D \varphi( \vBarHat_{n}^{\ell,\fc}) (I-K_{n}^{\ell,\ff}H)A^{\ell}_{k,n-1} ( v_{k}^{\ell,\ff}-v_{k}^{\ell,\fc}-\bar{v}_{k}^{\ell,\ff}+\bar{v}_{k}^{\ell,\fc}) \big]\big\|_{p}\\
&\quad +\big\|E_{P_\ell}\big[(v_{k}^{\ell,\ff}-v_{k}^{\ell,\fc}-\bar{v}_{k}^{\ell,\ff}+\bar{v}_{k}^{\ell,\fc})( \bar{v}_{k}^{\ell,\ff})^\bT\big]\big\|_{2p}
  +\big\|E_{P_\ell}[ v_{k}^{\ell,\ff}-v_{k}^{\ell,\fc}-\bar{v}_{k}^{\ell,\ff}+\bar{v}_{k}^{\ell,\fc}]\big\|_{2p}\\
  &\quad +P_{\ell}^{-1/2}N_{\ell}^{-\beta/2}+P_{\ell}^{-1}.
     \end{split}
\end{equation*}
\end{corollary}

\begin{proof}[Sketch of proof]
  Proceeding as in the proof of Lemma~\ref{lemma:phiDoubleUpdate}, we obtain three terms that respectively
  are similar to $\cJ_{1,\ell}, \cJ_{2,\ell}$ and $\cJ_{3,\ell}$ in~\eqref{dbEnswithVarphi}, but now with the prefactor $D \varphi(\vBarHat_{n}^{\ell,\fc_j})$
  replaced by $D \varphi( \vBarHat_{n}^{\ell,\fc_j})(I-K_{n}^{\ell,\ff}H) A^{\ell,j}_{k+1,n-1}D \Psi^{N_\ell}_k(\vBarHat_{k}^{\ell,\fc_j})$. The proof is obtained through the equality
  \[
  A_{k+1,n-1}^{\ell,j}D \Psi^{N_\ell}_k(\vBarHat_{k}^{\ell,\fc_j})(I-K_{k}^{\ell,\ff} H) =A_{k,n-1}^{\ell,j},
  \]
  the boundedness of $A_{k,n-1}^{\ell,j}$ and $D \varphi( \vBarHat_{n}^{\ell,\fc_j})(I-K_{n}^{\ell,\ff}H)$,
  and by bounding the terms corresponding to $\cJ_{2,\ell}$ and $\cJ_{3,\ell}$ in this corollary
  similarly as in said lemma.
\end{proof}

\begin{corollary} \label{corol:AUpdPred}
For any $n\ge1$, $k\le s \le n-1$ and $p\ge2$
it holds that
\begin{equation}\label{eq:AUpdPred}
  \begin{split}
    &\big\| E_{P_{\ell}}\big[A_{k+1,s}^{\ell} D\Psi^{N_\ell}_k(\vBarHat_{k}^{\ell,\fc})(\vHat_{k}^{\ell,\ff}-\vBarHat_{k}^{\ell,\ff}-\vHat_{k}^{\ell,\fc}+\vBarHat_{k}^{\ell,\fc}) \big]\big\|_{p}\\
    & \lesssim \big\|E_{P_{\ell}}\big[A_{k,s}^{\ell} ( v_{k}^{\ell,\ff}-v_{k}^{\ell,\fc}-\bar{v}_{k}^{\ell,\ff}+\bar{v}_{k}^{\ell,\fc}) \big]\big\|_{p}
    +\big\|E_{P_\ell}\big[(v_{k}^{\ell,\ff}-v_{k}^{\ell,\fc}-\bar{v}_{k}^{\ell,\ff}+\bar{v}_{k}^{\ell,\fc})(\bar{v}_{k}^{\ell,\ff})^\bT\big] \big\|_{2p}\\
    &\quad +\big\|E_{P_\ell}[v_{k}^{\ell,\ff}-v_{k}^{\ell,\fc}-\bar{v}_{k}^{\ell,\ff}+\bar{v}_{k}^{\ell,\fc}\big]\big\|_{2p}+P_{\ell}^{-1/2}N_{\ell}^{-\beta/2}+P_{\ell}^{-1}.
  \end{split}
\end{equation}	
\end{corollary}
\begin{proof}[Sketch of proof] The statement of this corollary is similar to Corollary~\ref{corol:UpdPred},
  but with the $1\times d$ prefactor $D \varphi(\vBarHat_{n}^{\ell,\fc_j})(I-K_{n}^{\ell,\ff}H)
  A^{\ell,j}_{k+1,n-1}$
  replaced by $A_{k+1,s}^{\ell, j} D\Psi^{N_\ell}_k(\vBarHat_{k}^{\ell,\fc_j})$.
  As the latter factor also is bounded (it is an element of $\cap_{\tilde r\ge 2} L^{r}(\Omega, \bR^{d\times d})$),
  the proof follows by a similar argument as in Corollary~\ref{corol:UpdPred}.
  However, since the prefactor in this case is a $d \times d$ matrix
  rather than a vector, the term in~\eqref{eq:AUpdPred} that corresponds to
  $\cJ_{3,\ell}$ in the proof of Lemma~\ref{lemma:phiDoubleUpdate} becomes
  \begin{equation*}
    \begin{split}
      &\Big\|E_{P_{\ell-1}}\Big[\sum_{j=1}^{2} A_{k+1,s}^{\ell,j}D \Psi^{N_\ell}_k(\vBarHat_{k}^{\ell,\fc_j}) (K_{k}^{\ell,\ff}-K_{k}^{\ell,\fc_j}-\bar{K}_{k}^{\ell,\ff}+\bar{K}_{k}^{\ell,\fc})(H\bar{v}_{k}^{\ell,\fc_j}-\tilde{y}_{k}^{\ell,j})\Big]\Big\|_p\\
      &\leq \sum_{i=1}^{d} \Big\| E_{P_{\ell-1}}\Big[\sum_{j=1}^{2} \big(A_{k+1,s}^{\ell,j}D \Psi^{N_\ell}_k(\vBarHat_{k}^{\ell,\fc_j})\big)_i (K_{k}^{\ell,\ff}-K_{k}^{\ell,\fc_j}-\bar{K}_{k}^{\ell,\ff}+\bar{K}_{k}^{\ell,\fc})(H\bar{v}_{k}^{\ell,\fc_j}-\tilde{y}_{k}^{\ell,j})\Big] \Big\|_p.
\end{split}
  \end{equation*}
  Here $\big(A_{k+1,s}^{\ell,j}D \Psi^{N_\ell}_k(\vBarHat_{k}^{\ell,\fc_j})\big)_i$ denotes the $i$-th row vector of the matrix. Hence, we have $d$ terms to bound
  that are similar to the $\cJ_{3,\ell}$-term in Corollary~\ref{corol:UpdPred}.
\end{proof}

\begin{corollary} \label{corol:UpdPredvTrans}
For any $n\ge1$, $k\le s \le n-1$, $r \le n$ and $p\ge2$
it holds that
\begin{equation}\label{eq:UpdPredvTrans}
  \begin{split}
    &\big\| E_{P_{\ell}}\big[A_{k+1,s}^{\ell} D\Psi^{N_\ell}_k(\vBarHat_{k}^{\ell,\fc})(\vHat_{k}^{\ell,\ff}-\vBarHat_{k}^{\ell,\ff}-\vHat_{k}^{\ell,\fc}+\vBarHat_{k}^{\ell,\fc}) (\vBar_{r}^{\ell,\ff})^\bT \big]\big\|_{p}\\
    &\lesssim \big\|E_{P_{\ell}}\big[A_{k,s}^{\ell} ( v_{k}^{\ell,\ff}-v_{k}^{\ell,\fc}-\bar{v}_{k}^{\ell,\ff}+\bar{v}_{k}^{\ell,\fc})(\vBar_{r}^{\ell,\ff})^\bT \big]\big\|_{p}
    +\big\|E_{P_\ell}[v_{k}^{\ell,\ff}-v_{k}^{\ell,\fc}-\bar{v}_{k}^{\ell,\ff}+\bar{v}_{k}^{\ell,\fc}\big]\big\|_{2p}\\
    &\quad +\big\|E_{P_\ell}\big[(v_{k}^{\ell,\ff}-v_{k}^{\ell,\fc}-\bar{v}_{k}^{\ell,\ff}+\bar{v}_{k}^{\ell,\fc})(\bar{v}_{k}^{\ell,\ff})^\bT\big] \big\|_{2p}
     +P_{\ell}^{-1/2}N_{\ell}^{-\beta/2}+P_{\ell}^{-1}.
  \end{split}
\end{equation}	
\end{corollary}

\begin{proof}[Sketch of proof]
  The statement of this corollary is similar to Corollary~\ref{corol:UpdPred},
  but here with the additional $1\times d$ postfactor $(\vBar_{r}^{\ell,\ff})^\bT$.
  The only technicality this introduces is in bounding the term in~\eqref{eq:UpdPredvTrans}
  that corresponds to $\cJ_{3,\ell}$ in the proof of Lemma~\ref{lemma:phiDoubleUpdate}. That is,
  \begin{equation*}
    \begin{split}
      &\Big\|E_{P_{\ell-1}}\Big[\sum_{j=1}^{2} A_{k+1,s}^{\ell,j}D \Psi^{N_\ell}_k(\vBarHat_{k}^{\ell,\fc_j}) (K_{k}^{\ell,\ff}-K_{k}^{\ell,\fc_j}-\bar{K}_{k}^{\ell,\ff}+\bar{K}_{k}^{\ell,\fc})(H\bar{v}_{k}^{\ell,\fc_j}-\tilde{y}_{k}^{\ell,j})(\vBar_{r}^{\ell,\ff})^\bT\Big]
      \Big\|_p\\
      &\leq \sum_{i=1}^{d} \Big\| E_{P_{\ell-1}}\Big[\sum_{j=1}^{2} A_{k+1,s}^{\ell,j}D \Psi^{N_\ell}_k(\vBarHat_{k}^{\ell,\fc_j}) (K_{k}^{\ell,\ff}-K_{k}^{\ell,\fc_j}-\bar{K}_{k}^{\ell,\ff}+\bar{K}_{k}^{\ell,\fc})(H\bar{v}_{k}^{\ell,\fc_j}-\tilde{y}_{k}^{\ell,j}) (\vBar_{r}^{\ell,\ff_j})^\bT_i\Big] \Big\|_p.
\end{split}
  \end{equation*}
  Here $(\vBar_{r}^{\ell,\ff_j})^\bT_i$ denotes the $i$-th component of the 
  row vector. We thus have $d$ terms to bound which
  are similar to the $\cJ_{3,\ell}$-term in
  Corollary~\ref{corol:UpdPred}.
\end{proof}

\begin{lemma}\label{lemma:Induction}
For any $n\ge1$, $k\le s \le n-1$, $r\le n$ and $p\ge2$ it holds that
\begin{equation}\label{eq:induction1}
\begin{split}
  &\big\| E_{P_{\ell}}\big[D \varphi( \vBarHat_{n}^{\ell,\fc})(I-K_{n}^{\ell,\ff}H) A_{k+1,n-1}^{\ell}( v_{k+1}^{\ell,\ff_j}-v_{k+1}^{\ell,\fc}-\bar{v}_{k+1}^{\ell,\ff}+\bar{v}_{k+1}^{\ell,\fc})  \big]\big\|_{p}\\
  & \lesssim
  \big\| E_{P_{\ell}}\big[ D \varphi( \vBarHat_{n}^{\ell,\fc})(I-K_{n}^{\ell,\ff}H) A_{k,n-1}^{\ell}( v_{k}^{\ell,\ff}-v_{k}^{\ell,\fc}-\bar{v}_{k}^{\ell,\ff}+\bar{v}_{k}^{\ell,\fc})  \big]\big\|_{2p}\\
  &\quad +\big\|E_{P_\ell}\big[(v_{k}^{\ell,\ff}-v_{k}^{\ell,\fc}-\bar{v}_{k}^{\ell,\ff}+\bar{v}_{k}^{\ell,\fc})( \bar{v}_{k}^{\ell,\ff})^\bT\big]\big\|_{2p}+\big\| E_{P_\ell}\big[ v_{k}^{\ell,\ff}-v_{k}^{\ell,\fc}-\bar{v}_{k}^{\ell,\ff}+\bar{v}_{k}^{\ell,\fc} \big] \big\|_{2p} \\
  &\quad  + N_{\ell}^{-\beta/2}P_\ell^{-1/2}+P_\ell^{-1},
\end{split}
\end{equation}
\begin{equation}\label{eq:induction2}
  \begin{split}
    &\big\|E_{P_\ell}\big[ A_{k+1,s}^{\ell}(v_{k+1}^{\ell,\ff}-v_{k+1}^{\ell,\fc}-\bar{v}_{k+1}^{\ell,\ff}+\bar{v}_{k+1}^{\ell,\fc}) \big]\big\|_{p}\\
    &
    \lesssim
    \big\|E_{P_\ell}\big[ A_{k,s}^{\ell}(v_{k}^{\ell,\ff}-v_{k}^{\ell,\fc}-\bar{v}_{k}^{\ell,\ff}+\bar{v}_{k}^{\ell,\fc}) \big]\big\|_{p}
    +\big\|E_{P_\ell}[v_{k}^{\ell,\ff}-v_{k}^{\ell,\fc}-\bar{v}_{k}^{\ell,\ff}+\bar{v}_{k}^{\ell,\fc}\big]\big\|_{2p}\\
    &\quad +\big\|E_{P_\ell}\big[(v_{k}^{\ell,\ff}-v_{k}^{\ell,\fc}-\bar{v}_{k}^{\ell,\ff}+\bar{v}_{k}^{\ell,\fc})(\bar{v}_{k}^{\ell,\ff})^\bT\big] \big\|_{2p}
    +P_{\ell}^{-1/2}N_{\ell}^{-\beta/2}+P_{\ell}^{-1},
\end{split}
\end{equation}
and
\begin{equation}\label{eq:induction3}
  \begin{split}
    &\big\|E_{P_\ell}\big[ A_{k+1,s}^{\ell}(v_{k+1}^{\ell,\ff}-v_{k+1}^{\ell,\fc}-\bar{v}_{k+1}^{\ell,\ff}+\bar{v}_{k+1}^{\ell,\fc})( \bar{v}_{r}^{\ell,\ff})^\bT \big]\big\|_{p}\\
    &
    \lesssim
    \big\|E_{P_\ell}\big[ A_{k,s}^{\ell}(v_{k}^{\ell,\ff}-v_{k}^{\ell,\fc}-\bar{v}_{k}^{\ell,\ff}+\bar{v}_{k}^{\ell,\fc})( \bar{v}_{r}^{\ell,\ff})^\bT \big]\big\|_{p}
    +\big\|E_{P_\ell}[v_{k}^{\ell,\ff}-v_{k}^{\ell,\fc}-\bar{v}_{k}^{\ell,\ff}+\bar{v}_{k}^{\ell,\fc}\big]\big\|_{2p}\\
    &\quad +\big\|E_{P_\ell}\big[(v_{k}^{\ell,\ff}-v_{k}^{\ell,\fc}-\bar{v}_{k}^{\ell,\ff}+\bar{v}_{k}^{\ell,\fc})(\bar{v}_{k}^{\ell,\ff})^\bT\big] \big\|_{2p}
    +P_{\ell}^{-1/2}N_{\ell}^{-\beta/2}+P_{\ell}^{-1}.
\end{split}
\end{equation}
\end{lemma}
\begin{proof}
The mean-value theorem yields the expansion
\begin{equation}\label{eq:meanValueTheorem}
    \begin{split}
        & v_{k+1}^{\ell,\ff}-v_{k+1}^{\ell,\fc}-\bar{v}_{k+1}^{\ell,\ff}+\bar{v}_{k+1}^{\ell,\fc}
      = \Psi^{N_\ell}_k(\vHat_{k}^{\ell,\ff})- \Psi^{N_{\ell-1}}_k(\vHat_{k}^{\ell,\fc})
      - \Psi^{N_\ell}_k(\vBarHat_{k}^{\ell,\ff})+ \Psi^{N_{\ell-1}}_k(\vBarHat_{k}^{\ell,\fc})\\
      &= D \Psi^{N_\ell}_k (\vBarHat_{k}^{\ell,\fc})(\vHat_{k}^{\ell,\ff}-\vBarHat_{k}^{\ell,\ff}-\vHat_{k}^{\ell,\fc}+\vBarHat_{k}^{\ell,\fc})
      +(\vHat_{k}^{\ell,\ff}-\vBarHat_{k}^{\ell,\ff})^\bT D^2 \Psi^{N_\ell}_k (\theta_k)(\vBarHat_{k}^{\ell,\ff}- \vBarHat_{k}^{\ell,\fc})\\
       &+\prt{D \Psi^{N_\ell}_k (\vBarHat_{k}^{\ell,\fc})-D \Psi^{N_{\ell-1}}_k (\vBarHat_{k}^{\ell,\fc})}(\vHat_{k}^{\ell,\ff}-\vBarHat_{k}^{\ell,\ff})
      +\frac{1}{2}(\vHat_{k}^{\ell,\ff}-\vBarHat_{k}^{\ell,\ff})^\bT D^2 \Psi^{N_\ell}_k (\tilde{\theta}_{k}) (\vHat_{k}^{\ell,\ff}-\vBarHat_{k}^{\ell,\ff})\\
      &-\frac{1}{2}(\vHat_{k}^{\ell,\fc}-\vBarHat_{k}^{\ell,\fc})^\bT D^2 \Psi^{N_\ell}_k (\check{\theta}_{k}) (\vHat_{k}^{\ell,\fc}-\vBarHat_{k}^{\ell,\fc}),
    \end{split}
\end{equation}
where $\theta_k \in \mathrm{Conv}(\vBarHat_{k}^{\ell,\ff}, \vBarHat_{k}^{\ell,\fc})$,
$\tilde \theta_k \in \mathrm{Conv}(\vHat_{k}^{\ell,\ff}, \vBarHat_{k}^{\ell,\ff})$,
$\check \theta_k \in \mathrm{Conv}(\vHat_{k}^{\ell,\fc}, \vHat_{k}^{\ell,\fc})$
and $\mathrm{Conv}(x, \check x):= \{ xt+ (1-t)\check x \mid t \in [0,1]\}$ for $x,\check x \in \bR^d$.
The expansion, Assumption~\ref{ass:psi2} and Corollary~\ref{corol:UpdPred}
yield~\eqref{eq:induction1}. And, similarly, Corollaries~\ref{corol:AUpdPred} and~\ref{corol:UpdPredvTrans}
respectively yield~\eqref{eq:induction2} and~\eqref{eq:induction3}.

\end{proof}

\begin{lemma} \label{lemma:dbDiffphi} For any
  $n\ge0$ and $p\ge2$, it holds that 
  \begin{equation*}
	\begin{split}
	& \big\|E_{P_{\ell}}\big[\varphi(\vHat_{n}^{\ell,\ff})-\varphi(\vHat_{n}^{\ell,\fc})
	    -\varphi(\vBarHat_{n}^{\ell,\ff})+\varphi(\vBarHat_{n}^{\ell,\fc})\big]\big\|_p
          \lesssim N_{\ell}^{-\beta/2}P_\ell^{-1/2}+P_\ell^{-1}.
	\end{split}
	\end{equation*}
\end{lemma}
\begin{proof}
By assumption~\ref{ass:psi}(ii) and a similar application of the mean-value theorem as in~\eqref{eq:meanValueTheorem},

\[
\begin{split}
\big\|E_{P_{\ell}}\big[&\varphi(\vHat_{n}^{\ell,\ff})-\varphi(\vHat_{n}^{\ell,\fc}) 
	    -\varphi(\vBarHat_{n}^{\ell,\ff})+\varphi(\vBarHat_{n}^{\ell,\fc})\big]\big\|_p \\
& \lesssim \big\|E_{P_{\ell}}\big[D \varphi( \vBarHat_{n}^{\ell,\fc})(\vHat_{n}^{\ell,\ff}-\vHat_{n}^{\ell,\fc}-\vBarHat_{n}^{\ell,\ff}+\vBarHat_{n}^{\ell,\fc})\big]\big\|_p +N_{\ell}^{-\beta/2}P_\ell^{-1/2}+P_\ell^{-1}.
\end{split}
\]

By Lemma~\ref{lemma:phiDoubleUpdate}, and thereafter using that $A^\ell_{n,n-1}=A^\ell_{n-1,n-2}=I$ and
applying Lemma~\ref{lemma:Induction}, we obtain
\begin{equation*} 
  \begin{split}
    &\cK_\ell:= \big\|E_{P_{\ell}}\big[D \varphi( \vBarHat_{n}^{\ell,\fc})(\vHat_{n}^{\ell,\ff}-\vHat_{n}^{\ell,\fc}-\vBarHat_{n}^{\ell,\ff}+\vBarHat_{n}^{\ell,\fc})\big]\big\|_p\\
    & \lesssim \Big\|E_{P_{\ell}}\Big[ D\varphi(\vBarHat_{n}^{\ell,\fc}) (I-K_n^{\ell,\ff} H)A^\ell_{n,n-1}( v_{n}^{\ell,\ff}-v_{n}^{\ell,\fc}-\bar{v}_{n}^{\ell,\ff}+\bar{v}_{n}^{\ell,\fc}) \Big]\Big\|_{p}\\
    &+\big\| E_{P_\ell}[A^\ell_{n,n-1}(v_{n}^{\ell,\ff}-v_{n}^{\ell,\fc}-\bar{v}_{n}^{\ell,\ff}+\bar{v}_{n}^{\ell,\fc})( \bar{v}_{n}^{\ell,\ff})^\bT]\big\|_{2p}
    +\big\|E_{P_\ell}[A^\ell_{n,n-1}(v_{n}^{\ell,\ff}-v_{n}^{\ell,\fc}-\bar{v}_{n}^{\ell,\ff}+\bar{v}_{n}^{\ell,\fc})]\big\|_{2p}\\
&+P_{\ell}^{-1/2}N_{\ell}^{-\beta/2}+P_{\ell}^{-1}\\
    & \lesssim \big\|E_{P_{\ell}}\big[ D\varphi(\vBarHat_{n}^{\ell,\fc}) (I-K_n^{\ell,\ff} H) A_{n-1,n-1}^{\ell}( v_{n-1}^{\ell,\ff}-v_{n-1}^{\ell,\fc}-\bar{v}_{n-1}^{\ell,\ff}+\bar{v}_{n-1}^{\ell,\fc}) \big]\big\|_{p}\\
    &+ \big\| E_{P_\ell}[A^\ell_{n-1,n-1}(v_{n-1}^{\ell,\ff}-v_{n-1}^{\ell,\fc}-\bar{v}_{n-1}^{\ell,\ff}+\bar{v}_{n-1}^{\ell,\fc})( \bar{v}_{n}^{\ell,\ff})^\bT]\big\|_{2p}\\
    &+ \big\|E_{P_\ell}[A^\ell_{n-1,n-1}(v_{n-1}^{\ell,\ff}-v_{n-1}^{\ell,\fc}-\bar{v}_{n-1}^{\ell,\ff}+\bar{v}_{n-1}^{\ell,\fc})]\big\|_{2p}\\
    &+ \big\|E_{P_\ell}\big[A^\ell_{n-1,n-2}(v_{n-1}^{\ell,\ff}-v_{n-1}^{\ell,\fc}-\bar{v}_{n-1}^{\ell,\ff}+\bar{v}_{n-1}^{\ell,\fc})( \bar{v}_{n-1}^{\ell,\ff})^\bT\big]\big\|_{4p}\\
    &+ \big\| E_{P_\ell}\big[ A^\ell_{n-1,n-2}(v_{n-1}^{\ell,\ff}-v_{n-1}^{\ell,\fc}-\bar{v}_{n-1}^{\ell,\ff}+\bar{v}_{n-1}^{\ell,\fc}) \big] \big\|_{4p}
    +P_{\ell}^{-1/2}N_{\ell}^{-\beta/2}+P_{\ell}^{-1}.
  \end{split}
\end{equation*}
Recalling that $\vHat_{0}^{\ell,\ff}-\vHat_{0}^{\ell,\fc}-\vBarHat_{0}^{\ell,\ff}+\vBarHat_{0}^{\ell,\fc}=0$
and applying Lemma~\ref{lemma:Induction} iteratively $n-1$ times, we obtain that
\begin{equation*} 
  \begin{split}
    \cK_\ell &\lesssim \big\|E_{P_{\ell}}\big[ D\varphi(\vBarHat_{n}^{\ell,\fc}) (I-K_n^{\ell,\ff} H) A_{1,n-1}^{\ell}D \Psi^{N_\ell}_0(\vBarHat_{0}^{\ell,\fc}) ( \vHat_{0}^{\ell,\ff}-\vHat_{0}^{\ell,\fc}-\vBarHat_{0}^{\ell,\ff}+\vBarHat_{0}^{\ell,\fc}) \big]\big\|_{p}\\
    &+ \sum_{k=1}^{n-1}\Big\{ \big\| E_{P_\ell}[A^\ell_{1,n-k} D \Psi^{N_\ell}_0(\vBarHat_{0}^{\ell,\fc}) ( \vHat_{0}^{\ell,\ff}-\vHat_{0}^{\ell,\fc}-\vBarHat_{0}^{\ell,\ff}+\vBarHat_{0}^{\ell,\fc})( \bar{v}_{n}^{\ell,\ff})^\bT]\big\|_{2^kp}\\
    &    +  \big\|E_{P_\ell}[A^\ell_{1,n-k}D \Psi^{N_\ell}_0(\vBarHat_{0}^{\ell,\fc}) ( \vHat_{0}^{\ell,\ff}-\vHat_{0}^{\ell,\fc}-\vBarHat_{0}^{\ell,\ff}+\vBarHat_{0}^{\ell,\fc})]\big\|_{2^kp}\Big\} +N_{\ell}^{-\beta/2}P_\ell^{-1/2}+P_\ell^{-1}\\
    &\lesssim N_{\ell}^{-\beta/2}P_\ell^{-1/2}+P_\ell^{-1}.
  \end{split}
\end{equation*}

\end{proof}

\subsection{Proof of Corollary~\ref{cor:mlenkf}}
\label{sec:mlenkfCost}

By Theorem~\ref{thm:mlenkfConv}, we recall that 
\[
\|\mu_n^{\ML}[\varphi] - \overline \mu_n[\varphi]\|_p \lesssim
N_L^{-\beta/2}P_L^{-1/2} + P_L^{-1} +N_L^{-\alpha}+\sum_{\ell=0}^L M_\ell^{-1/2}(N_\ell^{-\beta/2}P_\ell^{-1/2}+P_\ell^{-1}),
\]
with $P_\ell \eqsim 2^{\ell}$, $N_\ell \eqsim 2^{s \ell}$ for some $s>0$.
Our objective is to prove that the parameter choices for $s$, $L$, $\{M_\ell\}$
stated in Corollary~\ref{cor:mlenkf} ensure that the goal
\begin{equation*}\label{eq:approxAccuracy}
\|\mu_n^{\ML}[\varphi] - \overline \mu_n[\varphi]\|_p \lesssim \epsilon
\end{equation*}
is reached at the asymptotic computational cost~\eqref{eq:mlenkfCosts}, where
\[
\text{Cost(MLEnKF)} := \sum_{\ell=0}^L M_\ell N_\ell P_\ell.
\]

Fix the value of $s>0$. It then
follows straightforwardly that to control the ``bias''
\[
N_L^{-\beta/2}P_L^{-1/2} + P_L^{-1} +N_L^{-\alpha} \lesssim \epsilon,
\]
one must have $L \eqsim \log_2(\epsilon^{-1})/\min(1, (1+\beta s)/2, \alpha s) + 1$.
It remains to minimize $\sum_{\ell=0}^L M_\ell N_\ell P_\ell$
subject to the constraint
\begin{equation}\label{eq:costDef}
\sum_{\ell=0}^L M_\ell^{-1/2}(N_\ell^{-\beta/2}P_\ell^{-1/2}+P_\ell^{-1}) \lesssim \epsilon,
\end{equation}
and also having in mind that $M_\ell$ must be a natural number for all $\ell \le L$.

The method of Lagrange multipliers applied to 
\[
\mathcal{L}(\{M_\ell \}, \lambda) := \sum_{\ell =0}^{L} M_\ell N_\ell P_\ell + \lambda (\sum_{\ell=0}^{L} M_\ell^{-1/2} D_\ell - \epsilon)
\]
with $D_\ell :=(N_\ell^{-\beta/2}P_\ell^{-1/2}+P_\ell^{-1})$ yields
\begin{equation}\label{eq:asymptoticMl}
M_\ell \eqsim \lambda^{2/3} (N_\ell P_\ell)^{-2/3}D_\ell^{2/3} +1 \quad \text{and} \quad
\lambda = \epsilon^{-3} \prt{\sum_{\ell=0}^{L} (N_\ell P_\ell)^{1/3}D_\ell^{2/3}}^3.
\end{equation}
Since
\[
P_\ell N_\ell \eqsim 2^{(1+s)\ell} \text{ and }  D_\ell \eqsim 2^{-(\min({\beta s,1})+1)\ell/2},
\]
we have
\[
\sum_{\ell=0}^{L} (N_\ell P_\ell)^{1/3}D_\ell^{2/3}
\eqsim \sum_{\ell=0}^{L}  2^{(s-\min (\beta s,1))\ell/3}
\eqsim \begin{cases}
  1 & \text{if } \min(\beta s,1)>s, \\
 L & \text{if }  \min(\beta s,1)=s,\\
2^{(s-\min(\beta s, 1))L/3} & \text{if } \min(\beta s,1)<s,
\end{cases}
\]
and the optimal formula for $\{M_\ell\}$ for a fixed $s>0$,
cf.~\eqref{eq:chooseMlr}, follows from the last equality and~\eqref{eq:asymptoticMl}.

By~\eqref{eq:costDef}, the choice for $\{M_\ell\}$ leads to
\begin{equation*}\label{eq:mlenkfCosts2}
\mathrm{Cost(MLEnKF)} \eqsim
\begin{cases}
\epsilon^{-2}+\epsilon^{-f(s)} & \text{if } \min(\beta s,1)>s,\\              
\epsilon^{-2} L^3+\epsilon^{ -f(s)} & \text{if } \min(\beta s,1)=s,\\
\epsilon^{-2\mathlarger{-\frac{s-\min(\beta s, 1)}{\min(1, (\beta s +1)/2, \alpha s)}}}+
\epsilon^{-f(s)} & \text{if } \min(\beta s,1)<s,
\end{cases}
\end{equation*}
where
\begin{equation*}\label{eq:fDef}
f(s):= \frac{1+s}{ \min(1, (\beta s +1)/2, \alpha s)} =
                    \begin{cases}
                      \vspace*{0.1cm}
                      \mathlarger{\frac{1+s}{\alpha s}} &  \text{if } s \le \frac{\min((1+\beta s)/2,1)}{\alpha},\\
                      \mathlarger{\frac{2(1+s)}{\beta s +1}} & \text{if } \frac{(1+\beta s)}{2\alpha}<  s < \alpha^{-1},\\
                      1+s & \text{if } \min((1+\beta s)/2, \alpha s) \ge 1.
                    \end{cases}
\end{equation*}
We next consider the problem of determining the value/inclusion set of $s$ which
minimizes the asymptotic growth rate of $\mathrm{Cost(MLEnKF)}$. We consider three cases separately:

{\bf 1.} If $\beta<1$, then $\min(\beta s,1)<s$ for all $s>0$, cf.~Figure~\ref{fig:AllCond}(a), and
\[
\min(1, (\beta s +1)/2, \alpha s) \leq \frac{1+\min\prt{1, \beta s}}{2},
\]
implies that
\[
2 + \frac{s-\min(\beta s, 1)}{\min(1, (\beta s +1)/2, \alpha s)} \le
f(s).
\]
\begin{figure}[hbt!]
	\makebox[0pt]{\includegraphics[width=1.3\textwidth ]{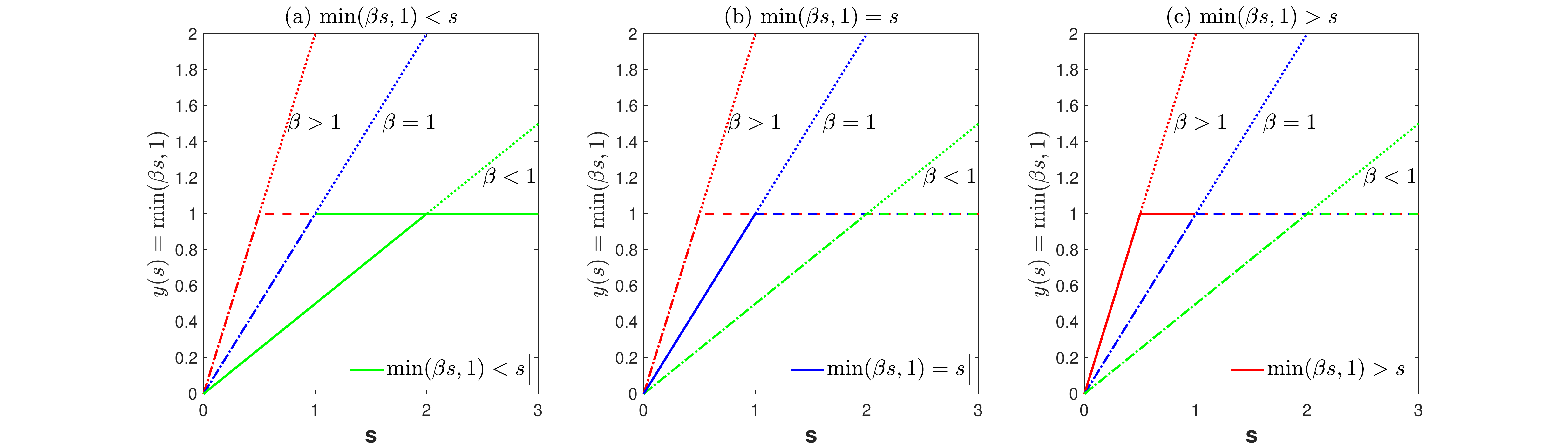}}
	\captionsetup{width=1.0\textwidth, font=footnotesize}
	\caption{ (a) The inequality $\min(\beta s,1)<s$ (green line). (b) The equality $\min(\beta s,1)=s$ (blue line). (c) The inequality $\min(\beta s,1)>s$ (red line).  The dash lines correspond to the function $y(s)=\min(\beta s, 1)$ and the dotted lines refer to the function $y(s)=\beta s$ varying by different cases of $\beta$ value.}
	\label{fig:AllCond}
\end{figure}
Consequently,
\[
\mathrm{Cost(s)} \eqsim \epsilon^{\mathlarger{-\frac{1+s}{ \min(1, (\beta s +1)/2, \alpha s)}}}.
\]
Observing that $f$ is strictly decreasing on the set $(s < \frac{\min((1+\beta s)/2,1)}{\alpha})$
and strictly increasing on $(s > \frac{\min((1+\beta s)/2,1)}{\alpha})$, we obtain the unique minimizer
\begin{equation}\label{eq:fArgMin}
\arg \min_{s >0} f(s) =
\begin{cases}
  \alpha^{-1} & \text{if } \alpha \le \beta,\\
  (2\alpha-\beta)^{-1} & \text{otherwise}.
\end{cases}
\end{equation}

{\bf 2.} If $\beta=1$, then $\min(\beta s,1)=s$ for all $s \in (0,1]$ and $\min(\beta s,1)<s$ for $s>1$,
cf.~Figure~\ref{fig:AllCond}(a-b). Consequently,
\[
\mathrm{Cost(s)} \eqsim
\begin{cases}
  \epsilon^{-2} L^3+\epsilon^{ -f(s)} & \text{if } s \in (0,1],\\
                    \epsilon^{-f(s)} & \text{if } s >1.
\end{cases}
\]
By~\eqref{eq:fArgMin}, now simplifying to $f(s)= \max(1+s,2,(1+s)/(\alpha s) )$,
we obtain that
\[
\mathrm{Cost(s)} \eqsim
\begin{cases}
\epsilon^{-2} L^3 & \text{if } \alpha \ge 1 \quad \& \quad s \in [\alpha^{-1},1]\\
\epsilon^{-(1+\alpha^{-1})} & \text{if } \alpha < 1 \quad \& \quad s = \alpha^{-1}.
\end{cases}
\]                      
                      
{\bf 3.} If $\beta>1$, then $\min(\beta s,1)>s$ for all $s \in (0,1)$ and $\min(\beta s,1)=s$ for $s=1$,
and $\min(\beta s,1)<s$ for $s>1$. 
Thus
\[
\mathrm{Cost(s)} \eqsim
\begin{cases}
  \epsilon^{-2} +\epsilon^{ -f(s)} & \text{if } s\in (0,1),\\
  \epsilon^{-2} L^3+\epsilon^{ -f(s)} & \text{if } s=1,\\
  \epsilon^{-f(s)} & \text{if } s >1.
\end{cases}
\]
If $\alpha >1$, then $f(s) \le 2 \iff s \in [\alpha^{-1},1]$.
If $\alpha =1$, then $f(s) \le 2 \iff s = 1$. And if $\alpha < 1$, then
\[
f(s) = \min_{s >0} f(s) = (1+\alpha^{-1}) \iff s = \alpha^{-1}.
\]
This yields
\[
\mathrm{Cost(s)} \eqsim
\begin{cases}
\epsilon^{-2} & \text{if } \alpha >1 \quad \& \quad s \in [\alpha^{-1},1)\\
  \epsilon^{-2} L^3 & \text{if } \alpha= 1 \quad \& \quad s =1\\
\epsilon^{-(1+\alpha^{-1})} & \text{if } \alpha < 1 \quad \& \quad s = \alpha^{-1}.
\end{cases}
\]         
             
\section{DMFEnKF algorithm}
\label{sec:dmfenkf}
This section describes the algorithm for the density-based
deterministic approximation of the MFEnKF, which iteratively computes
the prediction density $\rho_{\vBar_{n}}$ and updated density
$\rho_{\hat{\vBar}_{n}}$ for $n=1,2, \dots$. Each iteration cycle
consists of two steps: one transition from $\rho_{\hat{\vBar}_{n}}$ to
$\rho_{\vBar_{n+1}}$ governed by the Fokker-Planck equation (FPE) and
another from $\rho_{\vBar_{n+1}}$ to $\rho_{\hat{\vBar}_{n+1}}$ via an
affine transformation and a convolution with a Gaussian density. For
simplicity, we show the algorithm for the one-dimensional state-space
case, i.e, $d=1$.

Let $\mathcal{S}^t \rho$ denote a solution at time $t$ of the FPE
\begin{equation*}
\partial_t p(x,t) = \partial_x(V'(x) p(x,t))+\frac{\sigma^2}{2}\partial^2_x p(x,t), \qquad  (x,t) \in \bR \times (0,\infty),
\end{equation*}
with the initial condition $p(\cdot ,0)=\rho$. Note that for  
the mean-field dynamics~\eqref{mf:prediction} with $N= \infty$
that satisfies the SDE~\eqref{sde:genform}, it holds that $\Psi(\hat{\vBar}_n) \sim \mathcal{S}^{1}\rho_{\hat{\vBar}_n}$ for any $n\ge0$.
 

Since the updated mean-field ensembles can be viewed as the sum of 
the following independent random variables
 \[\hat \vBar_{n}=\underbrace{(I-\bar K_{n} H)\vBar_{n}+\bar K_{n}{y}_{n}}_\text{X}+\underbrace{\bar K_{n}\tilde \eta_{n}}_\text{Y},
 \] 
 the updated density can be written
\[
\rho_{\hat \vBar_{n}}(x)=\rho_{X+Y}(x)= \int_{\bR}\rho_X(z)\rho_Y(x-z)dz=\rho_X\ast \rho_Y (x),
\]
where $\ast$ denotes the convolution operator.

\begin{algorithm}[H]
	\DontPrintSemicolon
	
	\KwInput{The initial updated density $\rho_{\hat \vBar_{0}}=\rho_{u_0|Y_0}$, the number of time steps $N_t$, the number of spatial steps $N_x$, the discretization interval $[x_0, x_1]$, the simulation length $\mathcal{N}$.}
	\KwOutput{The prediction and updated density, $\rho_{\vBar_{n}}$ and $\rho_{\hat{\vBar}_{n}}$, respectively.}
    $\Delta t=\frac{1}{N_t}$, $\Delta x=\frac{x_1-x_0}{N_x}$.
	
	\For{n=1 : $\mathcal{N}$}
	{  
		Compute the prediction density $\rho_{\vBar_{n}}(x) = \mathcal{S}^1 \rho_{\hat \vBar_{n-1}}$ by Crank-Nicolson numerical method with the discretization steps ($\Delta t, \Delta x$).
		
		Compute the prediction covariance $\bar C_n=\int x^2\rho_{\vBar_{n}}(x) dx-(\int x\rho_{\vBar_{n}}(x) dx)^2$ using a quadrature rule.
		
		Compute the Kalman gain $\bar K_{n}=\bar C_n H^{\bT}(H\bar C_n H^{\bT} + \Gamma)^{-1}$.
		
		Compute the updated density $\rho_{\hat \vBar_{n}}=\rho_X\ast \rho_Y $ by 
		discrete convolution of the two functions represented on the spatial mesh.

	}
	\caption{DMFEnKF}
\end{algorithm}
\textit{Remark.} The discretization interval $[x_0, x_1]$ must be
chosen such that the truncation error pertaining to the integral on
the complement of said interval is close to zero, i.e.,
\[
\int_{[x_0,x_1]^c} \rho_{\hat \vBar_{n}}(x) dx \approx 0.
\]
For the problem in Section~\ref{ssec:dw}, we found by numerical
experiments using the Crank-Nicolson method that $[x_0,x_1]=[-5,5]$,
$\Delta t = 10^{-3}$ and $\Delta x = 10^{-5}$ were suitable resolution
parameters to obtain solutions with negligible approximation errors
relative to the statistical and bias error introduced by EnKF and
MLEnKF.

\section{Extension from MLEnKF to multi-index EnKF (MIEnKF)}\label{app:extension}
Following the construction of multi-index Monte Carlo mehods for mean-field
dynamics approximations~\cite{abdo2018}, we sketch an extension from
MLEnKF to MIEnKF. On the particle level, this extension may be viewed
as an extension from two-particle coupling to four-particle coupling.
Let for simplicity $N_{\ell_1}\eqsim 2^{\ell_1}$ and $P_{\ell_2}\eqsim 2^{\ell_2}$ and
introduce the following MIEnKF estimator 
\[
\mu^{MIEnKF}_n[\varphi] := \sum_{ (\ell_1, \ell_2)  \in \cI}
\sum_{m=1}^{M_{\ell_1,\ell_2}} \frac{\Delta_{\ell_1,\ell_2} \mu_n^m[\varphi]}{M_{\ell_1,\ell_2}},
\]
where the index-set $\cI \subset \bN_0^2$ that the estimator is summing over is a function of the accuracy $\epsilon>0$
and $\{M_{\ell_1,\ell_2}\}_{(\ell_1,\ell_2) \in \cI}$
denotes the number of four-coupled iid samples of EnKF estimators on level
$(\ell_1,\ell_2)$:
\[
\begin{split}
  \Delta_{\ell_1,\ell_2} \mu_n^m [\varphi] := &\Bigg(\mu_n^{N_{\ell_1}, P_{\ell_2},m} - \Big(\mu_n^{N_{\ell_1}, P_{\ell_2-1},1,m} + \mu_n^{N_{\ell_1}, P_{\ell_2-1},2,m}\Big)/2\\
  &\qquad - \mu_n^{N_{\ell_1-1}, P_{\ell_2},m} + \Big(\mu_n^{N_{\ell_1-1}, P_{\ell_2-1},1,m} + \mu_n^{N_{\ell_1-1}, P_{\ell_2-1},2,m}\Big)/2 \Bigg)[\varphi].
\end{split}
\]
Tentative numerical tests for both filtering problems considered in Section~\ref{sec:numerics}
yield that
\begin{equation*}
  \begin{split}
    \abs{\E[\mathbf{\Delta}_{(\ell_1, \ell_2)} \mu_n^m[\varphi]] } & \lesssim N_{\ell_1}^{-1} P_{\ell_2}^{-1},\\
      \norm{\mathbf{\Delta}_{(\ell_1, \ell_2)} \mu_n^m[\varphi]]}_p &\lesssim N_{\ell_1}^{-1}P_{\ell_2}^{-1}.
    \end{split}
  \end{equation*}
On the basis of these rates, one may optimally determine the index set
$\cI$ and $\{M_{\ell_1,\ell_2}\}_{(\ell_1,\ell_2) \in \cI}$,
cf.~\cite{abdo2016}. A detailed description and performance study
of the MIEnKF method is left as future work.

\medskip

{\bf Acknowledgements } This work was supported by the KAUST Office of
Sponsored Research (OSR) under Award No. URF/1/2584-01-01 and the
Alexander von Humboldt Foundation. G.Shaimerdenova and R. Tempone are
members of the KAUST SRI Center for Uncertainty Quantification in
Computational Science and Engineering.

\bibliography{manuscript}
\bibliographystyle{plain}
\end{document}